\documentclass[a4paper,12pt]{amsart}
\usepackage{amsfonts,amsthm,amssymb,amsmath,amscd}
\usepackage[colorlinks=true,linkcolor=blue,citecolor=blue]{hyperref}
\usepackage{amsmath}
\usepackage{tikz-cd}

\usepackage{tabularx}
\usepackage{tcolorbox}
\usepackage[english]{babel}
\usepackage{amssymb,amsmath}
\usepackage{xcolor}

\newtheorem{Theo}{Theorem}[section]
\newtheorem{Prop}[Theo]{Proposition}
\newtheorem{Coro}[Theo]{Corollary}
\newtheorem{Lemm}[Theo]{Lemma}

\newtheorem{Rema}[Theo]{Remark}

\newcommand{\C}{\mathbb{C}}

\DeclareMathOperator{\re}{Re}

\newcommand{\supp}{\operatorname{supp}}

\def\N{\mathbb{ N}}
\def\R{\mathbb{ R}}

\begin{document}

\title[Functions of finite order generated by Dirichlet series]{Holomorphic functions of finite order generated by Dirichlet series}

\address[]{Andreas Defant\newline  Institut f\"{u}r Mathematik,\newline Carl von Ossietzky Universit\"at,\newline
26111 Oldenburg, Germany.
}
\email{defant@mathematik.uni-oldenburg.de}

\address[]{Ingo Schoolmann\newline  Institut f\"{u}r Mathematik,\newline Carl von Ossietzky Universit\"at,\newline
26111 Oldenburg, Germany.
}
\email{ingo.schoolmann@uni-oldenburg.de}

\author[Defant and Schoolmann]{Andreas Defant and Ingo Schoolmann}

\begin{abstract}
Given a frequency $\lambda = (\lambda_n)$  and $\ell \ge 0$, we introduce  the scale of Banach spaces $H_{\infty,\ell}^{\lambda}[\re > 0]$ of holomorphic functions $f$ on the open right half-plane
$[\re > 0]$, which satisfy $(A)$ the growth condition $|f(s)| = O((1 + |s|)^\ell)$, and $(B)$  have a Riesz germ, i.e. on some open subset and for some $m \ge 0$ the function $f$ coincides
with the pointwise  limit (as $x \to \infty$)  of the so-called
$(\lambda,m)$-Riesz means
$\sum_{\lambda_n < x} a_n e^{-\lambda_n s}\big( 1-\frac{\lambda_n}{x}\big)^m ,\,x >0$
 of some $\lambda$-Dirichlet series $\sum a_n e^{-\lambda_n s}$. Reformulated in our terminology, an  important result of M. Riesz  shows that in this case the function $f$
for every $k >\ell$ is the pointwise limit of the $(\lambda,k)$-Riesz means of $D$ on $[\re > 0]$.

Our main contribution is an extension --  showing that 'after translation' every bounded set in $H_{\infty,\ell}^{\lambda}[\re > 0]$ is uniformly approximable by all its $(\lambda,k)$-Riesz means of order $k>\ell$.
This follows from an appropriate maximal theorem, which in fact turns out to be at the very heart of a seemingly interesting structure theory of the  Banach spaces $H_{\infty,\ell}^{\lambda}[\re > 0]$. One of the many consequences is that $H_{\infty,\ell}^{\lambda}[\re > 0]$  basically  consists of those holomorphic functions
on $[\re >0]$, which  have a Riesz germ and are of finite uniform order $\ell$ on $[\re >0]$.

To establish all this and more, we need to  reorganize
(and to improve) various aspects and keystones of the classical theory of Riesz summability  of general Dirichlet series as invented by Hardy and M. Riesz.
\end{abstract}
\maketitle


\tableofcontents

\noindent
\renewcommand{\thefootnote}{\fnsymbol{footnote}}
 \footnotetext{}

\section{\bf Introduction}

Recently, the theory of ordinary Dirichlet series $\sum a_n n^{-s} = \sum a_n e^{-s \log n}$ and that of general Dirichlet series
$\sum a_n e^{-\lambda_n s}$, where the frequency $\lambda = (\lambda_n)$ is a strictly increasing, non-negative real sequence,
saw a remarkable renaissance.

\smallskip

General  Dirichlet series $\sum a_n e^{-\lambda_n s}$ converge on maximal half-planes $[\re >~\sigma]$  in $\C$, where they define holomorphic functions $f$, and even though in concrete situations these half-planes might be small, the holomorphic functions $f$ often extend to  holomorphic functions on larger half-planes.

\smallskip

The problem of determining whether a holomorphic $f$, defined on some half-plane, is or is not representable in the form of a
$\lambda$-Dirichlet series  $\sum a_n e^{-\lambda_n s}$, is in general difficult.

\smallskip

An important criterion is that the limit function of a $\lambda$-Dirichlet series
has finite order in each closed half-plane of its domain of convergence. Moreover, if this limit function can be holomorphically extended
to a larger half-plane, it may happen that its continuation still has finite order in a larger domain.

\smallskip

Recall that,  given a  holomorphic function $f: [\re >\sigma_0] \to \C,\, \sigma_0 \in \R$, its  order $\mu_f(\sigma)\in [\sigma_0,\infty]$ on the line
$[\re~=~\sigma],\, \sigma > \sigma_0$ is the infimum over all $\ell \in [0,\infty] $
such that
\begin{equation}\label{finiteorderdefintro}
\exists~ C, t_{0}>0~\forall~|t|>t_{0}\colon~ |f(\sigma+it)|\le C|t|^{\ell}\,.
\end{equation}
The function $f$ is said to be of finite order on  $[\re >\sigma_0]$ whenever $\mu_f < \infty$ for all  $\sigma >\sigma_0$, and in this case  $\mu_f:]\sigma_0,\infty[ \to \mathbb{R}_{\ge 0}$ is continuous and convex.
\smallskip

Similarly,  the uniform  order  $\nu_f \in [0,\infty]$ of $f$ on the half-plane $[\re >\sigma_0]$  is given by the infimum over all $\ell \in [0,\infty] $ such that
the above  condition holds uniformly, i.e.,
\begin{equation}\label{finiteorderdefintroU}
\exists~ C, t_{0}>0~\forall ~\sigma>\sigma_0, |t|>t_{0}\colon~ |f(\sigma+it)|\le C|t|^{\ell}.
\end{equation}

Moreover, the function $f$ is said to have finite uniform order on $[\re >\sigma_0]$ whenever $\nu_f < \infty$.

\smallskip

To see a famous example, denote by $\zeta:  \C \setminus \{1\} \to \C$ the zeta-function, which
is holomorphic with a simple pole in $s=1$, and which
on $[\re > 1]$ is the pointwise limit of the zeta-Dirichlet series $\sum n^{-s}$. The famous Lindel\"of conjecture asks whether the order of $\zeta$
on the critical line $[\re = \frac{1}{2}]$ equals  $0$. Equivalently, consider the entire function
\[
\eta: \C  \to \C\,, \,\,\, \eta(s) = (1 - 2^{1-s}) \zeta(s)\,,
\]
which on $[\re >0]$ is nothing else then the  pointwise limit of the $\eta$-Dirichlet series $\sum (-1)^{n+1} n^{-s}$.
It is known that
$\mu_\eta(\sigma) = \frac{1}{2}-\sigma$ for $\sigma < 0$ and $\mu_\eta(\sigma) = 0$ for $\sigma >1$. Hence
Lindel\"of's conjecture is equivalent to the question whether or not we have that
\[
\mu_\eta(\sigma)=
\begin{cases}
\frac{1}{2}-\sigma & \,\,\,\, 0 \leq \sigma < \frac{1}{2}
\\
0 & \,\,\,\, \frac{1}{2} < \sigma \leq 1.
\end{cases}
\]

The main aim here is to study holomorphic functions of finite uniform order generated by Dirichlet series from the point of view of functional analysis -- and our main inspiration comes from the classical monograph \cite{HardyRiesz} of Hardy and Riesz on so-called Riesz summation of
general Dirichlet series.

\smallskip

Given a frequency $\lambda$ and a holomorphic function $f: [\re > 0] \to \mathbb{C}$, we say that  the
$\lambda$-Dirichlet series $D= \sum a_n(D) e^{-\lambda_n s}$
is a $\lambda$-Riesz germ of $f$ whenever $D$ for some $m \ge 0$ and some $\sigma \in \mathbb{R}$  satisfies that
\[
f(s) = \lim_{x \to \infty} \sum_{x < \lambda_n} a_n(D) e^{-\lambda_ns} \Big(  1 - \frac{\lambda_n}{x}\Big)^m
\,\,\,\, \text{for all $s \in [\re > \sigma]$}.
\]
The trigonometric polynomials
\[
R_x^{\lambda,m}(f)(s)=\sum_{\lambda_n <x} a_n(D) e^{-\lambda_ns} \Big(  1 - \frac{\lambda_n}{x}\Big)^m\,,\,\, x >0
\]
are said to be the  $(\lambda,m)$-Riesz means of $f$ of order $m$.

\smallskip

An important fact will be that such $\lambda$-Riesz germs $D$ of $f$, whenever they exist~(!), are unique (Corollary~\ref{normed}) --
and so in this situation
the coefficients $a_n(f) := a_n(D)$ uniquely determine the function $f$.
\smallskip

We study a graduated  scale of  Banach spaces that consist of holomorphic functions on the right half-plane, which are of finite uniform order on $[\re >0]$ and are generated by $\lambda$-Riesz germs. More precisely, given a frequency  $\lambda=(\lambda_{n})$ and $\ell \ge 0$, the linear space
\begin{equation}\label{defin}
 H_{\infty,\ell}^\lambda[\re>0]
\end{equation}
 consists of all
holomorphic functions $f:~[\re>0] \to \mathbb{C}$, which are generated by a $\lambda$-Riesz germ and satisfy the growth condition
\begin{equation} \label{normintro}
  \|f\|_{\infty,\ell} = \sup_{\re s > 0}
\frac{|f(s)|}{(1+|s|)^\ell}< \infty\,.
\end{equation}
This defines a scale  of Banach spaces (a non-trivial fact from Theorem~\ref{completeness}),
$$\big( H_{\infty,\ell}^\lambda[\re>0], \|\pmb{\cdot}\|_{\infty,\ell}\big)_{\ell \ge0}\,,$$
 which is intimately connected with the notion of finite uniform order (Theorem~\ref{Theoremfiniteorder}):
 \begin{itemize}
   \item[$\bullet$]
   Each 
$f\in H_{\infty,\ell}^{\lambda}[\re >0]$ has finite uniform order $\nu_f \leq \ell $ on $[\re>0]$,
and conversely,
   \item[$\bullet$]
   if $f$ is a holomorphic function on $[\re>0]$, which is generated by a  $\lambda$-Riesz germ and is of finite uniform order on $[\re~>~0]$,  then for every
$k>\nu_f$ all translations $f_{\sigma}=f(\sigma+ \cdot)\in H_{\infty,k}^{\lambda}[\re > 0]\,, \, \sigma >0$.
 \end{itemize}

\bigskip

  \noindent{\bf More background.}
Mainly inspired by the classical  monograph \cite{HardyRiesz} of Hardy and  M. Riesz from 1915, the recent works \cite{Ba20}, \cite{CaDeMaSc_VV}, \cite{defantschoolmann2019Hptheory}, \cite{defant2020riesz}, \cite{defant2020variants}, \cite{HelsonBook}, and \cite{schoolmann2018bohr}  suggest a modern study of general Dirichlet series
$D=\sum a_{n}e^{-\lambda_{n}s}\,.$
Whereas the dominant  tool of the early days of this theory was complex analysis,
the idea now is to implement  modern techniques like  functional analysis, Fourier analysis, or abstract harmonic analysis
on compact abelian groups.

\smallskip
An important object of study are Hardy spaces $\mathcal{H}_p(\lambda), \, 1 \leq p \leq \infty$, which may be defined as follows:

Let $G$ be a compact  abelian group, and $\beta: (\mathbb{R},+) \to~G$ a continuous homomorphism with dense range. Then
$(G,\beta)$ is said to be $\lambda$-Dirichlet group, provided, for all $n \in \mathbb{N}$, there exists $h_{\lambda_n} \in \widehat{G}$
such that $h_{\lambda_n}~\circ~\beta =~e^{-i \lambda_n}$. The  Banach space $\mathcal{H}_p(\lambda)$ then consists
of all Dirichlet series $D=\sum a_{n}e^{-\lambda_{n}s}$ for which there is $f \in L_p(G)$  such that
$\supp{\widehat{f}} \subset  \{ h_{\lambda_n}  \colon n \in \mathbb{N}   \}$
and $\widehat{f}(h_{\lambda_n})=a_n$
 for all $n \in \mathbb{N} $, endowed with the norm $\|D\|_p = \|f\|_p$.
 It is important to note that for every frequency $\lambda$ there always exist a $\lambda$-Dirichlet group, and that  $\mathcal{H}_p(\lambda)$ does not depend on the choice of this group.
 For ordinary Dirichlet series these spaces have been introduced in \cite{Ba02} and \cite{HLS}, and this has in fact caused
 a fruitful renaissance of the analysis of such series. Within general Dirichlet series  various aspects of $\mathcal{H}_p(\lambda)$ have been studied in \cite{Ba20},
 \cite{CaDeMaSc_VV}, \cite{defantschoolmann2019Hptheory}, \cite{defant2020riesz}, \cite{defant2020variants}, and \cite{schoolmann2018bohr}.

  \smallskip

The particular case $p=\infty$ is of special interest, since then $\mathcal{H}_{\infty}(\lambda)$ may be described in terms of holomorphic functions on the right half-plane,  and in fact we are  only interested trying to develop this case further here.
\smallskip

Recall that $H_{\infty}^{\lambda}[\re > 0]$ (as defined in \cite{defant2020riesz}) denotes the linear  space of all holomorphic and bounded functions  $f\colon [\re > 0]\to \C$, which are  uniformly almost periodic function on all  vertical lines $[\re=\sigma]$
(or equivalently, some  line $[\re=\sigma]$)  with Bohr coefficients
  \begin{equation*} \label{bohrcoeffintro}
a_{x}(f)=\lim_{T\to \infty} \frac{1}{2T} \int_{-T}^{T} f(\sigma +it) e^{-(\sigma+it)x} dt, ~~ x>0.
\end{equation*}
  supported in $\{\lambda_{n} \mid n \in \N\}$. Note that here
   the integrals are independent of  the choice of $\sigma$. Together with the sup norm on the right half-plane, $H_{\infty}^{\lambda}[\re > 0]$ forms a Banach space, and by \cite[Theorem 2.16]{defant2020riesz} there is a coefficient preserving isometric linear bijection
   identifying $\mathcal{H}_{\infty}(\lambda)$ and $H_{\infty}^{\lambda}[\re > 0]$,
   \begin{equation} \label{periodic-coin}
     \mathcal{H}_{\infty}(\lambda) \,= \,H_{\infty}^{\lambda}[\re > 0]\,.
   \end{equation}
    The most fundamental theorem of the theory of almost periodic functions on $\R$ is due to Bohr, and states that a function $f:~\mathbb{R} \to \mathbb{C}$ is unifomly almost periodic if and only if it is
  uniformly approximable by trigonometric polynomials of the form $p(t)= \sum_{k=1}^{n} \alpha_k e^{-i\alpha_k t}$ with $\alpha_k \in \mathbb{R}, 1 \leq k \leq n$.
  We refer to \cite{Be54} for all needed information on (uniformly) almost periodic functions.

  \smallskip

  For bounded, holomorphic functions on half-planes the following result from
    \cite[Corollary 2.19]{defant2020riesz} (see also \cite[Lemma~3.3]{DeFeScSe2020}
  for a may be more elementary proof) should be viewed as an  analogue of Bohr's  approximation theorem:
    If $f \in H^\lambda_\infty[\re > 0]$, then for
    every $k>0$ and $s\in [\re >0]$,
\begin{equation} \label{periodic}
f(s)=\lim_{x\to \infty} \sum_{\lambda_{n}<x} a_{\lambda_n}(f)e^{-\lambda_{n}s} (1-\frac{\lambda_{n}}{x})^{k}\,,
\end{equation}
and the  convergence is even uniform on every half-plane $[\re > \sigma], \, \sigma >0$.

  \smallskip
  In  Corollary \ref{oldcase} we characterize $H_{\infty}^{\lambda}[\re > 0]$ in terms of our new setting showing that
  the canonical embedding from $H_{\infty,0}^{\lambda}[\re > 0]$ into $H_{\infty}^{\lambda}[\re>0]$
  is in fact  isometric (obvious), onto  and coefficient preserving,
  \begin{equation}\label{l=0}
    H_{\infty}^{\lambda}[\re>0] = H_{\infty,0}^{\lambda}[\re > 0]\,.
  \end{equation}
    But then in view of Bohr's  approximation theorem for uniformly almost periodic functions on $\mathbb{R}$, and in view of the approximation theorem from \eqref{periodic} for almost periodic holomorphic functions on $[ \re >0]$, the following  question arises:

  Given $\ell >0$, to which extend
  is it possible to approximate any function $f \in H_{\infty,\ell}^{\lambda}[\re>0]$ by  $(\lambda,k)$-Riesz means
  \[
  R_{x}^{\lambda,k}(D)(s) = \sum_{\lambda_{n}<x} a_n(f)e^{-\lambda_{n}s} (1-\frac{\lambda_{n}}{x})^k\,,\,\,x >0
  \]
  pointwise on  $[\re > 0]$ or even in the norm of $H_{\infty,\ell}^{\lambda}[\re>0]$? More precisely, is there some  $k >0$, and if
  yes for which set of $k$'s,  is it possible to approximate  any $f \in H_{\infty,\ell}^{\lambda}[\re>0]$
  pointwise on  $[\re > 0]$ or  in the norm of $H_{\infty,\ell}^{\lambda}[\re>0]$ by their $(\lambda,k)$-Riesz means?

  \smallskip
  A first answer to this question is given by an  important theorem of M. Riesz, first published in \cite[Theorem~41]{HardyRiesz}, which
  inspired the definition from
\eqref{defin} as well as the whole article.
It  basically shows that in view of \eqref{l=0} the validity of \eqref{periodic} exceeds the class of functions in $H_{\infty}^{\lambda}[\re>0]$ considerably, and in our terminology it reads as follows:

\begin{Theo}\label{Theo41intro}Let $\lambda$ be any frequency and $\ell \ge 0$. Then every $f\in H_{\infty,\ell}^{\lambda}[\re>0]$ is $(\lambda,k)$-Riesz summable on~$[\re>0]$ for every $k>\ell$, i.e.
$$f(s)=\lim_{x\to \infty} R_{x}^{\lambda,k}(f)(s), ~~ s\in [\re>0].$$
\end{Theo}

We reprove this result in Corollary~\ref{Theo41reproved}, where we also  carefully compare our formulation with the original formulation given in \cite[Theorem~41]{HardyRiesz}.

\bigskip

\noindent {\bf Main results.}
 Theorem~\ref{Theo41intro} is a  pointwise approximation  theorem of holomorphic functions
on half-planes by  Riesz means, and we are going to extend various of its different aspects --  including equivalent reformulations and sufficient conditions on $\lambda$ for the validity of Theorem~\ref{Theo41intro} when $k=\ell >0$ (proved in Theorem~\ref{41improved}).  But in general, it  fails for $k=\ell=0$; see the forthcoming Theorems~\ref{equivalence} and~\ref{equiv}.

Since we recover this result of Riesz in terms of the  Banach spaces $H_{\infty,\ell}^{\lambda}[\re>0]$, we in  Theorem~\ref{T1} are able to isolate  the corresponding maximal inequality:

Given a frequency $\lambda$ and $k>\ell\ge 0$, there is a constant
$C= C(k,\ell, \lambda)$ such that for all
$f \in H_{\infty,\ell}^\lambda[\re >0]$
\[
\sup_{x>0} \|R_x^{\lambda,k}(f)\|_{\infty,\ell}
\leq
C \|f\|_{\infty,\ell}\,.
\]
This then in Theorem~\ref{Perronestimate} leads to a uniform approximation theorem of
bounded sets in $H_{\infty,\ell}^{\lambda}[\re>0]$ under Riesz summation:

For every bounded set $M\subset H_{\infty,\ell}^{\lambda}[\re>0]$, $k>\ell$ and $\varepsilon, u >0$ there exists $x_{0}>0$ such that
\[
\sup_{x>x_{0}}\sup_{f\in M} \|f(u + \pmb{\cdot})-R_{x}^{\lambda,k}(f)(u + \pmb{\cdot})\|_{\infty,\ell}\le \varepsilon.
\]
Indeed,  the preceding two facts form the main body of this article, since they  turn out to be the driving forces   for most of the structure theory on the Banach spaces $H_{\infty,\ell}^{\lambda}[\re~>~0]$ we intend to present.

This applies both, to the results on $H_{\infty,\ell}^{\lambda}[\re>0]$ that we have already mentioned here (like e.g. the facts that these  spaces  are complete and intimately linked with the notion of 'finite order'), but also to a few others to come (like e.g. a Montel theorem for $H_{\infty,\ell}^{\lambda}[\re>0]$
from Theorem \ref{generalMontel}).

Studying problems for general Dirichlet series $\sum a_n e^{-\lambda_n s}$, the very  first orientation usually comes from
\begin{itemize}
\item
 the frequency $\lambda=( n)$, which after the substitution $e^{-s} = z$ generates power series $\sum a_n z^n$\,,
\item
and  the frequency  $\lambda=(\log n)$ generating ordinary Dirichlet series $\sum a_n n^{-s}$\,.
\end{itemize}
In Proposition~\ref{ex2} we show that the scale $H_{\infty,\ell}^{(n)}[\re>0]\,,\, \ell \ge0,$ collapses:
 \[
 H_{\infty}^{(n)}[\re >0] = H_{\infty,\ell}^{(n)}[\re >0]\,.
 \]
But, as shown in Proposition~\ref{ex1} the situation  for the frequency $\lambda=(\log n)$ is very different. For example
$\eta \notin H_{\infty,0}^{(\log n)}[\re >0]$, but $\eta \in H_{\infty,\ell}^{(\log n)}[\re >0]$ for all $\ell > \frac{1}{2}$.
In Proposition~\ref{ex1} we show that  $f \in H_{\infty,1}^{(\log n)}[\re >0]$ if and only if the growth condition \eqref{normintro} holds, and moreover
there exists an  ordinary Dirichlet series which on $[\re >0]$ is  Ces\`{a}ro summable with limit $f$.

\bigskip

\noindent {\bf Twofold interest.}
The proof of our main results  and  all its consequences for the structure theory of our new scale of Banach spaces $H_{\infty,\ell}^{\lambda}[\re>0]$
is  very much inspired by the theory of Riesz summation from \cite{HardyRiesz}.

Hardy and Riesz write in the preface of their book {\it '...The task of condensing any account of so extensive a theory into the compass of one of these tracts has proved an exceedingly difficult one. Many important theorems are stated without proof, and many details are left to the reader.'}

In fact, at several occasions we even need improvements of the results from  \cite{HardyRiesz}, and this has prompted us to add an appendix devoted to a  detailed, self contained, differently organized, and in parts more modern   presentation of several key stones of the theory of Hardy and Riesz.
We hope that apart from our new results on Banach spaces of holomorphic functions of finite order generated by Dirichlet series, this creates an additional benefit of our work.

\bigskip

\noindent {\bf Outlook.}
The classical Carleson-Hunt result implies that the Taylor series of
bounded holomorphic functions on the open unit disk converge almost everywhere on the boundary. In contrast, there exists an  ordinary Dirichlet series, which  on the open right half-plane converges to a bounded, holomorphic function -- but diverges at each point of the imaginary line, although its  limit function extends continuously to the closed right half plane. As a continuation of the present article we in the forthcoming
work~\cite{defant2021riesz}  study Riesz summability of functions in the spaces $H_{\infty,\ell}^\lambda[\re >0],\, \ell \ge 0$, on the imaginary
axis. Again our inspiration comes from an important, surprisingly deep  result of M.~Riesz
published (without proof) in \cite[Theorem 42]{HardyRiesz}.

\bigskip

\noindent {\bf Structure of the article.}
 In the Sections \ref{summationsection} and \ref{appendix} we develop parts of the summation theory on Riesz means, which provides us with an important integral formula of  the limit function of $\lambda $-Dirichlet series under Riesz summation,
  which then in a second step by 'Laplace inversion' leads to the so-called Perron formulas for the  summatory function of these series. Then Section~\ref{structuresection} covers the main results of this article including the basic properties of our new scale of Banach spaces $H_{\infty,\ell}^{\lambda}[\re>0], \ell \ge 0$.

\bigskip

\section{\bf Riesz summation  revisted} \label{summationsection}

The aim  of this first section is to recall the most crucial ingredients of
the theory of Riesz summation of general Dirichlet series. We isolate three fundamental pillars on which
 the theory of Riesz-summability of  general Dirichlet series $D = \sum a_n e^{-\lambda_n s}$ essentially rests : the so-called 'theorems of consistency' of such series as presented in Theorem~\ref{basic}, their  integral representation  from
Theorem~\ref{corona2}, and the Bohr-Cahen formula for their abscissas of convergence formulated in  Theorem~\ref{Bohr-Cahen-Riesz}.
We offer full self contained proof  of these results in our appendix from Section~\ref{appendix}. All important consequences,
as e.g. Perron's formula from Theorem~\ref{Perron0} are already given in this introductory section. Fixing a frequency $\lambda$, for simplicity the collection of all such $\lambda$-Dirichlet series $\sum a_n e^{-\lambda_n s}$ is denoted by $\mathcal{D}(\lambda)$\,.

\subsection{Ordinary summation} \label{summationsection0}
Every   presentation of the classical convergence theory of general Dirichlet series necessarily starts  with the following theorem, and for the sake of completeness we repeat  the standard proof based on Abel smmation (see e.g.
\cite[Theorem 1]{HardyRiesz}).

\begin{Theo}\label{corona2.00}
Let  $D=\sum a_{n}e^{-\lambda_{n}s}$ be a $\lambda$-Dirichlet series which converges  at $s_{0}\in [\re\ge 0]$.
Then $D$ converges uniformly on each cone $|\arg (s-s_{0})|\le \gamma < \frac{\pi}{2}$.
\end{Theo}

\begin{proof}
Assume without loss of generality that $s_0 =0$, and fix some $\gamma < \frac{\pi}{2}$ and $0\neq s \in \mathbb{C}$ with $|\arg (s)|\le \gamma$. Then by standard Abel summation
for   all~$m < n$
\begin{align} \label{abelanfang}
  \sum_{\nu =m}^n a_n e^{-\lambda_n s} = \sum_{\nu =m}^{n-1} \Big(\sum_{\ell =m}^{\nu}a_{\ell}\Big)(e^{-\lambda_\nu s}-e^{-\lambda_{\nu+1} s}) +
\Big(\sum_{\nu =m}^{n}a_{\nu}\Big)e^{-\lambda_n s}\,.
\end{align}
But with $\re~s=\sigma$
\[
|e^{-\lambda_\nu s}-e^{-\lambda_{\nu+1} s}|  \leq \frac{|s|}{\re s}(e^{-\lambda_\nu \sigma}-e^{-\lambda_{\nu+1} \sigma})
\leq \sec{\gamma}\,(e^{-\lambda_\nu \sigma}-e^{-\lambda_{\nu+1} \sigma})\,,
\]
and hence, for $m$ and $n$ large, we have that
\[\Big|\sum_{\nu =n}^m a_n e^{-\lambda_n s}\Big| \leq \varepsilon\]
uniformly on the cone $|\arg (s)|\le \gamma $\,.
\end{proof}

Theorem~\ref{corona2.00} suggests the following basic definition.
For each $\lambda$-Dirichlet series  $D=\sum a_{n}e^{-\lambda_{n}s}$  the number
\[
\sigma^\lambda_c(D) = \inf \big\{ \sigma \in \R  \colon D \,\,\,\text{converges at $\sigma$}\big\} \in \mathbb{R} \cup\{\pm \infty\}
\]
is called  abscissa of convergence of $D$. Obviously,  $D$ converges on $[\re > \sigma^\lambda_c(D)]$ and diverges on $[\re <  \sigma^\lambda_c(D)]$\,.

\smallskip
The following integral formula describes  the limit function $f$ of a $\lambda$-Dirichlet series $D=\sum a_{n}e^{-\lambda_{n}s}$
on the half-plane  $[\re > \sigma^\lambda_c(D)]$ of convergence in terms of the  Laplace transform of the so-called summatory function
which for $s \in \mathbb{C}$ and $t \geq  0$ is given~by
\[
S^\lambda_{t}(D)(s)=\sum_{\lambda_{n}<t} a_{n}e^{-\lambda_{n}s}\,.
\]
Our proof  is inspired by Helson \cite[(2.3)]{HelsonBook}, and serves as a model for the proof of the much more involved variant  Theorem \ref{corona2}  on Riesz summation.

\smallskip

\begin{Theo}\label{corona2.0} Let  $D=\sum a_{n}e^{-\lambda_{n}s}$ be convergent at $s_{0}\in [\re\ge 0]$. Then $D$ converges on $[\re> \re~s_{0}]$, and the limit function
\begin{equation} \label{anfang}
f\colon [\re>\re~s_{0}] \to \C, ~~ s\mapsto \lim_{x\to \infty} \sum_{\lambda_{n}<x} a_{n}e^{-\lambda_{n}s}
\end{equation}
is holomorphic and satisfies for all $s\in [\re>\re~s_{0}]$
\begin{equation*}\label{laplace3}
\frac{f(s)}{s}=\int_{0}^{\infty}e^{-s t}  \sum_{\lambda_{n}<t} a_{n}  \,\,dt.
\end{equation*}
Moreover, the convergence in \eqref{anfang} is  uniform on each cone $|\arg (s-s_{0})|\le \gamma < \frac{\pi}{2}$.
\end{Theo}

\bigskip

Note that the  preceding theorem shows that the function $s \mapsto f(s)/s$ is nothing else than the Laplace transform
of the summatory function $t \mapsto  S_t(D)(0)$.

\bigskip

We prepare the proof with two simple lemmas.

\smallskip

\begin{Lemm}\label{Abelfirstcase}
Let $D\in \mathcal{D}(\lambda)$. Then for all $s,w \in \mathbb{C}$ and all $x \geq 0$
\begin{equation*} \label{Abelfirstcase}
S^\lambda_{x}(D)(s+w)=S^\lambda_{x}(D)(w)e^{-sx}- \int_{0}^{x} S^\lambda_{t}(D)(w)s e^{-st} dt.
\end{equation*}
\end{Lemm}

\begin{proof}
Since $S^\lambda_{x}(D)(s+w) = S^\lambda_{x}(D_{w})(s)$, we concentrate on the case $w=0$; recall that here $D_w$ stands for the translation of $D = \sum a_n e^{-\lambda_n s}$ about $w$,
i.e. $D_w = \sum a_n e^{-\lambda_n w} e^{-\lambda_n s}$. Then
\begin{align*}
S^\lambda_{x}(D)(s) = \int_0^x  e^{-ts} dS^\lambda_{t}(D)(0)
= \int_0^x  e^{-ts} \big(S^\lambda_{\bullet}(D)(0)\big)'(t) dt
\,,
\end{align*}
where the first integral is a  Stieltjes integral. Obviously, summatory functions are  of bounded variation, hence by partial integration
(see e.g. Helson~\cite[Appendix]{Helson3})
\begin{equation*}
S^\lambda_{x}(D)(s) = e^{-xs} S^\lambda_{x}(D)(0)
- \int_0^x  s e^{-ts} S^\lambda_{t}(D)(0) dt\,. \qedhere
\end{equation*}
\end{proof}

\smallskip

\begin{Lemm}\label{Estimate0}
Let $D\in \mathcal{D}(\lambda)$, and $s_0, s \in \mathbb{C}$. Then for all $x >0$
\begin{equation*}
|S_x^\lambda(D)(s)| \leq (1+|s_0|) e^{x\re s_0}  \sup_{ y \leq x} |S_y^\lambda(D)(s_0 + s)|\,.
\end{equation*}
\end{Lemm}

\begin{proof}
Since $S_x^\lambda(D)(s)= S_x^\lambda(D_s)(0)$, we may assume that $s=0$.
By Lemma~\ref{Abelfirstcase}
\begin{equation} \label{as}
S_x^\lambda(D)(0)= S_x^\lambda(D)(-s_0+ s_0)  =
e^{s_0x}S^\lambda_{x}(D)(s_0)- \int_{0}^{x} S^\lambda_{t}(D)(s_0)(-s_0) e^{s_0t} dt.
\end{equation}
Then for the first summand we  have
\[
|e^{s_0x}S^\lambda_{x}(D)(s_0)|\leq e^{x\re s_0} \sup_{y \leq x} |S_y^\lambda(D)(s_0)|\,,
\]
and for the second
\[
\Big| \int_{0}^{x} S^\lambda_{t}(D)(s_0)s_0 e^{s_0t} dt \Big|
\leq
|s_0| e^{x\re s_0} \sup_{y \leq x} |S_y^\lambda(D)(s_0)|
\,.\qedhere
\]
\end{proof}

\bigskip

\begin{proof}[Proof of Theorem~\ref{corona2.0}]
Deduce  first  by Lemma~\ref{Estimate0} (with $s =0$) and the assumption ($D$ converges in $s_0$) that there is a constant $C= C(s_0) >0$ such that for all $x>0$
\begin{equation} \label{now}
 |S_x^\lambda(D)(0)| \leq C e^{x \re s_0} \,.
\end{equation}
Now fix some $s\in [\re>\re~s_{0}]$.
Using Lemma~\ref{Abelfirstcase},  for all $x >0$
\begin{equation} \label{start1}
S^\lambda_{x}(D)(s)=S^\lambda_{x}(D)(0)e^{-sx}- \int_{0}^{x} S^\lambda_{t}(D)(0)s e^{-st} dt.
\end{equation}
By \eqref{now}  for all $x>0$
\[
e^{-x\re s} |S_x^\lambda(D)(0)| \leq C e^{x (\re s_0-\re s)} \,,
\]
which converges to $0$ whenever $x \to \infty$. Hence it remains to show that
\begin{equation*} \label{lebesgue}
\lim_{x \to \infty}\int_{0}^{x} S^\lambda_{t}(D)(0) e^{-st} dt = \int_{0}^{\infty} S^\lambda_{t}(D)(0) e^{-st} dt\,.
\end{equation*}
But this follows from the dominated convergence theorem -- indeed, for $x >0$ we conclude from another application of  \eqref{now} that  for all $t >0$
\[
|S^\lambda_{t}(D)(0) e^{-st} \chi_{[0,x]}(t)| \leq C e^{t \re s_0}e^{-t \re s}  \chi_{[0,x]}(t)
\leq C e^{t (\re s_0- \re s)}\,.
\]
The last statement of the theorem is then a consequence of Theorem~\ref{corona2.00}.
 \end{proof}

\smallskip

The following formula, which we for historical reasons  call Bohr-Cahen formula, is taken from \cite[Theorem 7]{HardyRiesz}.

\smallskip

\begin{Coro}
\label{Bohr-Cahen} Let $D\in \mathcal{D}(\lambda)$. Then
$$\sigma_{c}^{\lambda}(D)\le \limsup_{x\to \infty} \frac{\log( |S_{x}^{\lambda}(D)(0)|)}{x},$$
with equality whenever $\sigma_{c}^{\lambda}(D)$ is non-negative.
  \end{Coro}

\begin{proof}
We denote the limes superior  by $L$, assume that it is finite, and choose some $\sigma_{0}>L$.
Then there is a constant $C >0$ such that for all $x >0$ we have that
\[
|S_{x}^{\lambda}(D)(0)| \leq C e^{\sigma_0 x}\,.
\]
Now, starting as in  \eqref{start1}, we show that $D$ converges on $[\re> \sigma_0]$, so $\sigma^{\lambda}_{c}(D) \leq L$.
 To finish, assume that $\sigma^{\lambda}_{c}(D)\ge 0$. Let $\varepsilon>0$, and define $\sigma_{0}=\sigma^{\lambda}_{c}(D)+\varepsilon$. Then  $D$ converges at $\sigma_{0}$ (Theorem~\ref{corona2.00}), and so by Lemma \ref{Estimate0} for all $x >0$
$$
|S_{x}^{\lambda}(D)(0)|\le C(\sigma_{0})e^{\sigma_{0}x}\,,
$$
implying $L\le \sigma_{0}=\sigma^{\lambda}_{c}(D)+\varepsilon$ for all $\varepsilon>0$.
\end{proof}

\smallskip

Given a frequency $\lambda$, we  also need the abscissa of uniform convergence of a $\lambda$-Dirichlet series  $D=\sum a_{n}e^{-\lambda_{n}s}$ determined by the number
\[
\sigma^\lambda_u(D) = \inf \big\{ \sigma \in \R  \colon D \,\,\,\text{converges uniformly on  $[\re > \sigma]$}\big\}\,.
\]
Clearly,  $D$ for each $\varepsilon >0$ converges uniformly on $[\re >  \sigma^\lambda_u(D) +\varepsilon ]$, whereas it  does not converge  uniformly on $[\re >  \sigma^\lambda_u(D)- \varepsilon ]$\,.

\smallskip

We finish with the Bohr-Cahen formula for this abscissa.

\smallskip

\begin{Coro}
\label{Bohr-Cahen-uniform} Let $D\in \mathcal{D}(\lambda)$. Then
$$\sigma_{u}^{\lambda}(D)\le \limsup_{x\to \infty} \frac{\log \big( \sup_{t \in \mathbb{R}}|S_{x}^{\lambda}(D)(it)|\big)}{x},$$
with equality whenever $\sigma_{u}^{\lambda}(D)$ is non-negative.
  \end{Coro}

  \smallskip

   In fact this is
 a relatively simple  consequence of Corollary~\ref{Bohr-Cahen},  once one realizes that all we proved so far, after  proper modifications,
 also holds for $\lambda$-Dirichlet series $E=\sum x_{n}e^{-\lambda_{n}s}$ with coefficients $x_n$ in a Banach space $X$.
\smallskip

 \begin{proof}
Define
 $X = H_\infty[\re >0]$ (the Banach  space of all holomorphic and bounded functions on the right half-plane) and the
 $X$-valued $\lambda$-Dirichlet series
 \[
 E = \sum (a_n e^{-\lambda_n s}) e^{-\lambda_n w}
 \]
  with the coefficients $x_n = a_n e^{-\lambda_n s}\in X$. We easily  see that
  \[
  \sigma_{u}^{\lambda}(D)=   \sigma_{c}^{\lambda}(E)\,,
  \]
  and that by the Hahn-Banach theorem and the maximum modulus theorem we,   for all $x >0$,  have
  \[
   \|S_{x}^{\lambda}(E)(0)\|_X = \sup_{\re s >0}|S_{x}^{\lambda}(D)(s)|=\sup_{t \in \mathbb{R}}|S_{x}^{\lambda}(D)(it)|\,.
  \]
  Then by the vector-valued extension of Corollary~\ref{Bohr-Cahen} we immediately obtain

  \begin{align*}
    \sigma_{u}^{\lambda}(D)=
  \sigma_{c}^{\lambda}(E)
  &
  \le \limsup_{x\to \infty} \frac{\log( \|S_{x}^{\lambda}(D)(0)\|_X)}{x}
  \\&
  =\limsup_{x\to \infty} \frac{\log \big( \sup_{t \in \mathbb{R}}|S_{x}^{\lambda}(D)(it)|\big)}{x},
  \end{align*}
  with equality whenever $\sigma_{u}^{\lambda}(D)=\sigma_{c}^{\lambda}(E) \ge 0$.
 \end{proof}

\bigskip

\subsection{Riesz means} \label{Riesz means}

All definitions following are  inspired by  \cite{HardyRiesz} (see also~\cite{defant2020riesz}). Let $\lambda$ be a frequency, $k \ge 0$, and
$C=\sum a_n$ a series in a normed  space $X$. Then $C$ is said to be  $(\lambda,k)$-Riesz summable whenever
the limit
\[
\lim_{x \to \infty} \sum_{\lambda_n < x} \Big(1- \frac{\lambda_n}{x} \Big)^k a_n
\]
exists, and the finite sums
\[
R_x^{\lambda, k} (C) := \sum_{\lambda_n < x} \Big(1- \frac{\lambda_n}{x} \Big)^k a_n\,, \,\,\, x >0
\]
are called   $(\lambda,k)$-Riesz means of $C=\sum a_n $ of first kind.
More generally, we consider functions
$$r: [0,\infty[ \times [0,\infty[ \to \mathbb{R}_{\ge 0}\,,\,\,\, (x,t) \mapsto r(x,t)\,, $$
which are  $C^\infty$ in $t$
and satisfy $r(x,x) = 0$ for all $x \ge 0$, and call them  Riesz weights. This definition helps to unify some of our coming proofs --
although we mainly focus on the two special examples:
\[
u_k(x,t) = (x-t)^k \,\,\,\, \text{ and } \,\,\,\, v_k(x,t) = (e^x-e^t)^k \,.
\]
Given a   frequency $\lambda$, a series $C=\sum a_n$  (in a normed  space),
and a Riesz weight $r$,  we define the summatory function by
\[
S^{\lambda,r}_{x}(C) = \sum_{\lambda_n < x} a_n  r(x, \lambda_n)\,,\,  \,\,\,\, x \ge 0\,,
\]
and for the most important cases $u_k$ and $v_k$ we
  abbreviate
\[
S^{\lambda,k}_{x}(C) = S^{\lambda,u_k}_{x}(C) \,,
\]
and
\[
U^{\lambda,k}_{x}(C)= S^{\lambda,v_k}_{x}(C)\,.
\]
 Note that,
 if $C= \sum a_n$  is $(e^\lambda,k)$-Riesz summable, then
\begin{equation*}
 \lim_{x \to \infty}
R_{x}^{e^\lambda,k}(C)
=
\lim_{x \to \infty}
R_{e^x}^{e^\lambda,k}(C)
=
\lim_{x \to \infty} T_{x}^{\lambda,k}(C)\,,
\end{equation*}
where
\begin{equation}\label{secondkind}
 T_{x}^{\lambda,k}(C)= \sum_{\lambda_n < x} \Big(1- \frac{e^{\lambda_n}}{e^x} \Big)^k
 a_n = e^{-kx} U^{\lambda,k}_{x}(C)
\end{equation}
are the $(\lambda,k)$-Riesz means of second kind.

\smallskip

 Hardy and Riesz in \cite[Theorem 16, 17]{HardyRiesz}
collected the following two  basic properties of Riesz summability (in the order given).
Since we try to keep our article as self contained as possible, we are going to give  proofs with full details  in our appendix from Section~\ref{appendix}
(in particular, the original arguments for the  second statement are, according to  \cite[p.30]{HardyRiesz},  somewhat
'intricate').

\smallskip

\begin{Theo}\label{basic}
  Let $\lambda$ be a frequency, $k \ge 0$, and $D=\sum a_n$ a series in a normed space~$X$.
\begin{itemize}
\item[(i)]
If $D$  is $(\lambda,k)$-Riesz summable, then it is
  $(\lambda, \ell)$-Riesz summable for any $\ell \geq k$, and the associated limits coincide.
    \item[(ii)]
 If $D$  is $(e^\lambda,k)$-Riesz summable, then
  $D$  is $(\lambda,k)$-Riesz summable,  and the associated limits coincide.
    \end{itemize}
\end{Theo}

\smallskip

Let us come back to $\lambda$-Dirichlet series $D=\sum a_{n}e^ {-\lambda_{n}s}$. For  such series  and $k\ge 0$, the Dirichlet polynomials
\begin{equation*}\label{Rmean}
  R_{x}^{\lambda,k}(D)(s)=\sum_{\lambda_{n}<x}a_{n}e^{-\lambda_{n}s}(1-\frac{\lambda_{n}}{x})^{k},~~~ \,\,x \ge 0, s \in \mathbb{C}
\end{equation*}
are called  $(\lambda,k)$-Riesz means of $D$, and  $D$  is said to be  $(\lambda,k)$-Riesz summable at $s_{0}\in \C$, whenever the limit
\begin{equation*}\label{Rieszlimitintro}
\lim_{x\to \infty} R_{x}^{\lambda,k}(D)(s_{0})
\end{equation*}
exists. Clearly, by Theorem~\ref{basic},(ii)  we know that $D=\sum a_{n}e^ {-\lambda_{n}s}$ is $(e^\lambda,k)$-Riesz summable at $s_{0}\in \C$, if
\begin{equation*}\label{Rieszlimitintro}
\lim_{x\to \infty} T_{x}^{\lambda,k}(D)(s_{0})
\end{equation*}
exists.

\smallskip

 A~basic property states that the $(\lambda,k)$-Riesz summability  of $D$ at $s_{0}$ implies the $(\lambda,k+\varepsilon)$-Riesz summability of $D$ at $s_{0}$ for every $\varepsilon>0$ (see  Theorem~\ref{basic},(i)).
  As done for $k = 0$ above,  we define
\[
\sigma_c^{\lambda,k} (D) = \inf \big\{ \sigma \in \R  \colon D \,\,\,\text{$(\lambda,k)$-Riesz summable  in $\sigma$}\big\}\,,
\]
which characterizes  the largest possible open half-plane on which $D$ is pointwise $(\lambda,k)$-Riesz summable.
Similarly, we put
\[
\sigma_u^{\lambda,k} (D) = \inf \big\{ \sigma \in \R  \colon D \,\,\,\text{uniformly $(\lambda,k)$-Riesz summable  in $\sigma$}\big\}\,,
\]
defining  the largest possible open half-plane on which $D$ is uniformly $(\lambda,k)$-Riesz summable.

\smallskip

Let us  compare  summability with respect to Riesz means of first and second kind. Note first that, given a
$\lambda$-Dirichlet series $D$,  we by Theorem~\ref{basic},(ii)
 always have
$ \sigma_c^{\lambda,k} (D) \leq  \sigma_c^{e^\lambda,k} (D)$, where
$$\sigma_c^{e^\lambda,k}  (D)$$ now characterizes the   largest possible open half-plane on which  $D$ is  $(e^\lambda,k)$-Riesz summable.  Hardy and Riesz in \cite[Theorem~30]{HardyRiesz} even prove equality.

\smallskip
\begin{Theo} \label{niceproof?}
  For every $\lambda$-Dirichlet series $D$ and every $k > 0$
  \[
  \sigma_c^{\lambda,k} (D) = \sigma_c^{e^\lambda,k} (D)\,.
  \]
\end{Theo}

 The proof of this result is only sketched in \cite{HardyRiesz}. As a by-product of what we intend to do, we in our appendix from Section~\ref{appendix} will give a full self contained proof.

\bigskip

\subsection{Integral representation}
The following analog of Theorem~\ref{corona2.0} rules the theory of Riesz summation of general Dirichlet series.
The result is taken from \cite[Theorem 23]{HardyRiesz}.
 In fact,  \eqref{laplace3} is an improvement of the  original form  from \cite{HardyRiesz}, which we repeat in Remark~\ref{ori!},(i).

\smallskip
Given $k\ge0$ and a frequency $\lambda$, the theory of $(\lambda,k)$-Riesz summabilty of $\lambda$-Dirichlet series $\sum a_n e^{-\lambda s}$ is ruled by the so-called summatory functions of order $k$, which for $s \in \C$ and $ x>0$ are defined  by
\begin{equation}\label{seed}
 S_{x}^{\lambda,k}(D)(s)=\sum_{\lambda_{n}<x} a_{n}e^{-\lambda_{n}s} (x-\lambda_{n})^{k}=x^{k} R_{x}^{\lambda,k}(D)(s).
\end{equation}

\begin{Theo}\label{corona2} Let $k\ge 0$, and  $D=\sum a_{n}e^{-\lambda_{n}s}\in \mathcal{D}(\lambda)$ be $(\lambda,k)$-summable at $s_{0}\in [\re\ge 0]$. Then $D$ is $(\lambda,k)$-summable on $[\re> \re~s_{0}]$ and the limit function
\begin{equation} \label{corona1new}
f\colon [\re>\re~s_{0}] \to \C, ~~ s\mapsto \lim_{x\to \infty} \sum_{\lambda_{n}<x} a_{n}(1-\frac{\lambda_{n}}{x})^{k}e^{-\lambda_{n}s},
\end{equation}
is holomorphic and satisfies for all $s\in [\re>\re~s_{0}]$
\begin{equation}\label{laplace3}
\Gamma(1+k) \frac{f(s)}{s^{1+k}}=\int_{0}^{\infty}e^{-s t}  S^{\lambda,k}_{t}(D)(0) dt,
\end{equation}
where as usual $\Gamma$ denotes the Gamma function.
Moreover, the convergence in \eqref{corona1new} is  uniform on each cone $|\arg (s-s_{0})|\le \gamma < \frac{\pi}{2}$.
\end{Theo}

The proof, although  far more complex, is analog to that for the case $k=0$ stated in  Theorem~\ref{corona2.0}. A proof of the two cases
$k \in \mathbb{N}_0$ and $0<k<1$ is given in  \cite{HardyRiesz},
whereas  the proof of the general case is omitted. Very much inspired by the ideas used in  \cite{HardyRiesz}, we give a self contained proof with full details in our appendix from Section~\ref{appendix} which incorporates all different cases simultaneously.

\smallskip
For  the rest of this section we collect a   few crucial consequences of the preceding theorem.

\smallskip
\begin{Rema} \label{ori}
As above, we like to mention that the integral formula from  \eqref{laplace3} may be read as follows: The function $s\mapsto \Gamma(1+k) \frac{f(s)}{s^{1+k}}$ is the Laplace transform of $t\mapsto S^{\lambda,k}_{t}(D)(0)$, in short
$$\mathcal{L}\big(t\mapsto S^{\lambda,k}_{t}(D)(0)\big)(s)=\Gamma(1+k)\frac{f(s)}{s^{1+k}}.$$
The case $f=1$ of \eqref{laplace3} (assuming $\lambda_{1}=0$) reproves that
$$\mathcal{L}(t^{k})(s)=\frac{\Gamma(1+k)}{s^{1+k}}, ~~ s\in [\re>0].$$
\end{Rema}

\smallskip

The following remark collects  two useful reformulations of Theorem~\ref{corona2} in terms of the
abscissa of $(\lambda,k)$-Riesz summability of $D = \sum a_n e^{-\lambda_n s}$.
\smallskip

\begin{Rema} \label{ori!}
\label{together}
 Let $k\ge 0$, and  $D=\sum a_{n}e^{-\lambda_{n}s}\in \mathcal{D}(\lambda)$. Assume that $\sigma_{c}^{\lambda,k}(D)\in \R$, and let  $f:[\re > \sigma_{c}^{\lambda,k}(D)] \to \mathbb{C}$ be the limit function of $D$.
\smallskip
\begin{itemize}
 \item[(i)]
                           For every  $s_0\in [\re \ge \sigma_{c}^{\lambda,k}(D)]$ and $s\in [\re>0]$
              \begin{equation*}\label{laplace}
f(s+s_{0})\frac{\Gamma(1+k)}{s^{1+k}}=\int_{0}^{\infty}e^{-s t}  S^{\lambda,k}_{t}(D)(s_{0}) dt.
\end{equation*}
 \item[(ii)]
 For every  $s_0\in [\re > \sigma_{c}^{\lambda,k}(D)]$ and $s\in [\re>0]$
$$f(s+s_{0})=f(s_{0})+\frac{s^{1+k}}{\Gamma(1+k)} \int_{0}^{\infty} e^{-st} \big(S_{t}^{\lambda,k}(D)(s_{0})-f(s_{0})\big) dt.$$
\end{itemize}
             \end{Rema}

\begin{proof}
In order to prove $(i)$ assume first that $\re s_0 > \sigma_{c}^{\lambda,k}(D)$, so $D$ is $(\lambda,k)$-summable in $s_0$.
Clearly, we obtain $(i)$, if we apply  Theorem \ref{corona2} to the translation $D_{s_0}= \sum a_{n}e^{-\lambda_{n}s_0}e^{-\lambda_{n}s}$.
If $\re s_0 = \sigma_{c}^{\lambda,k}(D)$, then
take $s=\sigma+i\tau$ and $\varepsilon>0$ such that $2\varepsilon<\sigma$. Then we conclude from the first case  that
$$f(s+ s_0+\varepsilon)\frac{\Gamma(1+k)}{s^{1+k}}=\int_{0}^{\infty}e^{-s t}  S^{\lambda,k}_{t}(D)(s_0+\varepsilon) dt,$$
and hence by continuity of $f$ the claim follows once we prove that
\begin{equation*} \label{kerli}
\lim_{\varepsilon\to 0} \int_{0}^{\infty}e^{-s t}  S^{\lambda,k}_{t}(D)(s_0+\varepsilon) dt=\int_{0}^{\infty}e^{-s t}  S^{\lambda,k}_{t}(D)(s_0) dt.
\end{equation*}
Indeed, by (the forthcoming) Lemma \ref{alwaysabel} we have
\begin{equation*}
|R^{\lambda,k}_{t}(D)(s_0+\varepsilon)|\le e^{(\sigma-\varepsilon) t} |\sup_{0<y<t}
|R^{\lambda,k}_{y}(D)(s_0+\varepsilon + (\sigma-\varepsilon))|\le C(\sigma)e^{\frac{\sigma}{2} t}\,,
\end{equation*}
and  hence the dominated convergence theorem does the job. Finally, in order to see  $(ii)$, apply $(i)$ to the function $g(s)=f(s)-f(s_{0})$, assuming without loss of generality  that $\lambda_{1}=0$.
\end{proof}

\smallskip

\begin{Coro}
Let $k, D$ and $f$ be as in Theorem \ref{corona2}. Then
$$ \forall ~\varepsilon, \delta>0 \,\,\,\exists ~t_{0}>0  \,\,\, \forall~ |t|\ge t_{0},~
\sigma>\delta + \sigma_c^{\lambda,k}(D)\colon |f(\sigma +it)|\le \varepsilon |t|^{k+1}.$$
In particular, the function  $f$ for every  $\delta>0$ has finite uniform order $\leq k+1 $ on $[\re >  \delta + \sigma_c^{\lambda,k}(D)]$.
\end{Coro}

\smallskip
The result is taken from \cite[Theorem 38]{HardyRiesz}.

\smallskip

\begin{proof}
Without loss of generality  we assume that $D$ is $(\lambda,k)$-Riesz summable in~$0$. We  choose $\varepsilon, \delta>0$, an arbitrary  $0 < \gamma \leq \frac{\pi}{2}$, and put
\[
M = \sup_{x >0} | R^{\lambda,k}_{x}(D)(0)| <\infty.
\]
In a first step we  consider
$f$ on the complement of $[|\arg \bullet |\le \gamma]$ in $[\re >0]$ intersected with $[\re > \delta]$, i.e. all $s = \sigma + it \in [\re > \delta] $ such that $\frac{|s|}{|t|} \leq  \frac{1}{\sin \gamma}$.
Fix such $s$, and choose $x_0 >0$ such that
\[
\int_{x_0}^\infty y^k e^{-\delta y} dy < \frac{\Gamma(1+k) \,(\sin\gamma)^{1+k}}{M} \, \varepsilon.
\]
Applying again  Theorem~\ref{corona2}, we see that
\begin{align*}
f(s)
= \frac{s^{k+1}}{\Gamma(1+k)}   \int_{0}^{x_0}
&
S_{y}^{\lambda,k}(D)(0)
e^{- s y} dy\\&
+
 \frac{s^{k+1}}{\Gamma(1+k)}\int_{x_0}^{\infty} S_{y}^{\lambda,k}(D)(0)e^{- s y} dy
=
J_1(s) + J_2(s)\,.
  \end{align*}
  Then
  \[
  J_2(s) \leq |s|^{k+1} \frac{M}{\Gamma(1+k)} \int_{x_0}^\infty y^k e^{-\delta y} dy  \leq
 \varepsilon |t|^{k+1} \,.
  \]
  Moreover,  by partial integration we have
  \[
  J_1(s) = \frac{s^k}{\Gamma(1+k)} \Big(  - S_{x_0}^{\lambda,k}(D)(0) e^{-s x_0} +
  \int_{0}^{x_0}\big( S_{\bullet}^{\lambda,k}(D)(0)  \big)'(y)e^{-s y} dy\Big)\,,
  \]
  and then, since $|e^{-s y}| \leq 1$ for all $0 \leq y \leq x_0$, there is $C = C(x_0,k) > 0$ such that
  \[
  |J_1(s)| \leq |s|^k C \leq \frac{C}{ \sin \gamma} |t|^k \leq  \varepsilon |t|^{k+1}\,,
  \]
 whenever $\frac{C}{ \varepsilon\sin \gamma } \leq |t|$.  All in all this shows that $|f(s)| \leq \varepsilon |t|^{k+1}$
 for all $s = \sigma + it$  in the complement of $[|\arg \bullet |\le \gamma]$ in $[\re >0]$ intersected with $[\re > \delta]$  provided $\frac{C}{ \varepsilon\sin \gamma } \leq |t|$.
 Using again
 Theorem~\ref{corona2}, we know that $f$ is bounded on the cone $[|\arg \bullet |\le \gamma]$, which then
 clearly leads to the desired conclusion.
 \end{proof}

\bigskip

\subsection{Perron formula}

We prove Perron's formula -- an integral formula for the  summatory function of a $\lambda$-Dirichlet series
in terms of its limit function. The  result was first presented in  \cite[Theorem 39]{HardyRiesz}.  In fact, it turns out that it follows by Fourier inversion from
the integral formula in Theorem~\ref{corona2}. The original proof of Perron's formula for general Dirichlet series as given in
\cite{HardyRiesz} is mainly a consequence of  Cauchy's integral theorem. Our approach, deeply inspired by  ideas used in \cite{HelsonBook} to cover
the  case~$k=0$,  differs considerable.

\smallskip

\begin{Theo}
\label{Perron0}
For $k \geq 0$ let  $D=\sum a_{n} e^{-\lambda_{n}s}$ be somewhere $(\lambda,k)$-Riesz summable and
 $f: [ \re >\sigma_c^{\lambda,k}(D)] \to \mathbb{C}$ its limit function. Then
for each $x>0$ and $c >\max\{\sigma_{c}^{\lambda,k}(D),0\}$ we have
$$S_{x}^{\lambda,k}(D)(0)=\frac{\Gamma(1+k)}{2\pi i}  \int_{c-i\infty}^{c+i\infty} \frac{f(s)}{ s^{1+k}} e^{xs} ds.$$
\end{Theo}

\begin{proof} Let us denote by $\mathcal{F}_{L_{1}(\R)}$ the Fourier transform on $L_{1}(\R)$ (with respect to the ordinary Lebesgue measure). Then for all $y\in \R$ by Theorem \ref{corona2}
\begin{align*}
\frac{f(c + iy)}{ (c + iy)^{1+k}}
&
=\frac{1}{\Gamma(1+k)}\int_{0}^{\infty}
 \big[S_{t}^{\lambda,k}(D)(0)e^{-c t}\big]  e^{-iy t}  dt
 \\&
 = \frac{1}{\Gamma(1+k)}\mathcal{F}_{L_1(\mathbb{R})}\big( t\mapsto S_{t}^{\lambda,k}(D)(0) e^{-c t }  \big)(y) \,.
\end{align*}
Since the function $g(t)=S_{t}^{\lambda,k}(D)(0)e^{-ct}$ satisfies a Lipschitz condition at all $t\notin\{\lambda_{n} \mid n\in \mathbb{N} \}$, we by \cite[Corollary on p.9]{Helson}  for all such $t$ obtain
$$S_{t}^{\lambda,k}(D)(0)e^{-ct}=\lim_{A\to \infty} \frac{1}{2\pi} \int_{-A}^{A} \Gamma(1+k)\frac{f(c + iy)}{ (c + iy)^{1+k}} e^{iyt} dy.$$
Hence,
\begin{align*}
S_{t}^{\lambda,k}(D)(0)&=\frac{\Gamma(1+k)}{2\pi} \lim_{A\to \infty} \int_{-A}^{A} \frac{f(c + iy)}{ (c + iy)^{1+k}} e^{t(c+iy)} dy
\,,
\end{align*}
which is the equality we intended to prove.
\end{proof}

\smallskip

It is important for our later purposes that the preceding integral formula for the summatory function of a given $\lambda$-Dirichlet series $D$, can be extended considerably, whenever the limit function of $f$ extends holomorphically  to the full right half-plane -- still satisfying a
'finite order type-condition' (a result sketched in \cite[p.51]{HardyRiesz}).

\smallskip
\begin{Coro} \label{werdervsgladbach}
For $k\ge 0$ let $D=\sum a_{n}e^{-\lambda_{n}s}$ be $(\lambda,k)$-Riesz summable for some $s_{0}\in [\re>0]$. Assume that the limit function $f$ of $D$ extends holomorphically to $[\re>0]$ such that
\[
\forall \varepsilon, \delta>0 \,\,\,\exists t_{0}>0  \,\,\, \forall |t|\ge t_{0}, \sigma>\delta \colon |f(\sigma +it)|\le \varepsilon |\sigma+it|^{1+k}.
\]
Then
for each $x>0$ and $c >0$
$$S_{x}^{\lambda,k}(D)(0)= \frac{\Gamma(1+k)}{2\pi i} \int_{c-i\infty}^{c+i\infty}  \frac{f(s)}{s^{1+k}} e^{xs} ds.$$
\end{Coro}
\begin{proof}  By Theorem \ref{Perron0} we know that that for $x >0$ and $\sigma>\re s_0$
$$S_{x}^{\lambda,k}(D)(0)= \frac{\Gamma(1+k)}{2\pi i} \int_{\sigma-i\infty}^{\sigma+i\infty}  \frac{f(s)}{s^{1+k}} e^{xs} ds.$$
We claim that for each $x>0$ and $c >0$
\begin{equation} \label{cauchy}
\int_{c-i\infty}^{c+i\infty}  \frac{f(s)}{s^{1+k}} e^{xs} ds=\int_{\sigma-i\infty}^{\sigma+i\infty}  \frac{f(s)}{s^{1+k}} e^{xs} ds.
\end{equation}
Indeed, an application of Cauchy's integral theorem shows that \eqref{cauchy} follows, once we prove that
\begin{equation}\label{cauchyclaim}
\lim_{T\to \infty} \int_{c}^{\sigma}  \frac{f(y\pm iT)}{(y\pm iT)^{1+k}} e^{x(y \pm iT)} dy=0.
\end{equation}
To see this, choose, given  $\varepsilon>0$ and $\delta =  c$,  some  $t_{0} >0$ according to the assumption. Then for all $T>t_{0}$
\begin{align*}
\Big|\int_{c}^{\sigma}  \frac{f(y\pm iT)}{(y\pm iT)^{1+k}} e^{x(y \pm iT)} dy\Big|\le \varepsilon \int_{c}^{\sigma} e^{xy} dy=\varepsilon (e^{x\sigma}-e^{xc})x^{-1},
\end{align*}
which implies \eqref{cauchyclaim}.
\end{proof}

\smallskip

Finally, we show that the coefficients of a $\lambda$-Dirichlet series are uniquely determined by the values of its limit function on some abscissa - another crucial point for our coming purposes.

\smallskip

\begin{Coro}\label{normed} Let $D=\sum a_{n}e^{-\lambda_{n}s}$ be $(\lambda,k)$-Riesz summable at $s_{0}=\sigma_{0}+i\tau_{0}$ for some $k\ge 0$. If the limit function $f$ of $D$ vanishes on some vertical line $[\re=\sigma]$, where $\sigma> \sigma_{0}$, then $a_{n}=0$ for all $n$.
\end{Coro}
\begin{proof}
Let first $\lambda_{1}<x<\lambda_{2}$. Then by Theorem~\ref{Perron0} and our assumption
$$a_{1}e^{-\lambda_{1}s_{0}}(1-\frac{\lambda_{1}}{x})^{k}= S_x^{\lambda,k}(D)(0) = \frac{\Gamma(1+k)}{2\pi i}\int_{\sigma -i\infty}^{\sigma +  i\infty} \frac{f(s)e^{xs}}{s^{1+k}} ds=0,$$
which implies $a_{1}=0$. Proceeding successively by considering a sequence $(x_{n})$ with $\lambda_{n}<x_{n}<\lambda_{n+1}$, we obtain that $a_{n}=0$ for all $n$.
\end{proof}

\bigskip

\subsection{Bohr-Cahen formula}
The following formula for the abscissa of convergence  is again due to Hardy and Riesz \cite[Theorem~31]{HardyRiesz}.

\begin{Theo} \label{Bohr-Cahen-Riesz} Let $D=\sum a_{n}e^{-\lambda_{n}s}$ and $k\ge 0$. Then
$$\sigma_{c}^{\lambda,k}(D)\le \limsup_{x\to \infty} \frac{\log( |R_{x}^{\lambda,k}(D)(0)|)}{x},$$
with equality whenever $\sigma_{c}^{\lambda,k}(D)$ is non-negative.
\end{Theo}

\smallskip

The proof will be given in  Section~\ref{appendix}, mainly as a consequence of  Theorem \ref{corona2}.
Note  that $$\limsup_{x\to \infty} \frac{\log( |R_{x}^{\lambda,k}(D)(0)|)}{x}=\limsup_{x\to \infty} \frac{\log( |S_{x}^{\lambda,k}(D)(0)|)}{x}.$$

Recall that  Theorem \ref{Bohr-Cahen-Riesz} for the special case $k=0$ was already presented in Corollary~\ref{Bohr-Cahen}, and that we in  Corollary~\ref{Bohr-Cahen-uniform} formulated  a uniform counterpart.

Again a simple analysis shows that the  proof of Theorem \ref{Bohr-Cahen-Riesz} also works for vector-valued Dirichlet series.
Hence we exactly as in Section~\ref{summationsection0} obtain the following uniform variant of  Theorem~\ref{Bohr-Cahen-Riesz}
(note that for $0<k<1$ a direct, but technically more involved, argument was given in  \cite[Lemma 3.8]{schoolmann2018bohr}).

\smallskip

\begin{Coro} \label{Bohr-Cahen-Riesz-uniform} Let $D=\sum a_{n}e^{-\lambda_{n}s}$ and $k\ge 0$. Then
$$\sigma_{u}^{\lambda,k}(D)\le \limsup_{x\to \infty} \frac{\log( \sup_{t\in \R}|R_{x}^{\lambda,k}(D)(it)|)}{x},$$
with equality whenever $\sigma_{u}^{\lambda,k}(D)$ is non-negative.
\end{Coro}

\bigskip
\section{\bf A new scale of Banach spaces} \label{structuresection}

We start  repeating  some of our basic definitions from the introduction.
Given a frequency $\lambda = (\lambda_n)$, we call a $\lambda$-Dirichlet series $D=\sum a_{n}e^{-\lambda_{n}s}$ a $\lambda$-Riesz germ of the holomorphic function
$f:~[\re > 0] \to \mathbb{C}$,
whenever $\sigma_{c}^{\lambda,m}(D)<\infty$ for some $m \geq 0$ and $f$ is the holomorphic extension
of the limit function of $D$ to all of $[\re >0]$.

\smallskip

Two remarks are in order.

\begin{Rema} \label{assign}
Let $f:~[\re > 0] \to \mathbb{C}$ be a holomorphic function generated by  the  $\lambda$-Riesz germ
  $D=\sum a_{n}e^{-\lambda_{n}s}$.
Then $D$ by Corollary~\ref{normed} is unique, and as a consequence we may assign to every such  $f$ the  unique sequence
$(a_n(f))_{n} = (a_n)_n$,  which we (like for Dirichlet series) again call  the 'sequence of Bohr coefficients of~$f$'.
          \end{Rema}

\smallskip

\begin{Rema}
\label{elambda}
Let $f:~[\re > 0] \to \mathbb{C}$ be holomorphic. Then $f$  has a $\lambda$-Riesz germ if and only if $f$ on some half-plane and for  some $m$ equals the limit function of some $\lambda$-Dirichlet series under $(e^\lambda,m)$-Riesz summation.
This is an immediate consequence of Theorem~\ref{niceproof?}.
        \end{Rema}

  \smallskip

  Moreover, given a holomorphic function  $f:~[\re > 0] \to \mathbb{C}$  generated by  the  $\lambda$-Riesz germ
  $D$, the $x$th Riesz mean of order $k \geq 0$
  of $f $ in $s \in \mathbb{C}$ is given by
  \[
  R_x^{\lambda,k}(f)(s) =\sum_{\lambda < x} a_n(f) e^{-\lambda_n s} \Big(  1- \frac{\lambda_n}{x}  \Big)^k\,,
  \]
and its  summatory function  in $s \in \mathbb{C}$ by
\[
  S_x^{\lambda,k}(f)(s) =\sum_{\lambda < x} a_n(f) e^{-\lambda_n s} (  x-  \lambda_n)^k\,,\,\, x >0\,.
  \]
As explained above the spaces $$H_{\infty,\ell}^\lambda[\re~>~0]\,, \,\,\,\, \ell \ge0,$$ accumulate  all
holomorphic functions $f:~[\re > 0] \to \mathbb{C}$, which are generated by a $\lambda$-Riesz germ and satisfy
the growth condition
\begin{equation}\label{norway}
  \|f\|_{\infty,\ell} = \sup_{\re s > 0}
\frac{|f(s)|}{(1+|s|)^\ell}< \infty\,.
\end{equation}
For each such $f$
\begin{equation*} \label{trivial}
\sup_{\re s >0} \frac{|f(s)|}{(1+|s|)^\ell}
\leq
 \sup_{\re s >0}\frac{|f(s)|}{|(1+s)^\ell|}
\leq
2^{\ell} \sup_{\re s >0} \frac{|f(s)|}{(1+|s|)^\ell}\,,
\end{equation*}
since $|1+s| \leq 1 + |s|  \leq   2|1+s|$ for all $s \in [\re  >0] $.
Obviously, the pairs $$(H_{\infty,\ell}^\lambda[\re >0], \|\cdot\|_{\infty,\ell})$$ form an increasing scale
of normed spaces. That all these spaces in fact are complete, so Banach spaces, is a non-trivial fact which is given in Theorem~\ref{completeness}.

\smallskip

Finally, we remark that a holomorphic function satisfying the growth condition \eqref{norway}, by Remark~\ref{elambda}
(see also \eqref{secondkind}) belongs to  $H_{\infty,\ell}^{\lambda}[\re>0]$ if and only
if there is a $\lambda$-Dirichlet series $D=\sum a_{n}e^{-\lambda_{n}s}$ such that for some $m,\sigma\ge 0$ and for every $s \in [\re > \sigma]$
\begin{equation} \label{alternativedefintro}
f(s)=\lim_{x\to \infty} R_{x}^{e^{\lambda},m}(D)
=\lim_{x\to \infty} T_{x}^{\lambda,m}(D)
=
\lim_{x\to \infty}\sum_{\lambda_{n}<x} a_{n}e^{-\lambda_{n}s}\Big(1- \frac{e^{\lambda_n}}{e^x} \Big)^m.
\end{equation}

\bigskip

\subsection{Ordinary and power case}

We start considering  the power case $\lambda = (n)$. Then, as already announced in the introduction, the scale $H_{\infty,\ell}^{(n)}[\re>0]\,,\, \ell \ge 0,$ collapses to the case $\ell =0 $.

\begin{Prop} \label{ex2} For every $\ell \ge 0$
 \[
 H_{\infty}^{(n)}[\re>0] = H_{\infty,\ell}^{(n)}[\re>0]\,.
 \]
\end{Prop}

\begin{proof}
  Take $f \in H_{\infty,\ell}^{(n)}[\re>0]$,  let $D = \sum a_n e^{-ns}$ be its associated Dirichlet series, and
 choose any $m \in \mathbb{N}$ such that $\ell < m$. Then we deduce from  Corollary~\ref{Theo41reproved} (proved below) and Theorem~\ref{niceproof?} (for which we give a complete proof in  in our appendix from
    Section~\ref{appendix}) that $D$ is $((e^n),m)$-Riesz summmable on $[\re >0]$ with limit $f$. Then by \cite[Theorem 21]{HardyRiesz}
\begin{equation*} \label{fromfirsttopartialsum}
\lim_{N\to \infty} \Big(\frac{e^{N+1}-e^{N}}{e^{N+1}}\Big)^{m} \Big(\sum_{n=1}^{N}a_n- \sum_{n=1}^{\infty}a_n\Big)=0\,,
\end{equation*}
    i.e.
  $D$  on $[\re >0]$ converges pointwise to $f$.

  We remark that \cite[Theorem 21]{HardyRiesz} is a result on arbitray frequencies $\lambda$ and arbitrary orders $k >0$, and we here only use this result for the integer case $k=m\in \mathbb{N}$, which according to  \cite[p.36]{HardyRiesz} has an 'extremely simple' proof.

  Continuing  the proof, we see that  $D$ converges uniformly on $[\re > \varepsilon]$ for every  $\varepsilon >0$ (indeed, for the frequency $\lambda=(n)$ the abscissas of convergence and absolutely convergence coincide -- see \cite[12, §3, Hilfssatz 2,3]{Bohr2}).
  So  in particular $f$ is bounded on $[\re > 1]$. Moreover, $f(\sigma + i \pmb{\cdot}): \R \to \C$ is $2\pi$-periodic for every $\sigma >0$,
  which  implies that
  \begin{align*}
    \sup_{\substack{0 < \sigma < 2\\ t \in \R}} |f(\sigma + it)| &
        =\sup_{\substack{0 < \sigma < 2\\ t \in [0,2\pi]}} |f(\sigma + it)|
        \\&
    \leq \|f\|_{\infty,\ell} \sup_{\substack{0 < \sigma < 2\\ t \in [0,2\pi]}} (1 + |\sigma + it|)^\ell
           \leq \|f\|_{\infty,\ell}  (1 + \sqrt{4 + 4\pi^2})^\ell\,.
      \end{align*}
      All in all we conclude that $f \in H_{\infty}^{(n)}[\re>0]$\,.
        \end{proof}

\smallskip

Next we  consider the frequency $\lambda=(\log n)$, which generates ordinary Dirichlet series $\sum a_n n^{-s}$.
Looking at  \eqref{alternativedefintro} (with $x=\log y$), we see that a holomorphic function $f\colon [\re>0] \to \C$ belongs to the Banach space $\mathcal{H}^{(\log n)}_{\infty,1}[\re>0]$ if only if the growth condition \eqref{norway} holds,
 and there exists an ordinary Dirichlet series $D=\sum a_{n} n^{-s}$ such that for some $m, \sigma\ge 0$
$$f(s)=\lim_{y\to \infty} \sum_{n<y} a_{n}n^{-s} \Big(1-\frac{n}{y}\Big)^{m}, ~~ s\in [\re>\sigma].$$
Actually, we will see in Corollary~\ref{41improved} that  Theorem \ref{Theo41} in the special case $\lambda=(\log n)$ is valid for $k=\ell$, and therefore in the particular case $\ell=1$
$$f(s)=\lim_{N\to \infty} \frac{1}{N} \sum_{n=1}^{N} \sum_{k=1}^{n} a_{k}k^{-s}, ~~ s\in [\re>0],$$
so that eventually we arrive at the following result.

\begin{Prop} \label{ex1}Let $f:[\re >0] \to \C$ be holomorphic. Then the growth condition \eqref{norway} holds
 and there exists an  ordinary Dirichlet series which on $[\re >0]$ is  Ces\`{a}ro summable with limit $f,$ if and only if $f \in H_{\infty,1}^{(\log n)}[\re >0]$.
\end{Prop}

Note that the space $H^{(\log n)}_{\infty,\ell}[\re>0]$ is strictly larger than $H^{(\log n)}_{\infty}[\re>0]$. To see an example, we again
look at the eta-function $\eta$ (defined in the introduction), for which
\begin{equation*} \label{ordereta}
\ell > \dfrac{1}{2}
\,\,\,\,\, \Rightarrow \,\,\,\,\,
\eta \in H^{(\log n)}_{\infty, \ell}[\re > 0]
\,\,\,\,\, \Rightarrow \,\,\,\,\,
\ell \ge \dfrac{1}{2}\,.
\end{equation*}
In particular, $\eta\notin H^{(\log n)}_{\infty}[\re > 0]$. Alternatively, this may also be seen as a consequence of   Bohr's inequality (see \cite[Corollary 4.3]{defant2018Dirichlet} or \cite[Theorem 4.4.1]{queffelec2013diophantine}): Assuming $\eta \in \mathcal{H}_{\infty}((\log n))$ implies that
$N= \sum_{n=1}^{N} |(-1)^{p_{n}}| \le \|D\|_{\infty}<\infty$ for all~$N$,
which clearly is a contradiction.

This example in particular  shows that, although we in general have
\[
\mathcal{H}^{\lambda}_{\infty,\ell_{1}}[\re>0]\subset \mathcal{H}^{\lambda}_{\infty,\ell_{2}}[\re>0]
\,\,\, \text{ for all $0\le \ell_{1}<\ell_{2}$ }\,,
\]
this  inclusion for certain $\lambda$'s may be strict  whereas for other $\lambda$'s it actually may be an equality (of sets).

\bigskip

\subsection{Perron formula -- a variant}

We now  for the new scale of Banach spaces prove an important  variant of Perron's formula from Theorem \ref{Perron0}.

\begin{Theo} \label{new-ban}
 Let  $f \in H^{\lambda}_{\infty,\ell}[\re >0]$ and $k >\ell\ge 0$.
 Then for all $s_{0}\in [\re\ge0]$, $x>0$ and $c >0$
 \begin{equation*}
 S_{x}^{\lambda,k}(f)(s_{0}) = \frac{\Gamma(1+k)}{2\pi i} \int_{c-i\infty}^{c+i\infty} \frac{f(s+s_{0})}{s^{1+k}} e^{xs} ds.
 \end{equation*}
\end{Theo}

\begin{proof} After translation we may assume that $s_{0}=0.$ By assumption the Dirichlet series $D=\sum a_{n}e^{-\lambda_{n}s}$ is $(\lambda,m)$-Riesz summable at some $s_{1}\in [\re>0]$ for some $m\ge 0$. Hence $D$ is $(\lambda,m+k)$-Riesz summable at $s_{1}$
(Theorem~\ref{basic},(i)), and for every $\varepsilon,\sigma>0$ and $|t|\ge t_{0}>0$
$$\frac{|f(\sigma+it)|}{|\sigma+it|^{1+m+k}}\le \|f\|_{\infty,\ell} |\sigma+it|^{-(1+m+k-\ell)}\le t_{0}^{-(1+m+k-\ell)} \|f\|_{\infty,\ell}\le \varepsilon\,,$$
whenever  $t_{0}$ is large enough. Corollary \ref{werdervsgladbach} then implies that for all $c,x>0$
$$\sum_{\lambda_{n}<x} a_{n}(x-\lambda_{n})^{m+k}= \frac{\Gamma(1+m+k)}{2\pi i}  \int_{c -i\infty}^{c +  i\infty} \frac{f(s)e^{xs}}{s^{1+m+k}} ds.$$
Now we differentiate $m$-times obtaining
$$\frac{d^{m}}{dx}\int_{c -i\infty}^{c +  i\infty} \frac{f(s)e^{xs}}{s^{1+k+m}} ds=  \int_{c -i\infty}^{c +  i\infty} \frac{f(s)e^{xs}}{s^{1+k}} ds$$
and
$$\frac{d^{m}}{dx}\sum_{\lambda_{n}<x} a_{n}(x-\lambda_{n})^{m+k}=\prod_{j=0}^{m-1}(m+k-j)\sum_{\lambda_{n}<x} a_{n}(x-\lambda_{n})^{k}\,;$$
for the second equality  see \eqref{B}, and for the first  note that the differentiation under the integral is legit, since for every $j=0,\ldots, m-1$
\begin{equation*}
\int_{c -i\infty}^{c +  i\infty} \Big|\frac{f(s)e^{xs}}{s^{1+k+j}}\Big| ds\le e^{xc}\|f\|_{\infty,\ell} \int_{c-i\infty}^{c+i\infty} |s|^{-(1+k+j-\ell)} ds<\infty.
\end{equation*}
Additionally by the functional equation of the Gamma function we have
$$\Gamma(1+k)\Gamma(1+m+k)=\prod_{j=0}^{m-1}(m+k-j),$$
so that altogether we obtain
\begin{equation*}
S_{x}^{\lambda,k}(f)(0)=\sum_{\lambda_{n}<x} a_{n}(x-\lambda_{n})^{k}= \frac{\Gamma(1+k)}{2\pi i}\int_{c -i\infty}^{c +  i\infty} \frac{f(s)e^{xs}}{s^{1+k}} ds. \qedhere
\end{equation*}
\end{proof}

\bigskip

\subsection{A maximal inequality for Riesz means}

The following maximal inequality  is going to be our major tool to perform a substantial
structure theory of the spaces $H_{\infty,\ell}^{\lambda}[\re>0]$. It shows that for $k > \ell$ the linear projections
\[
R_x^{\lambda,k}:  H_{\infty,\ell}^\lambda[\re >0] \to H_{\infty,\ell}^\lambda[\re >0]\,, \,\,\,\,f \to  R_x^{\lambda,k}f\,, \,\,\,\, x>0,
\]
are uniformly bounded.

\begin{Theo} \label{T1}
Let $k>\ell\ge 0$ and $\lambda$ a frequency. Then  there is a constant
$C= C(k,\ell, \lambda) >0$ such that for all
$f \in H_{\infty,\ell}^\lambda[\re >0]$
\[
\sup_{x>0} \|R_x^{\lambda,k}(f)\|_{\infty,\ell}
\leq
C \|f\|_{\infty,\ell}\,.
\]
In particular,
for every $\sigma \ge 0$ there is a constant  $C= C(\ell,k,\lambda, \sigma)>0$ such that for every $f \in H_{\infty,\ell}^\lambda[\re >0]$
and $t \in \mathbb{R}$
\[
\sup_{x>0}|R_x^{\lambda,k}(f)(\sigma+it)| \leq C (1+|t|)^\ell \|f\|_{\infty,\ell}\,.
\]
\end{Theo}

\begin{proof} We assume first that $\lambda_{1}>0$, and take $s = \sigma + it \in [\re  >0]$ and $y \ge \lambda_1$. Then by Theorem~\ref{new-ban} with $c=y^{-1}$
\begin{align*}
|R_{y}^{\lambda,k}(D)(s)|\leq  \frac{e \Gamma(1+k)}{2\pi }y^{-k} \int_{\mathbb{R}} \frac{|f( s +\frac{1}{y}
+ix)|}{|\frac{1}{y}
+ix|^{1+k}} dx \,,
\end{align*}
  and
\begin{align*}
&\|f\|_{\infty,\ell}^{-1}y^{-k} \int_{\mathbb{R}} \frac{|f( s +\frac{1}{y}
+ix)|}{|\frac{1}{y}
+ix|^{1+k}} dx
\le
 \int_{\mathbb{R}} \frac{y|1 +s+\frac{1}{y}+ix|^{\ell}}{|1+ixy|^{1+k}} dx
\\ &
\le
 2^{\max\{0,\ell-1\}}  \int_{\mathbb{R}} \frac{y(|1 + s|^{\ell}+|\frac{1}{y}+ix|^{\ell})}{|1+ixy|^{1+k}} dx
\\ &
\le
 2^{\max\{0,\ell-1\}}\Big((1+|s|)^\ell \int_{\mathbb{R}} \frac{y}{|1+ixy|^{1+k}} dx+ y^{-\ell} \int_{\mathbb{R}} \frac{y}{|1+ixy|^{1+k-\ell}} dx\Big)
\\ &
\le
 2^{\max\{0,\ell-1\}}\Big((1+|s|)^\ell \int_{\mathbb{R}} \frac{1}{|1+iv|^{1+k}} dv+ \lambda_1^{-\ell} \int_{\mathbb{R}} \frac{1}{|1+iv|^{1+k-\ell}} dv\Big)
 \\ &
\le
 2^{\max\{0,\ell-1\}}\Big((1+|s|)^\ell \big(1+\frac{1}{k}\big)
 + \lambda_1^{-\ell} \big(1+\frac{1}{k-\ell}\big) \Big)
 \\ &
 \le
 2^{\max\{0,\ell-1\}}\Big(\big(1+\frac{1}{k}\big)
 + \lambda_1^{-\ell} \big(1+\frac{1}{k-\ell}\big) \Big) (1+|s|)^\ell \,.
 \end{align*}
  Assume now that $\lambda_{1}=0$ and choose $0<y<\lambda_{2}$. Then by the previous calculation
  $$|a_{0}|\le \|f\|_{\infty,\ell}2^{\max\{0,\ell-1\}}\Big(\big(1+\frac{1}{k}\big)
 + y^{-\ell} \big(1+\frac{1}{k-\ell}\big) \Big) (1+|s|)^\ell.$$
 Tending $y \to \lambda_{2}$ we obtain
 $$|a_{0}|\le C(\lambda,k,\ell)(1+|s|)^\ell\|f\|_{\infty,\ell}.$$
 The 'in particular part' is then a straight forward consequence, since
  \[
  1+|s| \leq  (1+\sigma)\big( 1+ \frac{|t|}{1+\sigma}\big) \leq (1+\sigma)(1+|t|) \,.\qedhere
  \]
\end{proof}

\subsection{Uniform Riesz approximation}

Next we give the main application of the preceding maximal theorem.  We prove that for $k > \ell$ all functions in a  bounded subset of $H_{\infty,\ell}^{\lambda}[\re>0]$, after translation by some arbitrary $u > 0$,
are uniformly $(\lambda,k)$-Riesz approximable in the norm of $H_{\infty,\ell}^{\lambda}[\re>0]$.

\smallskip

\begin{Theo} \label{Perronestimate}
Let $\lambda$ be a frequency and $k>\ell\ge 0$. Then for every bounded set $M\subset H_{\infty,\ell}^{\lambda}[\re>0]$ and every choice of $\varepsilon, u >0$ there is $x_{0}>0$ such that
$$\sup_{x>x_{0}}\sup_{f\in M} \|f(u + \pmb{\cdot})-R_{x}^{\lambda,k}(f)(u + \pmb{\cdot})\|_{\infty,\ell}\le \varepsilon.$$
\end{Theo}
\begin{proof} Let $m\in N_{0}$ such that $m< k \le m+1$, and $C = \sup_{f\in M} \|f\|_{\infty,\ell} < \infty$. Fix $f\in M$ and $u>0$. Then by Lemma  \ref{2} and Theorem \ref{T1} we for every $s\in [\re>0]$ have
\begin{align*}
&\Big|\frac{x^{-k}}{\Gamma(m+2)}\int_{0}^{x}\frac{S_{t}^{\lambda,m+1}(f)(s)}{(1+|s|)^{\ell}} u^{m+2}(x-t)^{k} e^{-ut} dt- \frac{R_{x}^{\lambda,k}(f)(s+u)}{(1+|s|)^{\ell}}\Big| \\ &\le e^{-u x} \frac{|R_{x}^{\lambda,k}(f)(s)|}{(1+|s|)^{\ell}} +C_{1}(m,k) x^{-k} \int_{0}^{x} \frac{|S_{y}^{\lambda,k}(f)(s)|}{(1+|s|)^{\ell}} \sum_{j=1}^{m+1} u^{j}e^{-u y} (x-y
)^{j-1}  dy \\ &\le e^{-ux} C+C_{2}(m,k,u)  C \sum_{j=1}^{m+1} x^{-k} \int_{0}^{x}y^{k}(x-y)^{j-1} e^{-uy} dy,
\end{align*}
where
\begin{align*}
&x^{-k} \int_{0}^{x}y^{k}(x-y)^{j-1} e^{-uy} dy=x^{j} \int_{0}^{x} x^{-1} (\frac{y}{x})^{k} (1-\frac{y}{x})^{j-1} e^{-uy} dy\\ &= x^{j} e^{-ux} \int_{0}^{1} \alpha^{k} (1-\alpha)^{j-1}  d\alpha = x^{j} e^{-ux} \frac{\Gamma(k+1)\Gamma(j)}{\Gamma(k+1-j)}.
\end{align*}
Hence for $x\ge 1$ and $f\in M$
\begin{align*}
&
\sup_{s\in [\re>0]}\Big|\frac{x^{-k}}{\Gamma(m+2)}\int_{0}^{x}\frac{S_{t}^{\lambda,m+1}(f)(s)}{(1+|s|)^{\ell}} u^{m+2}(x-t)^{k} e^{-ut} dt- \frac{R_{x}^{\lambda,k}(D)(s+u)}{(1+|s|)^{\ell}}\Big| \\ &\le Ce^{-ux} (1+C_{3}(m,k,u)x^{m+1}),
\end{align*}
which vanishes uniformly for all $f\in M$ as $x\to \infty$. Additionally, following the proof of Lemma~\ref{lemma-limit} and Lemma~\ref{heute}
the limit
\begin{align*}
\lim_{x\to \infty} \frac{x^{-k}}{\Gamma(m+2)}\int_{0}^{x}&\frac{S_{t}^{\lambda,m+1}(f)(s)}{(1+|s|)^{\ell}} u^{m+2}(x-t)^{k} e^{-ut} dt\\&= \frac{1}{\Gamma(1+k)}\int_{0}^{\infty} \frac{S_{t}^{\lambda,k}(f)(s)}{(1+|s|)^{\ell}} u^{k+1}  e^{-u t} dt=\frac{f(s+u)}{(1+|s|)^{\ell}},
\end{align*}
exists uniformly on $[\re>0]$ and $f\in M$, where the last equality holds by Theorem~\ref{corona2} (or more precisely Remark~\ref{together}). This proves the claim.
\end{proof}

\smallskip

The following corollary is an immediate consequence considering singleton sets.

\smallskip

\begin{Coro} \label{41}
Let $f\in H_{\infty,\ell}^\lambda[\re >0]$ and $k >\ell \ge 0$. Then
for every $u >0$
\[
\text{$
 \lim_{x\to \infty} R_x^{\lambda,k}(f)
 (u + \pmb{\cdot}) = f(u+\pmb{\cdot})$ \,\,\,in $H_{\infty,\ell}^\lambda[\re >0]$}\,.
\]
In particular, for every $u>0$
$$\lim_{x\to \infty} \sup_{s\in [\re>u]} \Big|\frac{f(s)-R_{x}^{\lambda,k}(f)(s)}{(1+s)^{\ell}}\Big|=0.$$
\end{Coro}

\smallskip

\noindent
 To check the 'in particular' part fix $u>0$ and let $s=w+u$, where $w\in [\re>0]$. Then for large $x$ by the first part
\begin{align*}
&
\frac{|f(s)-R_{x}^{\lambda,k}(f)(s)|}{|1+s|^{\ell}}
\\&
=\frac{|f(u+w)-R_{x}^{\lambda,k}(f)(u+w)|}{|1+w|^{\ell}} \frac{|1+w|^{\ell}}{|1+u+w|^{\ell}}
\le \varepsilon \frac{|1+u+w|^{\ell}}{|1+u+w|^{\ell}}=\varepsilon. \qedhere
\end{align*}

\bigskip

\subsection{Theorem 41 of M. Riesz revisted}

Corollary~\ref{41}  immediately implies the following reformulation
of an important result of M. Riesz published in \cite[Theorem 41]{HardyRiesz}. Actually, this result  motivated much of the content of this article.

\begin{Coro} \label{Theo41reproved} If $f\in H_{\infty,\ell}^{\lambda}[\re>0]$, then for every $k>\ell\ge 0$ and $s\in [\re>0]$
\begin{equation*}
f(s)=\lim_{x\to \infty} R_{x}^{\lambda,k}(f)(s).
\end{equation*}
\end{Coro}

\smallskip

Let us point out that the mentioned  result of Riesz actually is formulated  in terms of $(e^{\lambda},k)$-Riesz
summability (recall the definition of $(\lambda,k)$-Riesz means of second kind from \eqref{secondkind}). In its original form as given in \cite{HardyRiesz} it reads as follows (using our terminology of having a Riesz germ).

\smallskip

\begin{Theo}\label{Theo41} Let $f\colon [\re>0] \to \C$ be holomorphic with a $\lambda$-Riesz germ $D=\sum a_{n}e^{-\lambda_{n}s}$.
Assume that there is $\ell\ge 0$ such that
\begin{equation} \label{conditionintroA}
\forall~ \varepsilon>0~ \exists~ C>0\colon |f(s)|\le C|s|^{\ell}, ~~ s\in [\re>\varepsilon].
\end{equation}
Then for every $k>\ell$ and $s\in [\re>0]$
$$f(s)=\lim_{x\to \infty}T_{x}^{\lambda,k}(D)(s) = \lim_{x\to \infty} \sum_{\lambda_{n}<x} a_{n} \big(1- e^{\lambda_{n}-x})^{k} e^{-\lambda_{n}s}.$$
\end{Theo}

\smallskip
Let us explain why Theorem~\ref{Theo41} and Corollary~\ref{Theo41reproved} in fact are equivalent. Recall first, that by  Theorem~\ref{niceproof?}
 every $\lambda$-Dirichlet series is $(\lambda,k$)-Riesz summable
on all half-planes $[\re > \sigma], \, \sigma >0$
if and only if this series is $(e^{\lambda},k)$-Riesz summable on all such half-planes.
 So it is obvious that Theorem \ref{Theo41} is applicable to every  $f \in H_{\infty,\ell}^{\lambda}[\re>0]$;
 indeed, given $\varepsilon >0$, we  have
 \[
 \sup_{\re s >\varepsilon}\frac{|f(s)|}{|s|^\ell} \leq \|f\|_{\infty, \ell} \sup_{\re s >\varepsilon}\frac{(1+|s|)^\ell}{|s|^\ell}
 < \infty\,.
 \]
Conversely, if a holomorphic function $f$ on $[\re>0]$ has a $\lambda$-Riesz germ and satisfies \eqref{conditionintroA}, then $f_{\varepsilon}=f(\varepsilon+ \cdot) \in H_{\infty,\ell}^{\lambda}[\re>0]$ for every~$\varepsilon>0$,
and consequently we may apply  Corollary~\ref{Theo41reproved}.

\smallskip

The decision on why we focus on $(\lambda,k)$-Riesz means instead of $(e^{\lambda},k)$-Riesz means is based on the integral representation
\begin{equation*}\label{perronfirsttypeintro}
R_{x}^{\lambda,k}(f)(s_{0})=\frac{\Gamma(1+k)}{2\pi i} x^{-k} \int_{c-i\infty}^{c+i\infty}  f(s+s_{0}) \frac{e^{xs}}{s^{1+k}}  ds
\end{equation*}from Theorem~\ref{Perron0}
of the $x$th Riesz means of first kind  of a Dirichlet series $D=\sum a_{n}e^{-\lambda_{n}s}$ in terms of its limit function~$f$\,,
whereas by \cite[Theorem 40, p. 51]{HardyRiesz} for Riesz means of second kind we have
\begin{equation*} \label{perronfirsttypeintro2}
T_{x}^{\lambda,k}(D)(s_0)=\frac{\Gamma(1+k)}{2\pi i} \int_{c-i\infty}^{c+i\infty}f(s+s_{0}) \frac{\Gamma(s)e^{xs}}{\Gamma(1+k+s)} ds.
\end{equation*}
In fact, we believe that the  handling of the kernels $ e^{xs}/s^{1+k}$  is  more  simple than that of the kernels  $\Gamma(s)e^{xs}/\Gamma(1+k+s)$.

\bigskip

\subsection{The bounded case $\ell =0$}

For $\ell=0$ and $f\in H_{\infty,0}^{\lambda}[\re>0]$, Corollary \ref{41} shows that for every $\sigma>0$
\begin{equation*}
f=\lim_{x\to \infty} R_{x}^{\lambda,k}(f) \,\,\,\, \text{ uniformly on $[\re>\sigma]$ }\,.
\end{equation*}
Hence, as the uniform limit of  uniformly almost periodic function on $\mathbb{R}$ is again uniformly almost periodic,  the function $t\mapsto f(\sigma+it)$ is uniformly almost periodic  for every $\sigma>0$, this has the following consequence.

\begin{Coro}\label{oldcase} For every frequency $\lambda$
$$H_{\infty,0}^{\lambda}[\re>0]=H_{\infty}^{\lambda}[\re>0].$$
\end{Coro}

\bigskip

\subsection{Boundedness of coefficient functionals}

We explained above that every function in $H_{\infty,\ell}^{\lambda}[\re>0]\to \mathbb{C}$ is uniquely determined by its
Bohr coefficients $a_N(f), N~\in~\mathbb{N}$ (Remark~\ref{assign}). The following result shows that this assignment  in fact is continuous.

\smallskip
\begin{Theo}\label{boundednesscoeff} Let $k>\ell \ge 0$ and $\lambda=(\lambda_{n})$ a frequency. Then for each  $N\in \N$  the coefficient mapping
$$\mathcal{C}_{N}\colon H_{\infty,\ell}^{\lambda}[\re>0]\to \mathbb{C}, ~~ f \mapsto a_{N}(f)$$
is bounded. More precisely, with $\lambda_{-1}:=0$
$$\|\mathcal{C}_N\|\le C(k, \ell, \lambda_1)\left(\big(\frac{\lambda_{N+1}}{\lambda_{N+1}-\lambda_{N}}\big)^{k}+\big(\frac{\lambda_{N}}{\lambda_{N}-\lambda_{N-1}}\big)^{k} \right) \,, ~~ N\in \N.$$
\end{Theo}

\smallskip

Since $a_{N}(f)=\sum_{n=1}^{N}a_{n}(f)-\sum_{n=1}^{N-1} a_{n}(f)$, the proof follows from the following independently interesting lemma.

\smallskip

\begin{Lemm} \label{coeff}
Let $k>\ell \ge 0$. Then for every $f \in    H_{\infty,\ell}^{\lambda}[\re>0]$ and $M\in \N$
\begin{equation*} \label{zero to k}
\Big|\sum_{n=1}^{M} a_{n}(f)\Big|\le C(k, \ell, \lambda) \Big(\frac{\lambda_{M+1}}{\lambda_{M+1}-\lambda_{M}}\Big)^{k}
 \,
\|f\|_{\infty,\ell}\,.
\end{equation*}
\end{Lemm}
\begin{proof}
From \cite[Lemma 3.5]{schoolmann2018bohr} we know that for every $M\in \N$
$$\Big|\sum_{n=1}^{M} a_{\lambda_n}(f)\Big|\le 3 \big(\frac{\lambda_{M+1}}{\lambda_{M+1}-\lambda_{M}}\big)^{k} \sup_{\lambda_{1}<x< \lambda_{M+1}} \Big| \sum_{\lambda_{n}<x}a_{\lambda_n}(f) (1-\frac{\lambda_{n}}{x})^{k}\Big|.$$
Then the conclusion follows from  Theorem~\ref{T1}.
\end{proof}

\bigskip

\subsection{A Montel theorem} \label{monti}

According to \cite[Theorem~3.4]{DeFeScSe2020}, we for any frequency $\lambda$ have  a Montel-type theorem in $H_{\infty}^{\lambda}[\re>0]$.
The aim now is to extend this result  from  $H_{\infty}^{\lambda}[\re>0]$ to $H_{\infty,\ell}^{\lambda}[\re>0],\, \ell >0$.

Recall from \eqref{periodic-coin} that, given a frequency $\lambda$,
there is an isometric and  coefficient preserving identity $H_{\infty}^{\lambda}[\re>0] = \mathcal{H}_\infty(\lambda)$, so Montel's theorem
holds in $\mathcal{H}_\infty(\lambda)$, i.e.
for each bounded sequence $(D_N)$ in
$\mathcal{H}_\infty(\lambda)$  there is a subsequence $(D_{N_k})$ and $D \in \mathcal{H}_\infty(\lambda)$
such that for every $\varepsilon >0$ the sequence of all translations
$D^\varepsilon_{N_k} = \sum a_n e^{-\lambda_{n}\varepsilon} e^{-\lambda_{n}s}$ converges to $D\in \mathcal{H}_\infty(\lambda)$.

The important special case $\lambda= (\log n)$ was  first proved by  Bayart in \cite{Ba02}, and within the setting of Hardy space of general Dirichlet series it was  substantially extended  in \cite{defantschoolmann2019Hptheory},~\cite{defant2020riesz}, and~\cite{DeFeScSe2020}.
See also the discussion in Section~\ref{SectionCoincidence}.

\smallskip

\begin{Theo} \label{Montel}
Let $\lambda$ be an arbitrary  frequency, and  $(f_N)$ a bounded sequence in $H_{\infty,\ell}^\lambda~[\re>0], \ell \ge 0$. Then there is a subsequence
$(f_{N_j})_j$ of $(f_N)$ and $f \in H_{\infty,\ell}^\lambda[\re >0]$ such that
for every $ u>0$
\[
\text{$f_{N_j} (u + \pmb{\cdot}) \to f( u+\pmb{\cdot})$ \,\,\,\,in \,\,$H_{\infty,\ell}^\lambda[\re >0]$}\,.
\]
\end{Theo}

\smallskip

In fact we are going to deduce this result from the following more general version.

\begin{Theo}  \label{generalMontel}
Let $(f_N)$ be a bounded sequence in $H_{\infty,\ell}^\lambda[\re >0]$. Assume that for every $n$ the limit
$$a_{n}=\lim_{N\to \infty} a_{n}(f_{N})$$
exists. Then there exists  $f\in  H_{\infty,\ell}^\lambda[\re >0]$ such that $a_n(f) = a_n$ for all $n$,
and for every $u>0$
\[
\text{$f_{N} (u + \pmb{\cdot}) \to f(u+\pmb{\cdot})$ \,\,\,in\,\, $H_{\infty,\ell}^\lambda[\re >0]$}\,.
\]
\end{Theo}
\begin{proof} We define formally the Dirichlet series $E=\sum a_{n}e^{-\lambda_{n}s}$. Then Theorem~\ref{T1} implies that
\begin{equation}\label{TT11}
  |R_{x}^{\lambda,k}(E)(0)|=\lim_{N\to \infty} |R_{x}^{\lambda,k}(f_{N})(0)|\le C \sup_{N\in \N} \|f_{N}\|_{\infty,\ell}.
\end{equation}
Hence, applying Theorem \ref{Bohr-Cahen-Riesz}, we see that $E$ is $(\lambda,k)$-Riesz summable on $[\re>0]$. Moreover, denoting by $f$ its limit function, we by \eqref{TT11} have
$$\frac{|f(s)|}{(1+|s|)^{\ell}}=\lim_{x \to \infty} \lim_{N \to \infty}\frac{|R_{x}^{\lambda,k}(f_{N})(s)|}{(1+|s|)^{\ell}}\le C \sup_{N\in \N} \|f_{N}\|_{\infty,\ell}\,,$$
and so $f\in H_{\infty,\ell}^{\lambda}[\re>0]$. It remains to show that $f_{N} (u + \pmb{\cdot}) \to f(u+\pmb{\cdot})$ in $H_{\infty,\ell}^\lambda[\re >0]$. Fixing $u>0$ and  applying Theorem \ref{Perronestimate} to the bounded set $M=\{f_{N} \mid N\in \N\}\cup\{f\}$, we for every $\varepsilon>0$ obtain some $x>0$ such that
\begin{align*}
&\|f(u+\cdot)-f_{N}(u+\cdot)\|_{\infty,\ell}\le \|f(u+\cdot)-R_{x}^{\lambda,k}(f)(u+\cdot)\|_{\infty,\ell}\\&+\|R_{x}^{\lambda,k}(f-f_{N})(u+\cdot)\|_{\infty,\ell}+\|f(u+\cdot)-R_{x}^{\lambda,k}(f_{N})(u+\cdot)\|_{\infty,\ell}\\&\le 2\varepsilon+\|R_{x}^{\lambda,k}(f-f_{N})(u+\cdot)\|_{\infty,\ell}.
\end{align*}
Now tending $N\to \infty$ gives the claim.
\end{proof}

\begin{proof}[Proof of Theorem~\ref{Montel}] From Theorem \ref{boundednesscoeff} we know that $|a_{n}(g)|\le C(n) \|g\|_{\infty,\ell}$ for every $g\in H_{\infty,\ell}^{\lambda}[\re>0]$. This allows to apply a diagonal process to find a subsequence $(N_{j})$ such that  the limit
$$a_{n}=\lim_{j\to \infty} a_{n}(f_{N_{j}})$$
exists for every $n$. Then Theorem \ref{generalMontel} finishes the proof.
\end{proof}

\bigskip

\subsection{Completeness}
As we have already announced several times, all spaces $H_{\infty,\ell}^{\lambda}[\re>0], \, \ell \ge 0$ are  indeed  Banach spaces, which after all seems non-trivial.

\begin{Theo} \label{completeness}
For each $\ell \ge 0$
$$H_{\infty,\ell}^{\lambda}[\re>0]$$
together with the norm
 $\|\cdot\|_{\infty,\ell}$
 forms a Banach space.
\end{Theo}
\begin{proof} Given a Cauchy sequence $(f_{N})$, Theorem~\ref{Montel} proves the existence of a subsequence $(N_{j})$ and $f\in H_{\infty,\ell}^{\lambda}[\re>0]$, such that for every $u>0$ we have
$$\lim_{j\to \infty} \|f_{N_{j}}(u+\cdot)-f(u+\cdot)\|_{\infty,\ell}=0,$$
which in particular shows that $(f_{N_{j}})$ converges pointwise to $f$ on $[\re>0]$.
Hence for large $N$ and $j=j(s)$ we obtain
\begin{align*}
&\Big|\frac{f(s)-f_{N}(s)}{(1+s)^{\ell}}\Big|\le \Big|\frac{f(s)-f_{N_{j}}(s)}{(1+s)^{\ell}}\Big|+\|f_{N_{j}}-f_{N}\|_{\infty,\ell}\le |f(s)-f_{N_{j}}(s)|+\varepsilon,
\end{align*}
and so $\|f-f_{N}\|_{\infty,\ell}\le 2\varepsilon$ tending $j\to \infty$.
\end{proof}

\bigskip

\subsection{Behavior far left} \label{far left}
The following result shows that the norm of any $f \in H^{\lambda}_{\infty,\ell}[\re >0]$  determines
the growth of $f$ on vertical abscissas close to the imaginary axis.

\begin{Theo} \label{new}
Let $f \in H^{\lambda}_{\infty,\ell}[\re >0]$ with $\ell \ge 0$.  Then
\[
\|f\|_{\infty,\ell} = \lim_{\sigma \to 0}  \sup_{t \in \R}  \Big| \frac{f(\sigma + i t)}{(1+(\sigma +i t))^\ell} \Big|\,.
\]
\end{Theo}

\smallskip

The proof needs a well-known Phragmen-Lindel\"of type lemma (see e.g. \cite[p.~138]{Be54}).

\smallskip

\begin{Lemm} \label{doetsch}
Let $g$ be a bounded function on the strip $[\sigma_1 \leq \re \leq \sigma_1]$, which is holomorphic on
$[\sigma_1 < \re  < \sigma_1]$. Then the function
\[
L_g(\sigma) = \sup_{t \in \R}  \big| g(\sigma + it)  \big|\,, \,\,\, \sigma \in [\sigma_1, \sigma_2]\,.
\]
is logarithmically convex on $[\sigma_1, \sigma_2]$\,, i.e. $\log L_g(\sigma)$ is convex on $[\sigma_1, \sigma_2]$\,.
\end{Lemm}

\begin{proof}[Proof of Theorem~\ref{new}] Define the function
\[
g(s) = \frac{f(s+\varepsilon)}{(1 +s)^\ell}\,,
\]
and for each $x$
\[
g_{x}(s) = \frac{\sum_{\lambda_{n}<x}a_{n} \big(1-\frac{\lambda_{n}}{x}\big)^{k}e^{-s \lambda_{n}}}{(1 +s)^\ell}=\frac{R_{x}^{\lambda,k}(f)(s)}{(1+s)^{\ell}} \,,
\]
which are both obviously holomorphic and bounded on $[\re > 0]$.
Note now that $L_{g_{x}}(\sigma)$ for each $x>0$  is non-increasing for  $\sigma >0$; indeed,
by  Lemma~\ref{doetsch} we know that $\log L_{g_x}(\sigma)$ is convex and  moreover $\log L_{g_x}(\infty) = -\infty$.
But since Corollary~\ref{41} implies that for every $u >0$
\[
g_{x} \to g \,\,\, \text{ uniformly on  $[\re >  u]$ as $x\to \infty$}\,,
\]
we see that $L_{g}(\sigma)$ is non-increasing for $\sigma >0$. This   gives
\[
\|f\|_{\infty,\ell}
=
\sup_{\sigma >0}
   \big| L_{g_{x}}(\sigma) \big|=\lim_{\sigma \to 0}
   \big| L_{g_{x}}(\sigma)\big|\,,
\]
which is exactly what we were looking for.
\end{proof}

\bigskip

\subsection{Behavior far  right}

In contrast to what we did in Section~\ref{far left} we now prove in which sense the growth of functions $f\in H_{\infty,\ell}^\lambda [\re >0]$
subsides on abscissas far right.

\smallskip

\begin{Theo} \label{Theo2} Let $\ell \ge 0$ and $\lambda=(\lambda_{n})$ a frequency with $\lambda_1 >0$. Then for every $f\in H_{\infty,\ell}^\lambda [\re >0]$ we have
\begin{equation*}
\lim_{ \sigma \to \infty} \sup_{s\in [\re>\sigma]}\frac{|f(s)|}{(1+|s|)^{\ell}}
=\lim_{ \sigma \to \infty} \sup_{\substack{\delta\ge \sigma\\ t \in \R}}\frac{|f(\delta + it)|}{(1+|t|)^{\ell}}
=0.
\end{equation*}
\end{Theo}

\begin{proof}
Observe, that  it suffices to check the equality for the second limit. Moreover recall that by Corollary \ref{Theo41reproved} the Dirichlet series $D=\sum a_{n}e^{-\lambda_{n}s}$ associated to $f$ is $(\lambda,k)$-Riesz summable on $[\re>0]$ for every $k>\ell$. Hence, fixing $k > \ell$, $t\in \R$ and $\sigma >0$, we apply Theorem~\ref{corona2} to the translation $\sum a_{n}e^{-\lambda_{n}it} e^{-\lambda_{n}s}$ and obtain for $\delta>\sigma$
\begin{equation*}
f(\delta+it)= \frac{\delta^{1+k}}{\Gamma(1+k)}\int_{\lambda_{1}}^{\infty} e^{-\delta y} S_{y}^{\lambda,k}(D)(it) dy.
\end{equation*}
Hence with Theorem \ref{T1}
\begin{align*}
&
\Gamma(1+k)|f(\delta+it)|\leq \delta^{1+k} \int_{\lambda_{1}}^{\infty} e^{-\delta y} y^{k} \Big|R_{y}^{\lambda,k}(D)(it)\Big| dy
\\ &
\le
\|f\|_{\infty,\ell}(1+|t|)^{\ell} \delta^{1+k}\int_{\lambda_{1}}^{\infty}  e^{-\delta y}  y^{k} dy
\\ &
\le
\|f\|_{\infty,\ell}(1+|t|)^{\ell} \delta\int_{\lambda_{1}}^{\infty}  e^{-\delta y} (\delta y)^{k} dy
\\ &
\le
\|f\|_{\infty,\ell}(1+|t|)^{\ell} \int_{\delta \lambda_{1}}^{\infty}  e^{- y} y^{k} dy\le \|f\|_{\infty,\ell}(1+|t|)^{\ell} \int_{\sigma \lambda_{1}}^{\infty}  e^{- y} y^{k} dy \,,
\end{align*}
and so
\[
\lim_{ \sigma \to \infty} \,\,\sup_{\substack{\delta\ge \sigma\\ t \in \R}}\frac{|f(\delta +it)|}{(1+|t|)^{\ell}}
\leq
\frac{\|f\|_{\infty,\ell}}{\Gamma(1+k)}
\lim_{ \sigma \to \infty} \,\,
\int_{\sigma \lambda_{1}}^{\infty}  e^{-y} y^{k} dy =0\,.\qedhere
\]
\end{proof}

\smallskip

With a similar technique we obtain the following lemma -- needed for the proof of Theorem~\ref{Theoremfiniteorder} below.

\smallskip

\begin{Lemm} \label{Lemma2} Given a frequency $\lambda$ with  $\lambda_{1}=0$, let $f$ be the limit function of $D=\sum a_{n} e^{-\lambda_{n}s}$ with
$\sigma_{c}^{\lambda,m}(D)\in \mathbb{R}$ for some $m\ge 0$. Then for every $\gamma>0$
$$\lim_{ \substack{\sigma \to \infty\\ \sigma >\sigma_{c}^{\lambda,m}(D)} }\,\, \sup_{\substack{s\in [\re\ge\sigma]\\ |Im(s)|\le \gamma}}|f(s)-a_{1}|=0.$$
\end{Lemm}

\begin{proof} We may assume that $a_{1}=0$.
Fix some $\sigma_{0}>\sigma_{c}^{\lambda,m}(D)$ and $\gamma >0$. Then by Theorem~\ref{corona2} (second part) the $(\lambda,m)$-Riesz means of $D$ converge uniformly on $\sigma_{0}~+~i[-\gamma,\gamma]$, hence
$$ \sup_{|t|\le \gamma} |R_{y}^{\lambda,m}(D)(\sigma_0+it)|=C(\sigma_{0}, \gamma)<\infty.$$
 Given $|t|> \gamma$, we  apply  (the integral formula from) Theorem~\ref{corona2}
to the translation $D_{\sigma_0 + it}$: For  all $\sigma >0 $
\begin{equation*}
f(\sigma_0 + \sigma + it)= \frac{\sigma^{1+m}}{\Gamma(1+m)}\int_{\lambda_{2}}^{\infty} e^{-\delta y} y^{1+m} R_{y}^{\lambda,m}(D_{\sigma_0})(it) dy.
\end{equation*}
Now we follow the esimates  of the proof of Theorem \ref{Theo2}. Replacing the expression $\|f\|_{\infty,\ell}(1+|t|)^{\ell}$ by $C(\sigma_{0}, \gamma)$, we conclude
\begin{equation*}
\lim_{ \sigma \to \infty} \sup_{\substack{s\in [\re\ge\sigma]\\ |Im(s)|\le \gamma}}|f(s)|\le \frac{C(\sigma_{0}, \gamma)}{\Gamma(1+m)} \lim_{\sigma\to \infty}\int_{\lambda_{2}\sigma}^{\infty} e^{-y} y^{m} dy=0. \qedhere
\end{equation*}
\end{proof}

\bigskip

\subsection{Finite order}
In \eqref{finiteorderdefintroU} we defined what it means that a holomorphic function $f:[\re>0] \to \mathbb{C}$
has finite uniform order $\nu_f$ (see also  \cite[p.14]{HardyRiesz}) The following result shows that our new scale
of Banach spaces is intimately linked with this notion.

\begin{Theo} \label{Theoremfiniteorder}
Let $\lambda$ be a frequency, $\ell \ge 0$, and $f\colon [\re>0] \to \C$  a holomorphic function which has a Riesz germ.
\begin{itemize}
\item[(1)]
If $f\in H_{\infty,\ell}^{\lambda}[\re>0]$, then there is $C >0$ such that for all $\sigma >0$ and all~$|t| >1$
\[
|f(\sigma + it)|\leq C |t|^\ell.
\]
In particular, if  $f\in H_{\infty,\ell}^{\lambda}[\re>0]$, then $f$ has finite uniform order $\nu_f \leq \ell$ on $[\re >0]$.
\smallskip
\item[(2)]
Conversely, $f_{\mu}=f(\mu+ \cdot)\in H_{\infty,\ell}^{\lambda}[\re>0]$
for all $\mu >0$ whenever
 there are $C,t_0 >0$
such that for all $\sigma >0$ and all $|t| >t_0$
\[
|f(\sigma + it)|\leq C |t|^\ell\,,
\]
so in particular if
$f$ has finite uniform order $\nu_f < \ell$  on $[\re >0]$.
\end{itemize}
\end{Theo}

\smallskip

The proof  relies on Lemma~\ref{Lemma2} and the following lemma.
\begin{Lemm} \label{coro}
 Let $\ell \ge 0$ and $\lambda$ an arbitrary frequency. Then for every $f\in H_{\infty,\ell}^\lambda [\re >0]$ we have
 \[\sup_{\substack{\sigma > 0\\ t \in \R}}\frac{|f(\sigma+it)|}{(1+|t|)^{\ell}} < \infty\,.\]
\end{Lemm}
\begin{proof} Since $|a_{1}|\le C \|f\|_{\infty,\ell}$ by Theorem~\ref{boundednesscoeff}, we may assume that $\lambda_{1}>0$.
Note first that by Theorem~\ref{Theo2} there is some $\sigma_0 >0$ such that
\[
\sup_{\substack{\sigma > \sigma_0\\ t \in \R}}\frac{|f(\sigma+it)|}{(1+|t|)^{\ell}} \leq 1\,.
\]
On the other hand, for $\sigma \leq \sigma_0$ and $t \in \R$ we have that
\[
\frac{|f(\sigma+it)|}{(1+|t|)^{\ell}}
=
\frac{|f(\sigma+it)|}{(1+|\sigma+it|)^{\ell}} \frac{(1+|\sigma+it|)^{\ell}}{(1+|t|)^{\ell}}
\leq
\frac{|f(\sigma+it)|}{(1+|\sigma+it|)^{\ell}}
( 1 +  \sigma_0)^\ell\,. \qedhere
\]
\end{proof}

\bigskip

\begin{proof}[Proof of Theorem \ref{Theoremfiniteorder}]
(1)  Let $f\in H_{\infty,\ell}^{\lambda}[\re>0]$. Then by Lemma~\ref{coro} we know that
\begin{equation*} \label{def1}
\exists~ R>0 ~ \forall ~ \sigma >0, ~ t \in \mathbb{R} \colon |f(\sigma+it)|\le R(1+|t|)^{\ell}.
\end{equation*}
Then $f$ satisfies \eqref{finiteorderdefintroU} with $t_0 =1$ and $C=R2^{\ell}$, since $(1+|t|)^{\ell}\le 2^{\ell} |t|^{\ell}$ for all $|t|\ge 1$.

(2) We assume that there are $C,t_0 >0$ such that for all $\sigma >0$
and all $|t| > |t_0|$
\[
|f(\sigma+it)|\le C|t|^{\ell} \le C(1+|t|)^{\ell}\,.
\]
Fix some $\mu >0$, and assume without loss of generality that $\lambda_{1}=0$. Then we know from Lemma~\ref{Lemma2} (with $\gamma=t_{0}$) that there is some $\sigma_{0}>\mu$ with $|f(\sigma+it)|\le C$ for all $\sigma>\sigma_{0}$ and $|t|\le t_{0}$. Since the continuous function $f_\mu$ is bounded on the rectangle $[\mu, \sigma_0]\times [-t_0, t_0]$,
 we have that
\begin{align*}
\|f_\mu\|_{\infty, \ell}
\leq \sup_{\substack{\sigma > 0\\ |t|\ge t_{0}}}\frac{|f(s)|}{|(1+s)^{\ell}|}+\sup_{\substack{\sigma_{0}>\sigma > \mu\\ |t|\le t_{0}}}\frac{|f(s)|}{|(1+s)^{\ell}|}+\sup_{\substack{\sigma > \sigma_{0}\\ |t|\le t_{0}}}\frac{|f(s)|}{|(1+s)^{\ell}|},
\end{align*}
which then is finite.
\end{proof}

\bigskip

\subsection{Equivalence} \label{SectionCoincidence}

Recall that  $\lambda$-Dirichlet series $D=\sum a_n e^{-\lambda_n s}$ converge on half-planes in the complex plane, where they define  holomorphic functions, and in general the largest possible  half-planes of convergence, uniform convergence, or absolute convergence differ.

In the early days of the theory, a very prominent research problem  was to characterize the class of those frequencies $\lambda$ for which  boundedness of the limit function of $D$ on  $[\re>0]$
implies uniform convergence of $D$ on every smaller half plane $[\re>\varepsilon]$. In other words, if  $\mathcal{D}_{\infty}(\lambda)$ stands for all $\lambda$-Dirichlet series $D=\sum a_{n}e^{-\lambda_{n}s}$ that  on  $[\re>0]$ converges to a  bounded
(holomorphic)  function, then the question is whether all these series even converges uniformly on all smaller half-planes $[\re >\varepsilon]$ .

Bohr  was the first who in \cite{Bohr} managed to isolate a prominent class of such $\lambda$'s, namely those for which there is some $\beta>0$ such that
\begin{equation} \label{BC}
\frac{1}{\lambda_{n+1}-\lambda_n} =O(e^{\beta \lambda_n} )\,; \tag{BC}
\end{equation}
slightly more precise, we in this case say that $\lambda$ satisfies \eqref{BC} with constant $\beta$.

We  say that 'Bohr's theorem holds for $\lambda$', whenever every $D\in \mathcal{D}_{\infty}(\lambda)$ converges uniformly on $[\re>\varepsilon]$ for every $\varepsilon>0$. In particular, frequencies $\lambda$ with \eqref{BC} satisfy Bohr's theorem.
 Later this result was extended by Landau in \cite{Landau}
assuming the less restrictive assumption
\begin{equation} \label{(LC)}
\forall ~\delta>0 ~\exists ~C>0~ \forall ~n  \colon ~~ \lambda_{n+1}-\lambda_{n}\ge C e^{-e^{\delta\lambda_{n}}}\,, \tag{LC}
\end{equation}
and very recently, Bayart in \cite{Ba20} added a further interesting condition, which provides a nontrivial extension of \eqref{(LC)}: A frequency
$\lambda$ satisfies $(NC)$ if
\begin{equation} \label{(NC)}
\forall ~\delta>0 ~\exists ~C>0~ \forall ~ m> n \colon ~~
\log\bigg(\frac{\lambda_m +\lambda_n}{\lambda_m -\lambda_n}   \bigg) + (m-n)\leq C e^{\delta \lambda_n}  \,. \tag{NC}
\end{equation}
Examples show that any concrete $\lambda$ may or may not satisfy Bohr's theorem, and we refer to \cite{Ba20} and \cite{schoolmann2018bohr}
for detailed information on all this.

In this context, an important  recent achievement  is   that the property '$\lambda$ satisfies Bohr's theorem' in fact
is equivalent to a number of facts which  seem absolutely unavoidable for a reasonable abstract theory
of general Dirichlet series. The following so-called 'equivalence theorem' is taken from \cite[Theorem 4.6]{CaDeMaSc_VV}) (in fact, the main part of this result was proved earlier in \cite[Theorem~2.16]{defant2020riesz} and \cite[Theorem~5.1]{defant2020variants}).
\smallskip

 \begin{Theo} \label{equivalence}
Let $\lambda$ be a frequency. Then the following are equivalent:
\begin{itemize}
\item[(1)] $\lambda$ satisfies Bohr's theorem.
\item[(2)] $\mathcal{D}_{\infty}(\lambda)$ is a Banach space.
\item[(3)] $\mathcal{D}_{\infty}(\lambda)=\mathcal{H}_{\infty}(\lambda)$ isometrically and coefficient preserving.
\item[(4)] $\mathcal{D}_{\infty}(\lambda)= H_{\infty}^{\lambda}[\re>0]$ isometrically and coefficient preserving.
\item[(5)] $\lambda$ satisfies Montel's theorem in $\mathcal{D}_{\infty}(\lambda)$\,.
\end{itemize}
\end{Theo}

\smallskip

Let us take some time to explain this result in  more detail.
Note  first  that $\mathcal{D}_{\infty}(\lambda)$
endowed with the sup norm $\|D\|_\infty = \sup_{\re >0} |f(s)|$, where $f$ denotes the limit function of $D$, forms a normed space. But  unfortunately
this space in general is not complete, and so the  theorem ensures  that  this only holds true if  $\lambda$ satisfies Bohr's theorem -- in particular under each of the conditions of Bayart, Bohr, or Landau.

 To see a non-trivial  example, $(\sqrt{\log n})$ doesn't satisfy $(BC)$ but $(LC)$. Hence,  Bohr's theorem holds for $(\sqrt{\log n})$, and so $\mathcal{D}_{\infty}((\sqrt{\log n }))$ forms  a Banach space -- a seemingly non-trivial fact.

  An important step  is to understand that  the Banach space
$H_{\infty}^{\lambda}[\re>0]$ always contains the (in general non-complete) space $\mathcal{D}_{\infty}(\lambda)$ isometrically
(see \cite[Proposition 3.4]{schoolmann2018bohr}). Hence  both spaces  in general differ, but in view of the theorem they coincide iff Bohr's theorem holds true for $\lambda$.

Moreover, as already mentioned in Section~\ref{monti}, Bayart proved a 'Montel type theorem' for $\mathcal{D}_{\infty}( (\log n))
=\mathcal{H}_{\infty}( (\log n))$, which turned out to be  a corner stone of the ordinary theory.  The
equivalence theorem shows that the analog  result holds true for $\mathcal{D}_{\infty}( \lambda)$
if and only if $\lambda$ satisfies Bohr's theorem.

\smallskip

The idea now is to extend parts of Theorem~\ref{equivalence} to the new setting.
Given a frequency $\lambda$ and $\ell \ge 0$, let  $$\mathcal{D}_{\infty,\ell}(\lambda)$$ denote the space of all $\lambda$-Dirichlet series
$D = \sum a_n e^{-\lambda_n s}$ that are $(\lambda,\ell)$-Riesz summable on all of $[\re>0]$ and  have limit functions  $f\in H_{\infty,\ell}^{\lambda}[\re>0]$. Together with the norm $\|D\|_{\infty,\ell}=\|f\|_{\infty,\ell}$, this leads to another scale
of normed spaces of $\lambda$-Dirichlet series.

Recall again from Corollary~\ref{Theo41reproved} that the $\lambda$-Riesz germ of each $f\in H_{\infty,\ell}^{\lambda}[\re>0]$
is $(\lambda,k)$-Riesz summable for each $k > \ell$, and so one of the questions in the following will be, under with additional assumptions on $\lambda$, Corollary~\ref{Theo41reproved} even holds for
$k = \ell$.

As mentioned  above, $\mathcal{D}_{\infty}(\lambda)$ is always an isometric subspace
of $H_{\infty}^{\lambda}[\re>0]$, hence in the particular case $\ell =0$ we by definition have the isometric equality
\begin{equation}\label{racci}
  \mathcal{D}_{\infty,0}(\lambda)=\mathcal{D}_{\infty}(\lambda)\,.
\end{equation}
Together with Corollary~\ref{oldcase} we see that  Theorem~\ref{equivalence} settles the case $\ell =0$.

For $\ell >0$ the following result partly serves as a sort of substitute of  Theorem~\ref{equivalence}.

\smallskip

\begin{Theo}\label{equiv} Let $\lambda$ be a frequency and $\ell\ge 0$. Then the following are equivalent:
\begin{itemize}
\item[(1)]  $\mathcal{D}_{\infty,\ell}(\lambda)$ satisfies Montel's theorem.
\item[(2)] $\mathcal{D}_{\infty,\ell}(\lambda)$  is a Banach space.
\item[(3)] The $\lambda$-Riesz germ  of every $f \in H_{\infty,\ell}^{\lambda}[\re>0]$ is $(\lambda,\ell)$-Riesz summable on $[\re > 0]$.
\item[(4)] $\mathcal{D}_{\infty,\ell}(\lambda)=H_{\infty,\ell}^{\lambda}[\re>0]$ holds isometrically and coefficient preserving.
\end{itemize}
\end{Theo}

\begin{proof}
If  $\mathcal{D}_{\infty,\ell}(\lambda)$ satisfies Montel's theorem, then the same proof as for Theorem~\ref{completeness} shows that $\mathcal{D}_{\infty,\ell}(\lambda)$ is complete, so we have that $(1) \Rightarrow (2)$. Let us show that $(2) \Rightarrow (4)$, and fix some  $f\in H_{\infty,\ell}^{\lambda}[\re>0]$ with  associated $\lambda$-Riesz germ~$D$.
By Corollary~\ref{41} we know that for every $u >0$ and every $k >\ell$
\[
\text
{
$\lim_{x \to \infty}R_{x}^{\lambda,k}(f)(u + \pmb{\cdot}) = f(u + \pmb{\cdot})$ in $H_{\infty,\ell}^{\lambda}[\re>0]$.
}
\]
 But since $$R_{x}^{\lambda,k}(D_u)=R_{x}^{\lambda,k}(f)(u + \pmb{\cdot}) \in \mathcal{D}_{\infty,\ell}(\lambda)\,,$$
  where  $D_u = \sum a_n e^{-\lambda_n u}e^{-\lambda_n s} \in \mathcal{D}_{\infty,\ell}(\lambda)$ as usual denotes the translation of $D$ about $u >0$,  we  eventually see that  $D \in \mathcal{D}_{\infty,\ell}(\lambda)$.  Hence $\mathcal{D}_{\infty,\ell}(\lambda)=H_{\infty,\ell}^{\lambda}[\re>0]$,
 the claim from $(4)$. The equivalence  $(3)\Leftrightarrow (4)$ is obvious, so  that it finally suffices to check that $(4) \Rightarrow (1)$. But if  $\mathcal{D}_{\infty,\ell}(\lambda)=H_{\infty,\ell}^{\lambda}[\re>0]$ holds true, then trivially by Theorem \ref{Montel} Montel's theorem holds in   $\mathcal{D}_{\infty,\ell}(\lambda)$.
\end{proof}

The case $\ell=0$ from the preceding theorem allows another  interesting   remark on Theorem~\ref{equivalence}.

\begin{Rema}
 Bohr's theorem is valid for $\lambda$  if and only if every Dirichlet series which for some $m \ge 0$ is somewhere $(\lambda,m)$-Riesz summable and has  a limit function that extends holomorphically to a bounded function on $[Re>0]$, converges uniformly on $[Re>\sigma]$ for every $\sigma>0$.
\end{Rema}

\smallskip
For concrete  frequencies $\lambda$, in particular $\lambda = (\log n)$, the following result is our main application of Theorem~\ref{equiv}.

\smallskip

\begin{Theo} \label{41improved}
Let $\ell>0$. If \eqref{BC} holds for $\lambda$, then the $\lambda$-Riesz germ $D$ of every $f \in H_{\infty,\ell}^{\lambda}[\re>0]$ is $(\lambda,\ell)$-Riesz summable on $[\re > 0]$. In particular, all equivalent statements from Theorem~\ref{equiv} hold.
\end{Theo}

It seems very interesting to decide whether or not the preceding result even holds under the condition \eqref{(LC)} or more generally  \eqref{(NC)}.

\smallskip

 Divided into three lemmas of independent interest, the proof of Theorem~\ref{41improved} is given at the very end of this section.
 The first lemma gives a necessary condition  under which  $H_{\infty,\ell}^{\lambda}[\re>0]$ and $\mathcal{D}_{\infty,\ell}(\lambda)$ coincide, and is based on
a Phragmen-Lindel\"of type argument again borrowed from \cite{HardyRiesz}.

\smallskip

\begin{Lemm}\label{christmas} Let $\ell>0$ and $f\in H_{\infty,\ell}^{\lambda}[\re>0]$. Assume that there exist $\sigma_0>0$ and $0\le \ell^{\prime}<\ell$ such that $f_{\sigma_0} = f(\sigma_0 + \cdot)\in H_{\infty,\ell^{\prime}}^{\lambda}[\re>0]$. Then the
$\lambda$-Riesz germ $D$ of $f$ is $(\lambda,\ell)$-Riesz summable on $[\re > 0]$, i.e. $D\in \mathcal{D}_{\infty,\ell}(\lambda)$.
\end{Lemm}

\begin{proof} Fix $\varepsilon>0$. By Theorem~\ref{Theoremfiniteorder}, part (1)  there is $D_0>0$  such that for all
$\sigma >0$ and $|t|>1$
\begin{equation} \label{fo1}
  |f(\sigma +it)|\le D_0 |t|^{\ell}\,,
\end{equation}
and some $D_1>0$ such that for all
$\sigma > \sigma_0$ and $|t|>1$
\begin{equation}\label{fo2}
  |f(\sigma +it)|\le D_1 |t|^{\ell'}\,.
\end{equation}
Combining \eqref{fo1} and \eqref{fo2}
by  \cite[Theorem~14]{HardyRiesz} (another  Phragmen-Lindel\"of principle), there  are $C, t_0 >0$ such that for every $\varepsilon\le \sigma \le \sigma_0+1$
and every $|t|\ge t_0 $
\begin{equation}\label{fo3}
  |f(\sigma +it)|\le C |t|^{\kappa(\sigma)}\,,
\end{equation}
where $\kappa: [\varepsilon, \sigma_0+1]\to [0, \infty[$ is the affine linear function linking the points  $(\varepsilon,\ell)$ and $(\sigma_0 +1,\ell^{\prime})$.
We claim
that
\begin{equation*} \label{fo4}
  \text{$f(\varepsilon +\sigma+\cdot)\in H_{\infty,\kappa(\beta)}^{\lambda}[\re>0]$ for every $\sigma >0$\,,}
\end{equation*}
and we do this with Theorem~\ref{Theoremfiniteorder}, part (2) showing for all $\sigma >0$ and $|t| \ge  \max\{1,t_0\} $
\begin{equation}\label{fo53}
  |f(\varepsilon + \sigma +it)|\le \max\{C,D_1\} |t|^{\kappa(\sigma)}\,.
\end{equation}
 Indeed, since $\kappa\ge \ell'$ on $[\varepsilon, \sigma_0+1]$, this is immediate from \eqref{fo2} and \eqref{fo3}.

Since $\kappa<\ell$ on  $[\varepsilon, \sigma_0+1]$ (recall that $\ell' < \ell$), we by Corollary~\ref{Theo41reproved} obtain from \eqref{fo4} that $\sigma_{c}^{\lambda,\ell}(D_\varepsilon)\le 0$ for every $\varepsilon>0$, or equivalently $\sigma_{c}^{\lambda,\ell}(D)\le \varepsilon$ for every $\varepsilon>0$.
\end{proof}

\smallskip

The next tool, based on the preceding one,  is a reformulation of \cite[Theorem~44]{HardyRiesz}. For  the sake of completeness we include an argument within our new setting.
\smallskip

\begin{Lemm}\label{fridayA} Let $\ell>0$, and $f\in H_{\infty,\ell}^{\lambda}[\re>0]$. Then the
$\lambda$-Riesz germ $D$ of $f$ is $(\lambda,\ell)$-Riesz summable on $[\re > 0]$, i.e. $D\in \mathcal{D}_{\infty,\ell}(\lambda)$, provided $\sigma_{a}(D)  <\infty$.
\end{Lemm}

\begin{proof}
Since $\sigma_{a}(D)<\infty$, we have that
$$f(s)=\sum_{n=1}^{\infty}a_{n}(f)e^{-\lambda_{n}s}, ~~ s\in [\re>\sigma_{a}(D)]\,,$$
and so we see that $f$ is bounded on $[\re>\sigma_{a}(D)+1]$. Hence Lemma~\ref{christmas}
with $\sigma_0=\sigma_{a}(D)+1$ and $\ell^{\prime}=0$ implies $D\in \mathcal{D}_{\infty,\ell}(\lambda)$.
\end{proof}

\smallskip

For the last lemma we need
further notation and information.
Given a frequency~$\lambda$,  define
\[
L(\lambda) := \sup_{D \in \mathcal{D}(\lambda)} \sigma_{a}(D)-\sigma_{c}(D)\,,
\]
where we recall that $\sigma_{c}(D)$  and $\sigma_{a}(D)$ define the abscissas of convergence and absolute convergence
(with respect to ordinary summation).
 Then straightforward arguments show that
\begin{equation} \label{Lstrip1vectorvalued}
L(\lambda)= \sigma_{c}\left(\sum e^{-\lambda_{n}s} \right)=
\sigma_{a}\left(\sum e^{-\lambda_{n}s} \right)  \,.
\end{equation}
With a less obvious arguments, Bohr in \cite[\S 3, Hilfssatz 4]{Bohr2} showed  that his condition \eqref{BC} implies   $L(\lambda) <~\infty$.
In passing we remark (although this  is not needed for what's coming) that by \cite[\S 3, Hilfssatz 4]{Bohr2}
\begin{equation} \label{Lstrip2vectorvalued}
L(\lambda) =\limsup_{N \to \infty} \frac{\log(N)}{\lambda_{N}}\,.
\end{equation}
It is evident that $L((n))=0$ and $L((\log n))=1$.

\smallskip

\begin{Lemm}\label{fridayB} Let $\ell>0$, and $f\in H_{\infty,\ell}^{\lambda}[\re>0]$ with its associated $\lambda$-Riesz germ~$D$. Then $\sigma_{a}(D)  \leq L(\lambda) + \beta \ell <\infty$, whenever $\lambda $ satisfies~\eqref{BC} with exponent~$\beta$.
\end{Lemm}

To see an example, we note that for $\lambda = (n)$ we get the  conclusion  $\sigma_{a}(D)\leq~0$, whereas for
$\lambda = (\log n)$ we obtain that $\sigma_{a}(D)  \leq  1 +  \ell$.

\begin{proof}
Condition \eqref{BC} for $\lambda$ with exponent $\beta$ means that  there is some $C >0$ such that for all $n$
\begin{equation*}
  \frac{1}{\lambda_{n+1}-\lambda_{n+1}} \leq Ce^{\lambda_n\beta}\,.
\end{equation*}
Adding more elements to $\lambda$, we assume  without loss of generality that $\lambda_{n+1}-\lambda_n \leq 1 $ for all $n$
(see the first part of the proof of \cite[Theorem~4.2]{schoolmann2018bohr} for details).
We now claim that the Bohr coefficients $(a_{n}(D))=(a_{n}(f))$ of $D$ for every $\delta >0$ and $k > \ell$ satisfy
\begin{align}\label{end}
A=\sup_{n}|a_{n}(f)e^{-(\beta k +\delta) \lambda_{n}}| < \infty\,.
\end{align}
Indeed, given $\delta >0$ and $k > \ell$, by  Theorem \ref{boundednesscoeff} we conclude that for  every $n$ and some constant
$C = C(k,\ell,\lambda_1)$
\begin{align*}
  |a_{n}(f)|
  &
  \le C
  \Big(\big(\frac{\lambda_{n+1}}{\lambda_{n+1}-\lambda_{n}}\big)^{k}+\big(\frac{\lambda_{n}}{\lambda_{n}-\lambda_{n-1}}\big)^{k} \Big) \|f\|_{\infty,\ell}
 \\&
 \le 2C
  \Big(\lambda_{n+1}^{k} e^{\beta k \lambda_n } + \lambda_{n}^{k}  e^{\beta k \lambda_{n-1} }\Big)\,\,
  \|f\|_{\infty,\ell}
  \\&
 \le 2C
  \Big((\lambda_{n+1})^{k} e^{\beta k \lambda_n }\Big)\,\,
  \|f\|_{\infty,\ell}
    \,.
\end{align*}
Hence, choosing  $D=D(\delta,k,\lambda) >0$ such that $(1+\lambda_n)^k\leq De^{\delta\lambda_n}$ for all $n$, we finally obtain \eqref{end}, namely that for all $n$
\[
|a_{n}(f)| \leq 2CD  e^{(\delta+\beta k) \lambda_n }
  \|f\|_{\infty,\ell}\,.
\]
 Then by \eqref{end} and by the fact that under Bohr's condition $L(\lambda)< \infty$, we see that  for every $\varepsilon >0$
 \begin{align*}
   \sum |a_n| e^{-((\beta k +\delta)+(L(\lambda)+\varepsilon )) \lambda_n}
    &
    = \sum |a_ne^{-(\beta k +\delta) \lambda_n}| e^{-(L(\lambda) +\varepsilon) \lambda_n}
    \\&
    \leq A \sum e^{-(L(\lambda) +\varepsilon) \lambda_n} < \infty\,.
 \end{align*}
Consequently,  $\sigma_{a}(D)\le (\beta k +\delta)+(L(\lambda)+\varepsilon )$, which is the conclusion
(whenever $k \to\ell$, $\delta \to 0$, and $\varepsilon \to 0$).
\end{proof}

\begin{proof}[Proof of Theorem~\ref{41improved}]
Applying first Lemma~\ref{fridayB} and then second Lemma~\ref{fridayA}, we finally obtain as desired  Theorem~\ref{41improved}.
\end{proof}

\section{\bf Appendix}   \label{appendix}
This section is devoted to the proofs of Theorem~\ref{basic},  Theorem~\ref{niceproof?}, Theorem~\ref{corona2},  and
Theorem~\ref{Bohr-Cahen-Riesz}. These are those results from \cite{HardyRiesz}, and  in a few cases even improvements
of them,  which were of central  importance for what we tried to explain in the  preceding sections. In fact, most of the proofs from  \cite{HardyRiesz} are not given in full detail and only sketches -- together with this appendix
our article is indeed self contained.

We need  a couple of technical lemmas  to properly prepare these proofs.
To do so, recall from  Section~\ref{Riesz means} the definition of summatory functions $S_x^{\lambda,r}(D)(s)$ with respect to a frequency $\lambda$, a Riesz
weight $r$, and a $\lambda$-Dirichlet series $D$.

\bigskip

\subsection{Integral forms} \label{integral forms}
Given a frequency $\lambda$ and a $\lambda$-Dirichlet series $D$, we in Section \ref{Riesz means} defined the summatory function $S^{\lambda,k}_{x}(D)$.
We start with the following integral form of this function, which, simple as it is, seems to be the seed for most of the coming arguments.

\begin{Lemm}\label{AbelfirstcaseB}
Let $D\in \mathcal{D}(\lambda)$, $r$ a Riesz weight, and  $x \geq 0$. Then for all $s,w \in \mathbb{C}$
\begin{equation} \label{Ostfriesland}
S^{\lambda,r}_{x}(D)(s+w)= -\int_{0}^{x} S^{\lambda,0}_{t}(D)(s)\Big( e^{-w \,\bullet} r(x,\bullet) \Big)'(t) dt\,,
\end{equation}
and in  particular,
\begin{equation} \label{A}
S^{\lambda,r}_{x}(D)(s)= -\int_{0}^{x} S^{\lambda,0}_{t}(D)(s)  r(x,\bullet)'(t)dt\,.
\end{equation}
Moreover, if $C = \sum a_n$ is a series and  $k \ge 0$, then for $r(x,t)=(x-t)^{k+1}$
\begin{equation}  \label{B}
\frac{d}{dt}S^{\lambda,k+1}_{t}(C)_{|t=x}=  (k+1)S^{\lambda,k}_{x}(C)\,.
\end{equation}
\end{Lemm}

\smallskip

\begin{proof}
  It suffices to do the proof for $s =0$. We write
  $S^{\lambda,r}_{x}(D)(w)$ as a Stieltjes integral,
  \[
  S^{\lambda,r}_{x}(D)(w) =
    \sum_{\lambda_n < x} a_n e^{-\lambda_n w} r(x,\lambda_n)
   =
  \int_{0}^{x} e^{-wt}r(x,t) d S^{\lambda,0}_{t}(D)(0)\,,
  \]
  and then \eqref{Ostfriesland}  follows by partial integration. Clearly, \eqref{A} is the special case $w=0$ in \eqref{Ostfriesland}. The proof of \eqref{B} follows applying \eqref{A} two times:
    \begin{align*}
  \frac{d}{dt} S_{t}^{\lambda,k+1}(D)
  &
  =
  \frac{d}{dt} \int_0^t S_{u}^{\lambda,0}(D) (k+1)(t-u)^k du
  \\&
  =
  (k+1) k \int_0^t S_{u}^{\lambda,0}(D) (t-u)^{k-1} du
    =
  (k+1) S_{t}^{\lambda,k}(D)\,. \qedhere
\end{align*}
\end{proof}

\smallskip

Recall that, the first statement of Theorem~\ref{corona2} claims  $(\lambda,k)$-Riesz summability of the $\lambda$-Dirichlet series on the half plane $[\re> \re~s_{0}]$, given $(\lambda,k)$-Riesz summability at $s_{0}$. Rougly speaking, the strategy for proving this claim is to substitute the expression $S^{\lambda,0}_{t}(D)(s)$ of \eqref{Ostfriesland} by $S^{\lambda,k}_{t}(D)(s)$. In other terms,  we want to increase the order of the summatory function of $f$ on the right hand side of \eqref{Ostfriesland} from zero to $k$.

 The following lemma is the first step in this direction --  partial integration allows to increase the order  zero to the order $m+1$, where  $m\in \N$ with $m<k\le m+1$. Then, for the case  $k \notin \N$, we   utilize the subsequent Lemma \ref{harrie} to drop the order to $k$ as desired.

\smallskip

\begin{Lemm} \label{Abelta}
Let $D\in \mathcal{D}(\lambda)$, $r$ a Riesz weight,  and  $m\in \N_{0}$. Then for all $s,w\in \C$ and all $x >0$
\begin{align*}
S_{x}^{\lambda,r}(D)(s+w)= -e^{-w x}S_{x}^{\lambda,r}(D)(s)+\frac{(-1)^{m}}{(m+1)!}\int_{0}^{x} S_{t}^{\lambda,m+1}(D)(s) \partial^{m+2}h_{w}(t) dt,
\end{align*}
where $h_{w}(t)= (e^{-w t}-e^{-w x})r(x,t)$.
In particular,
\begin{align} \label{finally}
\begin{split}
  S_{x}^{\lambda,r}(D)(s+w) &
 -
\frac{1}{\Gamma(m+2)}\int_{0}^{x} S_{t}^{\lambda,m+1}(D)(s) w^{m+2} e^{-wt} r(x,t)dt
\\&
= -e^{-wt}  S_{x}^{\lambda,r}(D)(s)
+  \int_0^x  S_{t}^{\lambda,m+1}(D)(s)  (g_1(t) + g_2(t)) dt \,,
           \end{split}
\end{align}
where
\begin{align*}
&
  g_1(t) = \sum_{j=1}^{m+1} \binom{m+2}{j} (-w)^j e^{-wt}  \,r(x,\bullet)^{(m+2-j)}(t)
  \\&
  g_2(t) = (e^{-wt} - e^{-wx}) \,r(x,\bullet)^{(m+2)}(t)
      \,.
\end{align*}
\end{Lemm}

\begin{proof}
 By partial integration
and  Lemma~\ref{AbelfirstcaseB} we  have
\begin{align*} \label{erst}
&
S^{\lambda,r}_{x}(D)(s+w)
\\&
= -\int_{0}^{x} S^{\lambda,0}_{t}(D)(s)\Big( e^{-w \,\bullet} \,r(x,\bullet) \Big)'(t) dt
\\&
=-e^{-wx} \int_{0}^{x} S_{t}^{\lambda,0}(D)(s) \,r(x,\bullet)^{\prime}(t) dt -\int_{0}^{x} S_{t}^{\lambda,0}(D)(s) \big((e^{-w\bullet}-e^{-wx})\,r(x,\bullet)\big)^{\prime}(t) dt.
\end{align*}
Using  \eqref{A} we know
\[
e^{-wx} \int_{0}^{x} S_{t}^{\lambda,0}(D)(s)\,r(x,\bullet)^{\prime}(t)  dt=e^{-wx} S_{x}^{\lambda,r}(D)(s)\,,
\]
and get
\begin{align*}
S^{\lambda,r}_{x}(D)(s+w)
=-e^{-wx} S_{x}^{\lambda,r}(D)(s) -\int_{0}^{x} S_{t}^{\lambda,0}(D)(s) \big((e^{-w\bullet}-e^{-wx})\,r(x,\bullet)\big)^{\prime}(t) dt.
\end{align*}
Finally, combining  partial integration and \eqref{B},  gives
 the first conclusion. For the proof of the second claim note that
\begin{align*}
\partial^{m+2}h_{w}(t)=&(-w)^{m+2} e^{-w t} r(x,t) + g_{1}(t)+g_{2}(t)\,,
\end{align*}
which finishes the proof.
\end{proof}

\smallskip
We go on analyzing the integral on the right side of the preceding lemma. We split it into  two summands --
one which will lead us to  the integral describing the  $(\lambda,k)$-Riesz sum in  Theorem~\ref{corona2}, and an other one  which we  in fact control whenever $x$ increases to infinity.

\smallskip

\begin{Lemm} \label{rekto}
Let $D\in \mathcal{D}(\lambda)$, $r$ a Riesz weight,  $k>0$ and $m\in \N_{0}$ with $m< k < m+1$. Then for all $s,w\in \C$ and all $x >0$
\begin{align*}
S_{x}^{\lambda,r}(D)(s+w) &
 -
\frac{1}{\Gamma(m+2)}\int_{0}^{x} S_{t}^{\lambda,m+1}(D)(s) w^{m+2} e^{-wt} r(x,t)dt
\\&
= -e^{-wt}  S_{x}^{\lambda,r}(D)(s)
\\&
\,\,\,\,\,\,\,
+  C(k) \int_0^x  S_{y}^{\lambda,k}(D)(s)  \int_y^x (t-y)^{m-k} (g_1(t) + g_2(t)) dt dy\,,
\end{align*}
where $g_1$ and $g_2$ are as in \eqref{Abelta} and
$  C(k)=(-1)^m \dfrac{1}{(m+1)\Gamma(1+k)\Gamma(1+m-k)}\,.$
\end{Lemm}

\smallskip

The proof of this lemma needs further independently useful preparation, and we postpone it to the next section.

\bigskip

\subsection{Changing orders} \label{order}

Given a frequency $\lambda$ and a series $C = \sum a_n$, we present  two devices which allow to increase or decrease the order $\kappa$ of a given  summatory function $S_x^{\lambda,\kappa}(C)$. Both lemmas are indispensable technical tools for the  proofs of Theorem~\ref{corona2}, Theorem~\ref{Bohr-Cahen-Riesz}, and Theorem~\ref{niceproof?}.

\smallskip

We begin collecting a few basic fact about the classical Gamma function given by
\begin{equation*}
\Gamma(z)= \int_{0}^{\infty} x^{z-1} e^{-x} dx,\,\,\, ~ \re ~z >0.
\end{equation*}
Recall that for every $z\in [\re>0]$
\begin{equation}
\Gamma(1+z)=\Gamma(z)z.
\end{equation}
Moreover, the  Beta function is given by
\begin{equation*}
B(p,q)=\int_{0}^{1} y^{p-1}(1-y)^{q-1} dy, \,\,\, ~ p,q\in [\re>0]\,,
\end{equation*}
which in terms of Gamma functions reads
\begin{equation}\label{betafunction}
B(p,q)=\frac{\Gamma(p)\Gamma(q)}{\Gamma(p+q)}\,.
\end{equation}
This shows that for $p,q > -1$ with $q>p$ and $x>0$
\begin{equation} \label{betafunctioncalc}
\int_{0}^{x}y^{p} (x-y)^{q-p-1} dy=x^{q} \frac{\Gamma(p+1)\Gamma(q-p)}{\Gamma(q+1)}.
\end{equation}
which indeed follows by a simple substitution:
\begin{align*}
\int_{0}^{x}y^{p} (x-y)^{q-p-1} dy& =x^{q}\int_{0}^{x} x^{-1} (\frac{y}{x})^{p}(1-\frac{y}{x})^{q-p-1} dy\\ &=x^{q} \int_{0}^{1} \alpha^{p} (1-\alpha)^{q-p-1} d\alpha = x^{q} \frac{\Gamma(p+1)\Gamma(q-p)}{\Gamma(q+1)}.
\end{align*}
We several times need \eqref{betafunctioncalc}  in combination   with the fact that   for arbitrary $c,d\in \mathbb{R}$, $w \in \mathbb{C}$ and $y < x$
\begin{align} \label{againagain}
\int_{y}^{x} e^{-w t} (t-y)^{c} (x-t)^{d} dt&= (x-y)^{c+d+1}e^{-wy} \int_{0}^{1} e^{-w\beta(x-y)} \beta^{c} (1-\beta)^{d} d \beta;
\end{align}
indeed,
\begin{align*}
\int_{y}^{x} e^{-w t} (t-y)^{c} (x-t)^{d} dt&= \int_{0}^{x-y} e^{-w(y+\alpha)} \alpha^{c} (x-y-\alpha)^{d} d\alpha\\ &=(x-y)^{c+d+1} \int_{0}^{1} e^{-w(y+\beta(x-y))} \beta^{c} (1-\beta)^{d} d \beta\\ &=(x-y)^{c+d+1}e^{-wy} \int_{0}^{1} e^{-w\beta(x-y)} \beta^{c} (1-\beta)^{d} d \beta.
\end{align*}

\smallskip
We are ready for the first lemma taken from \cite[Lemma~6]{HardyRiesz}. Following the idea of our article we  for the sake of completeness repeat its proof.
\smallskip

\begin{Lemm}\label{harrie}
Let $C=\sum a_n$  be a series, $\lambda$ a frequency, and $\kappa, \mu >0$. Then for all $x > 0$
\begin{equation*}
S_x^{\lambda,\kappa+\mu}(C)  =\frac{\Gamma(\kappa + \mu +1)}{\Gamma(\kappa + 1)\Gamma(\mu)}
\int_0^x S_u^{\lambda,\kappa}(C) (x-u)^{\mu-1}  du\,.
\end{equation*}
Moreover, if $\kappa >0$, $\mu <1$ and $\mu \leq \kappa$, then for all $x >0$
\[
S_x^{\lambda,\kappa-\mu}(C) = \frac{\Gamma(\kappa - \mu +1)}{\Gamma(\kappa + 1)\Gamma(1-\mu)}
\int_0^x \frac{d}{du} S_u^{\lambda,\kappa}(C) (x-u)^{-\mu}  du\,.
\]
\end{Lemm}

\smallskip

\begin{proof}
For the proof of the first formula we start with the integral on the left side, and use \eqref{A}, Fubini's theorem and \eqref{betafunctioncalc}
to obtain
\begin{align*}
&
\int_0^x S_u^{\lambda,k}(C) (x-u)^{\mu-1}  du
\\&
=
\int_0^x k \int_t^x S_t^{\lambda,0}(C) (u-t)^{k-1}  (x-u)^{\mu -1} du dt
\\&
=
k\int_0^x   S_t^{\lambda,0}(C)) \int_t^x (u-t)^{k-1}  (x-u)^{\mu -1} du dt
\\&
=
\frac{\Gamma(k )\Gamma(\mu)}{\Gamma(k + \mu) }
k\int_0^x   S_t^{\lambda,0}(C) (x-t)^{\mu +k -1} dt
=
\frac{\Gamma(k )\Gamma(\mu)}{\Gamma(k + \mu) }
\frac{k}{\mu +k}   S_t^{\lambda, \mu +k}(C)\,,
\end{align*}
the conclusion.
To prove the second formula we use \eqref{B}  and the first formula to see that  for all $x >0$
\begin{align*}
  S^{\lambda,\kappa -\mu}_x(C)
  &
   = \frac{1}{\kappa -\mu +1} \frac{d}{dx} S^{\lambda,\kappa +(1-\mu)}_x(C)
   \\&
  = \frac{\Gamma(\kappa - \mu +1)}{\Gamma(\kappa + 1)\Gamma(1-\mu)}
\frac{d}{dx}\int_0^x  S_u^{\lambda,\kappa}(C) (x-u)^{-\mu}  du\,.
\end{align*}
On the other hand, integrating by parts we have
\begin{align*}
  \int_0^x  S_u^{\lambda,\kappa}(C) (x-u)^{-\mu}  du
  =\frac{1}{1-\mu}
    \int_0^x  \frac{d}{du}S_u^{\lambda,\kappa}(C) (x-u)^{1-\mu}  du\,.
  \end{align*}
  For the desired formula we differentiate the right side with respect to $x$.
\end{proof}

\smallskip

As announced we use the preceding lemma to add the still missing proof  of Lemma~\ref{rekto}.

\smallskip

\begin{proof}[Proof of Lemma~\ref{rekto}]
  By Lemma~\ref{Abelta} for every $s,w\in \mathbb{C}$
\begin{align*}
S_{x}^{\lambda,r}(D)(s+w)= -e^{-w x}S_{x}^{\lambda,r}(D)(s)+\frac{(-1)^{m}}{\Gamma(m+2)}\int_{0}^{x} S_{t}^{\lambda,m+1}(D)(s) \partial^{m+2}h_{w}(t) dt,
\end{align*}
where $h_{w}(t)= (e^{-w t}-e^{-w x})r(x,t)$. We calculate
\begin{align*}
\partial^{m+2}h_{w}(t)=&(-w)^{m+2} e^{-w t} r(x,t) + g_{1}(t)+g_{2}(t).
\end{align*}
The aim is to replace the  summatory function of order $m+1$ by the  summatory function of order $k$. If $k=m+1$, then we are fine -- but if $k<m+1$, then we need
Lemma~\ref{harrie}  (with the choices $\mu=m-k+1$ und $\kappa =k$):
$$S_{t}^{\lambda,m+1}(D)(s)= C(m,k)  \int_{0}^{t} S_{y}^{\lambda,k}(D)(s)(t-y)^{m-k} dy\,,$$
where
\[
C(k)=\frac{\Gamma(1+m)}{\Gamma(1+k)\Gamma(1+m-k)}\,.
\]
For the rest of the proof we assume that  $k<m+1$; in fact, it is easier to handle the case $k=m+1$ following basically  the same lines.
We  apply Fubini's theorem and obtain for $i=1,2$
\begin{align*}
C(k)^{-1}&\int_{0}^{x} S_{t}^{\lambda,m+1}(D)(s)g_{i}(t) dt=\int_{0}^{x} \int_{0}^{t} S_{y}^{\lambda,k}(D)(s) (t-y)^{m-k}dy~ g_{i}(t) dt\\ &=\int_{0}^{\infty} S_{y}^{\lambda,k}(D)(s)  \int_{0}^{x} \chi_{(0,t)}(y) (t-y)^{m-k} g_{i}(t) ~dt~ dy\\ &=\int_{0}^{x} S_{y}^{\lambda,k}(D)(s)  \int_{y}^{x} (t-y)^{m-k} g_{i}(t) ~dt~ dy.
\end{align*}
Combining all this gives the conclusion.
\end{proof}

\smallskip

Recall that Theorem \ref{corona2}  provides us with an integral representation for $(\lambda,k)$-Riesz limits.
In fact this representation comes from the integral on the left side of the equality given in
Lemma~\ref{rekto}. As discussed earlier we have to switch from the order $m+1$ to the order $k$ which will be done using  the
following  general device.

\smallskip

\begin{Lemm}\label{heute}
Let $D \in \mathcal{D}(\lambda)$  and  $0 < p < q$.  Then for all $w,s \in \C$ and all $x>0$
\[
\int_{0}^{\infty} S_{x}^{\lambda,q}(D)(s) w^{q+1}e^{-wx}  dx
=
\frac{\Gamma(q+1)}{\Gamma(p+1)}
\int_{0}^{\infty} S_{u}^{\lambda,p}(D)(s) w^{p+1}e^{-wu}  du\,.
\]
\end{Lemm}

\smallskip

\begin{proof}
Again we assume that $s=0$.
  We first use Lemma~\ref{harrie} (applied to $k=p$ and $\mu= q-p$) to show that
\begin{align*}
&
  \int_{0}^{\infty} S_{x}^{\lambda,q}(D)(0) e^{-wx}  dx
 \\&
    =
    \frac{\Gamma(q+1)}{\Gamma(p+1)\Gamma(q-p)}
    \int_{0}^{\infty}
  \int_{0}^{x} S_{u}^{\lambda,p}(D)(0) (x-u)^{q-p-1}
      du \,\,e^{-wx}dx
    \\&
    =
    \frac{\Gamma(q+1)}{\Gamma(p+1)\Gamma(q-p)}
    \int_{0}^{\infty}
     S_{u}^{\lambda,p}(D)(0)
  \int_{u}^{\infty} (x-u)^{q-p-1} e^{-wx}  dx\,\,
      du\,.
      \end{align*}
      Now we claim that
$$\int_{u}^{\infty} (x-u)^{q-p-1} e^{-wx}  dx = w^{p-q}e^{-wu}\Gamma(q-p).$$
Indeed, since both sides define holomorphic functions in $w$, it suffices to check for all $w=\sigma >0$. In this case we obtain by substitution that
\begin{align*}
&\int_{u}^{\infty} (x-u)^{q-p-1} e^{-wx}  dx=\int_{0}^{\infty} \alpha^{q-p-1} e^{-\sigma(\alpha+u)} d\alpha\\  &=e^{-\sigma u} \sigma^{p-q} \int_{0}^{\infty} \beta^{q-p-1} e^{-\beta} d\beta= e^{-\sigma u} \sigma^{p-q} \Gamma(q-p).
\end{align*}
Altogether
\[w^{q+1}\int_{0}^{\infty} S_{x}^{\lambda,q}(D)(0) e^{-wx}  dx = w^{p+1}\frac{\Gamma(q+1)}{\Gamma(p+1)}\int_{0}^{\infty} S_{u}^{\lambda,q}(D)(0) e^{-wu}  du  \,. \qedhere
\]
\end{proof}

\bigskip

\subsection{Estimates}
We  need several   estimates for Riesz means which also play a crucial role for the proofs of
Theorem~\ref{corona2}, Theorem~\ref{Bohr-Cahen-Riesz}, and Theorem~\ref{niceproof?}.

\smallskip

The first one allows to estimate Riesz means of higher order by  Riesz means of smaller orders.

\smallskip
\begin{Lemm} \label{3} Given $q>p\ge 0$, $s\in \mathbb{C}$ and $D \in \mathcal{D}(\lambda)$, we for all $x>0$ have
$$|R_{x}^{\lambda,q}(D)(s)|\le \sup_{0<y<x}|R_{y}^{\lambda,p}(D)(s)|.$$
\end{Lemm}
\begin{proof}
From Lemma~\ref{harrie} (with $k=p$ and $\mu =q-p$)  we get
$$R_{x}^{\lambda,q}(D)(s) = x^{-q} \frac{\Gamma(q+1)}{\Gamma(1+p)\Gamma(q-p)}  \int_{0}^{x} R^{\lambda,p}_{y}(D)(s)y^{p}(x-y)^{q-p-1} dy\,,$$
and then we conclude from \eqref{betafunctioncalc} that
\[
|R_{x}^{\lambda,q}(D)(s)|\le \sup_{0<y<x}|R_{x}^{\lambda,p}(D)(s)|. \qedhere
\]
\end{proof}

\smallskip

The second devise  is an analog of Lemma~\ref{Estimate0} for Riesz means of arbitrary orders.

\begin{Lemm} \label{alwaysabel} Let $D\in \mathcal{D}(\lambda)$ and $k>0$. Then for $s_{0}=\sigma+it \in [\re >0],s \in \C$ and $x>0$
\begin{equation*}
|R^{\lambda,k}_{x}(D)(s)|\le C(s_{0},k)e^{\sigma x} \sup_{0<y<x}|R^{\lambda,k}_{y}(D)(s+s_{0})|\,.
\end{equation*}
\end{Lemm}

\begin{proof}  Using translation, we may assume  without loss of generality that  $s=0$. Then  the choice $s=s_{0}$ and $w=-s_{0}$
in  Lemma \ref{Abelta} leads to
\begin{align*}
R_{x}^{\lambda,k}(D)(0)&=R_{x}^{\lambda,k}(D)(s_{0}-s_{0})\\&=e^{s_{0} x}R_{x}^{\lambda,k}(D)(s_{0})+\frac{(-1)^{m}}{(m+1)!} x^{-k} \int_{0}^{x} S^{\lambda,m+1}_{t}(D)(s_{0}) \partial^{m+2}h_{-s_{0}}(t) dt.
\end{align*}
where
\[
h_{-s_{0}}(t)= (e^{s_{0} t}-e^{s_{0} x})(x-t)^k.
\]
By Lemma \ref{3} we have
$$|R^{\lambda,m+1}_{t}(s_{0})|\le \sup_{0<y<t} |R^{\lambda,k}_{y}(s_{0})|\,,$$
and so
\begin{align*}
\Big|x^{-k} \int_{0}^{x} S^{\lambda,m+1}_{t}(s_{0}) \partial^{m+2}h_{-s_{0}}(t) dt\Big|\le \sup_{0<y<x} |R^{\lambda,k}_{y}(s_{0})| x^{-k}\int_{0}^{x} t^{m+1} \partial^{m+2}h_{-s_{0}}(t) dt,
\end{align*}
where
\begin{align} \label{productform}
\begin{split}
   \partial^{m+2}h_{-s_0}(t)=s_{0}^{m+2}e^{s_{0} t} (x-t)^{k}&+\sum_{j=1}^{m+1} s_0^{j}e^{s_{0} t} c_{k,j}(x-t)^{k-(m+2-j)}(t)\\ &+(e^{s_{0} t}-e^{s_{0} x}) c_{k,m+1} (x-t)^{k-(m+2)}.
\end{split}
\end{align}
We now take into account that
\begin{equation}\label{meanvalue}
  |e^{s_{0} t}-e^{s_{0} x}|
\leq
|e^{\sigma t}-e^{\sigma x}| \frac{|s_{0}|}{\sigma} \leq (x-t) \sigma
e^{\sigma t}  \frac{|s_{0}|}{\sigma} = (x-t)
e^{\sigma t}  |s_{0}|\,,
\end{equation}
  and hence
\begin{align*}
\Big|x^{-k}\int_{0}^{x} t^{m+1} \partial^{m+2}h_{-s_{0}}(t) dt\Big|
&
\le C(s_{0},k) e^{\sigma x} x^{-k}\sum_{j=1}^{m+2}\int_{0}^{x}t^{m+1}(x-t)^{k-(m+2-j)}dt\\ & \le C(s_{0},k) e^{\sigma x} \sum_{j=1}^{m+2} x^{m+2+j} \int_{0}^{x}t^{m+1}\Big(1-\frac{t}{x}\Big)^{k-(m+2-j)}dt\,.
\end{align*}
Finally, we calculate
\begin{align*}
\int_{0}^{x}t^{m+1}\Big(1-\frac{t}{x}\Big)^{k-(m+2-j)}dt&=x^{m+2}\int_{0}^{x}\Big(\frac{t}{x}\Big)^{m+1}\Big(1-\frac{t}{x}\Big)^{k-(m+2-j)} x^{-1}dt\\ &=x^{m+2} \int_{0}^{1} y^{m+1} (1-y)^{k-(m+2-j)} dy\,,
\end{align*}
and by \eqref{betafunctioncalc} with $x=1$
\begin{align*}
\int_{0}^{1} y^{m+1} (1-y)^{k-(m+2-j)} dy=\frac{\Gamma(m+2)\Gamma(k+j-m-1)}{\Gamma(k+j+1)}.
\end{align*}
This finishes the proof.
\end{proof}

\smallskip
The third tool we present, is in fact one of the decisive ingredients for the proof of Theorem~\ref{niceproof?}. See Section~\ref{Riesz means} for the definition of the summatory function $U$ of the second kind.

\begin{Lemm} \label{newnow} Let $ C=\sum a_n$  be a series, $\lambda$ a frequency, and $k >0$. Then
there is $C = C(k) >0$ such that for all $x>0$
\begin{equation*}
|U^{\lambda,k}_{x}(C)|\le C e^{kx} \sup_{0<y<x}|S^{\lambda,k}_{y}(C)|\,.
\end{equation*}
\end{Lemm}

\smallskip
The proof of Lemma~\ref{newnow} needs another estimate taken from  \cite[Lemma 8]{HardyRiesz}, and (as above) we for the sake of completeness again repeat
its argument.

\begin{Lemm}\label{Lemma8}
 Let $ C=\sum a_n$  be a real series, $\lambda$ a frequency, $\mu \ge 0$, and $0 < \kappa\leq 1  $.  Then for all  $0 < \xi < x$
\begin{align*}
\frac{\Gamma(\kappa + \mu +1)}{\Gamma(\kappa)\Gamma(\mu+1)}
\Big|  \int_{0}^{\xi} S^{\lambda,\mu}_{t}(C)  (x-t)^{\kappa-1}  dt \Big|  \leq   \sup_{0<t<\xi}|S^{\lambda,\kappa + \mu}_{y}(C) |\,.
\end{align*}
\end{Lemm}

\begin{proof}
  We  put
  \begin{equation*}
  c(\mu,\kappa) =\frac{\Gamma(\mu +1)}{\Gamma(\mu +\kappa+ 1)\Gamma(1-\kappa)}\,,
      \end{equation*}
     and show with  the second formula from Lemma~\ref{harrie} and Fubini's theorem that
  \begin{align*}
  &
     c(\mu,\kappa)^{-1} \frac{1}{\Gamma(1-\kappa)\Gamma(\kappa)}\int_{0}^{\xi} S^{\lambda,\mu}_{t}(C)  (x-t)^{\kappa-1}  dt
     \\&
     = c(\mu,\kappa)^{-1} \frac{1}{\Gamma(1-\kappa)\Gamma(\kappa)}\int_{0}^{\xi} S^{\lambda,(\mu+\kappa)-\kappa}_{t}(C) (x-t)^{\kappa-1}  dt
     \\&
     = \frac{1}{\Gamma(1-\kappa)\Gamma(\kappa)}\int_{0}^{\xi}
     \int_0^t \frac{d}{du} S_u^{\lambda,\mu +\kappa}(C)(t-u)^{-\kappa}  du  (x-t)^{\kappa-1}  dt
     \\&
     = \int_{0}^{\xi}
      \frac{d}{du}  S_u^{\lambda,\mu +\kappa}(C)\frac{1}{\Gamma(1-\kappa)\Gamma(\kappa)}\int_u^\xi (t-u)^{-\kappa}   (x-t)^{\kappa-1}  dt \,du
      \\&
     = \int_{0}^{\xi}
        \frac{d}{du}S_u^{\lambda,\mu +\kappa}(C)\Big(1- \frac{1}{\Gamma(1-\kappa)\Gamma(\kappa)}\int_\xi^x (t-u)^{-\kappa}   (x-t)^{\kappa-1}  dt\Big) du\,.
  \end{align*}
  Now define on the interval $[0,\xi]$ the function
  \[
   h(u) = \frac{1}{\Gamma(1-\kappa)\Gamma(\kappa)}\int_u^\xi (t-u)^{-\kappa}   (x-t)^{\kappa-1}  dt\,.
  \]
  Since
  \[
  \int_u^x (t-u)^{-\kappa}   (x-t)^{\kappa-1}  dt = \Gamma(1-\kappa)\Gamma(\kappa)\,,
  \]
  we see that
  \[
h(u) = 1 - \frac{1}{\Gamma(1-\kappa)\Gamma(\kappa)}\int_\xi^x (t-u)^{-\kappa}   (x-t)^{\kappa-1}  dt\,.
  \]
  Moreover, on $[0,\xi]$ the function $(t-u)^{-\kappa}$ increases  in $u$, implying that $h$ on $[0,\xi]$ defines a positive decreasing function, always less than $1$. By the second mean value theorem there is $0 \leq \eta \leq \xi$ such that
  \begin{align*}
        \frac{\Gamma(\kappa + \mu +1)}{\Gamma(\kappa)\Gamma(\mu+1)}
  \int_{0}^{\xi} &S^{\lambda,\mu}_{t}(C)(x-t)^{\kappa-1}  dt
  \\&
  = h(0) \int_{0}^{\eta}
      \partial_u  S_u^{\lambda,\mu +\kappa}(C)du = h(0) S_\eta^{\lambda,\mu +\kappa}(C)\,,
  \end{align*}
    which completes the argument.
    \end{proof}

\smallskip
We still need one more tool for the proof of Lemma~\ref{newnow}.

\begin{Lemm} \label{todo} Let $C=\sum a_n$  be a real series, $\lambda$ a frequency, and  $k >0$ . Then there
$c= c(k)$ such that  for all $0 < \xi < x$
\begin{equation*}
\Big|  \int_{0}^{\xi} S^{\lambda,0}_{t}(C)  (x-t)^{k-1}  dt \Big|  \leq c  \sup_{0<t<\xi}|S^{\lambda,k}_{t}(C)|\,.
\end{equation*}
\end{Lemm}

\begin{proof}
Choose $m \in \mathbb{N}_0$ such that $0< k-m \leq 1$. By \eqref{B} and partial integration  we have
\begin{align*}
  \int_{0}^{\xi} S^{\lambda,0}_{t}(C)  (x-t)^{k-1} dt
  &
  =
  c(k)  \int_{0}^{\xi} \partial^m_t \Big[ S^{\lambda,m}_{t}(C)\Big]  (x-t)^{k-1} dt
  \\&
  =
  c(k)  \int_{0}^{\xi}  S^{\lambda,m}_{t}(C)  (x-t)^{(k-1)-m} dt\,.
  \end{align*}
 On the other hand we deduce from  Lemma~\ref{Lemma8} that
 \[
 \Big| \int_{0}^{\xi} S^{\lambda,m}_{t}(C) (x-t)^{(k-m)-1} dt   \Big|
 \leq c(k) \sup_{0<t<\xi}|S^{\lambda,k}_{t}(C)|\,. \qedhere
 \]
  \end{proof}

  \smallskip
  Finally, we are ready for the proof of Lemma~\ref{newnow}.

\begin{proof}[Proof of Lemma~\ref{newnow}]
  Without loss of generality we assume that $C = \sum a_n$ is a real series. By \eqref{A} we have
  \begin{align*}
   U^{\lambda,k}_{x}(D)
   &
   = k\int_{0}^{x} S^{\lambda,0}_{t}(C)  (e^x-e^t)^{k-1} e^t dt
\\&
=k\int_{0}^{x} S^{\lambda,0}_{t}(C) (x-t)^{k-1} \Big(\frac{e^x-e^t}{x-t}\Big)^{k-1} e^t dt\,.
  \end{align*}
  The function $\Big(\frac{e^x-e^\bullet}{x-\bullet}\Big)^{k-1}e^\bullet$ is positive and increasing on $[0,x]$ (for $k \ge 1$
  this is obvious since then $\frac{e^x-e^\bullet}{x-\bullet}$ increases, and for $0<k<1$ differentiate), and the limit
  as $t$ tends to $x$ equals $e^{kx}$. Then by the second mean value theorem there is $0<\xi <x$
\begin{align*}
  U^{\lambda,k}_{x}(C)
    =  ke^{kx}\int_{\xi}^{x}S^{\lambda,0}_{t}(C)   (x-t)^{k-1}  dt
    =ke^{kx}\Big(\int_{0}^{x}-\int_{0}^{\xi}\Big)S^{\lambda,0}_{t}(C)   (x-t)^{k-1}\,.
\end{align*}
 Hence the conclusion follows from Lemma~\ref{todo}.
\end{proof}

\smallskip

\subsection{The technical heart} \label{technicalsection}

The following two lemmas finish our preparation of the proofs of  Theorem~\ref{Perron0} and Theorem~\ref{Bohr-Cahen-Riesz}.
We believe that they in a sense form the 'technical heart' of much of the theory of Riesz summation as created in \cite{HardyRiesz}.
In fact, they will be used at various places of this work (see e.g. the proof of one of our main contributions, Theorem~\ref{Perronestimate}).

Basically, we execute the indicated strategy explained in the previous sections, by carefully analyzing the formula given in Lemma \ref{rekto}.

\smallskip
\begin{Lemm} \label{2}  Let
$D \in \mathcal{D}(\lambda)$, $k>0$ and $m\in \N_{0}$ such that $m< k\le  m+1$. Then there is a constant $C=C(m,k)$ such that for all
$w=\sigma+i\tau \in [\re>0]$, all $s\in \mathbb{C}$ and all $x>0$
\begin{align*}
\Big|\Gamma(m+2)^{-1}&\int_{0}^{x}S_{t}^{\lambda,m+1}(D)(s) w^{m+2}(x-t)^{k} e^{-wt} dt- S_{x}^{\lambda,k}(D)(s+w)\Big| \\ &\le e^{-\sigma x} |S_{x}^{\lambda,k}(D)(s)|+ C  \sum_{j=1}^{m+1} |w|^{j} \int_{0}^{x} |S_{y}^{\lambda,k}(D)(s)|e^{-\sigma y} (x-y
)^{j-1}  dy.
\end{align*}
In particular, there is a constant $L=L(m,k)$ such that for all
$w=\sigma+it \in [\re~>~0]$ with $|\arg(w)|\le \gamma$ and all $x >0$
\begin{align*}
\Big|\Gamma(m+&2)^{-1}\int_{0}^{x}S_{t}^{\lambda,m+1}(D)(0) w^{m+2}(x-t)^{k} e^{-wt} dt- S_{x}^{(\lambda,k)}(D)(w)\Big| \\ &\le e^{-\sigma x}|S_{x}^{\lambda,k}(D)(0)|+ L \sum_{j=1}^{m+1} |\sec(\gamma)|^{j}  \int_{0}^{x} |S_{y}^{\lambda,k}(D)(0)| y^{-j}(x-y
)^{j-1}  dy.
\end{align*}
\end{Lemm}

\begin{proof}
We first handel the more complicated case $k< m+1$, and at the end of this proof we comment on the easier case $k= m+1$.

If $m <k< m+1$, then we know from  Lemma~\ref{rekto}  that for all $s,w\in \C$ and all $x >0$
\begin{align*}
S_{x}^{\lambda,k}(D)(s+w) &
 -
\Gamma(m+2)^{-1}\int_{0}^{x} S_{t}^{\lambda,m+1}(D)(s) w^{m+2} e^{-wt}(x-t)^kdt
\\&
= -e^{-wt}  S_{x}^{\lambda,k}(D)(s)
\\&
\,\,\,\,\,\,\,
+  C(k) \int_0^x  S_{y}^{\lambda,k}(D)(s)  \int_y^x (t-y)^{m-k} (g_1(t) + g_2(t)) dt dy\,,
\end{align*}
where
\begin{align*}
&
  g_1(t) = \sum_{j=1}^{m+1} \binom{m+2}{j} (-w)^j e^{-wt}  \,\big[(x-\bullet)^k\big]^{(m+2-j)}(t)
  \\&
  g_2(t) = (e^{-wt} - e^{-wx}) \,\big[(x-\bullet)^k\big]^{(m+2)}(t)
    \,.
\end{align*}
We calculate (as in \eqref{productform})
\begin{align} \label{spd}
g_1(t) + g_2(t)
&
=\sum_{j=1}^{m+1} (-w)^{j}e^{-w t} c_{k,j}(x-t)^{k-(m+2-j)}
\\&
\,\,\,\,\,\,+(e^{-w t}-e^{-w x}) c_{k,m+1} (x-t)^{k-(m+2)}\,.
\end{align}
Then  it remains to control the integral
 \begin{equation*}
    C(k) \int_0^x  S_{y}^{\lambda,k}(D)(s)  \int_y^x (t-y)^{m-k} (g_1(t) + g_2(t)) dt dy\,.
 \end{equation*}
  For $i=1$ we use \eqref{betafunctioncalc} (with $x=1$) and \eqref{againagain} to obtain with $w=\sigma+i\tau\in [\re>0]$
\begin{align*}
& \Big|\int_{y}^{x} (t-y)^{m-k} g_{2}(t) ~dt\Big|\\ &=\Big|\sum_{j=1}^{m+1} (-1)^{j} w^{j} c_{k,j} (x-y)^{j-1} e^{-wy} \int_{0}^{1} e^{-w\beta(x-y)} \beta^{m-k} (1-\beta)^{k-(m+2-j)} d\beta\Big|
\\ &
\le C_{1}(m,k) \sum_{j=1}^{m+1} |w|^{j} (x-y)^{j-1} e^{-\sigma y}
\int_{0}^{1}  \beta^{m-k} (1-\beta)^{k-(m+2-j)} d\beta\\&=C_{1}(m,k)\sum_{j=1}^{m+1} \frac{\Gamma(1+m-k)\Gamma(k-m-1+j)}{\Gamma(j)} |w|^{j} (x-y)^{j-1} e^{-\sigma y}\,.
\end{align*}
For $i=2$ we claim that
\begin{align*}
&\Big|\int_{y}^{x} (t-y)^{m-k} g_{3}(t) ~dt\Big|\le C(m,k)|w|e^{-\sigma y}\,.
\end{align*}
Indeed,  by the mean value theorem (as in \eqref{meanvalue}) for $0 \leq t \leq x$ and $w \in [\re>0]$
\begin{equation} \label{cdu}
  |e^{-w t}-e^{-w x}|
\leq
|e^{-\sigma t}-e^{-\sigma x}| \frac{|w|}{\sigma} \leq (x-t) \sigma
e^{-\sigma t}  \frac{|w|}{\sigma} = (x-t)
e^{-\sigma t}  |w|\,.
\end{equation}
Then (again using \eqref{betafunctioncalc} with $x=1$)
\begin{align*}
&\Big|\int_{y}^{x} (t-y)^{m-k} g_{3}(t) ~dt\Big|\\&=\Big|c_{k,m+1}\int_{y}^{x} (t-y)^{m-k} (e^{-w t}-e^{-w x})  (x-t)^{k-(m+2)} dt\Big|\\ &\le c_{k,m+1} \int_{y}^{x} (t-y)^{m-k} |e^{-\sigma t}-e^{-\sigma x}| \frac{|w|}{\sigma} (x-t)^{k-(m+2)} dt\\ &\le  e^{-\sigma y} |w| c_{k,m+1} \int_{y}^{x} (t-y)^{m-k} (x-t)^{k-m-1} dt\\&=
e^{-\sigma y} |w| c_{k,m+1}
(x-y)^{m-k +(k-m-1)+1}
\int_{0}^{1} \beta^{m-k} (1-\beta)^{k-m-1} d\beta\\\ & =
C_{2}(m,k)\frac{\Gamma(1+m-k)\Gamma(k-m)}{\Gamma(1)} e^{-\sigma y} |w| .
\end{align*}
Putting everything together completes the argument of the first part. For the second, take $s=0$, and observe that $e^{\sigma y}\leq \frac{(\sigma y)^j}{j!}$
and $\frac{|w|}{\sigma} \leq  \sec(\gamma)$\,.
Finally, we as announced comment on the simpler case $m < k =m+1$. In this case we know from Lemma~\ref{Abelta} that
\begin{align*}
S_{x}^{\lambda,r}(D)(s+w) &
 -
\frac{1}{\Gamma(m+2)}\int_{0}^{x} S_{t}^{\lambda,m+1}(D)(s) w^{m+2} e^{-wt} r(x,t)dt
\\&
= -e^{-wt}  S_{x}^{\lambda,r}(D)(s)
+  \int_0^x  S_{t}^{\lambda,m+1}(D)(s)  (g_1(t) + g_2(t)) dt \,.
\end{align*}
As above,  the claim is an immediate consequence of \eqref{spd} and \eqref{cdu}.
\end{proof}

\smallskip

Roughly speaking, the integral representation \eqref{laplace3} (from Theorem \ref{corona2}) for $(\lambda,k)$-Riesz limits is determined by the first summand of the left-hand side of the first inequality of Lemma \ref{2}, tending $x\to \infty$. The next lemma proves existence of this limit -- given a growth condition on the summatory function.

\smallskip

\begin{Lemm} \label{lemma-limit}
Given $D \in \mathcal{D}(\lambda)$, let $s \in \mathbb{C}$ and $k \ge 0$, $q > 0$.
Then for all $w \in \mathbb{C}$ and $\varepsilon >0$ such that $0 < \varepsilon < \re w $ and
\[
|S_{t}^{\lambda,q}(D)(s)|\le C(q,\varepsilon) t^q e^{\varepsilon t}\,,\,\,\,\,t >0\,,
\]
we have
\begin{equation} \label{limit}
\lim_{x\to \infty} \int_{0}^{x}S_{t}^{\lambda,q}(D)(s) \Big(1- \frac{t}{x}\Big)^{k} w^{q+1}e^{-wt} dt=\int_{0}^{\infty} S_{y}^{\lambda,q}(D)(s) w^{q+1}e^{-wy}  dy\,.
\end{equation}
Moreover, the convergence in \eqref{limit} is
uniform on each cone  $|\arg (w-s)|\le \gamma < \frac{\pi}{2}$, whenever
$$|S_{t}^{\lambda,q}(D)(s)|\le C(q) t^q\,,\,\,\,\,t >0\,.$$
\end{Lemm}

\begin{proof}
We assume without loss of generality that $s=0$. In order to prove~\eqref{limit}, we fix $w$ and $\varepsilon$, and put $\sigma = \re w -\varepsilon$. Then for all $x,t$
\begin{align*}
\Big|S_{t}^{\lambda,q}(D)(0) \Big(1- \frac{t}{x}\Big)^{k} w^{q+1}e^{-wt} \chi_{[0,x]}(t)\Big|
\leq
C(q, \varepsilon) |w|^{q+1} e^{-\sigma t}t^q \,,
\end{align*}
and hence~\eqref{limit} is an immediate consequence of the dominated convergence theorem.

The proof of the 'moreover-part'
is similar: For each $w$ with $|\arg (w)|\le \gamma < \frac{\pi}{2}$ we abbreviate $\sigma_w = \re w$.
Then for all $x,t,w$
\begin{align*}
\Big|S_{t}^{\lambda,q}(D)(0) \Big(1- \frac{t}{x}\Big)^{k} w^{q+1}e^{-wt} \chi_{[0,x]}(t)\Big|
\leq
C(q) \sec{(\gamma)} t^q \sigma_w^{q+1} e^{-\sigma_w t}\,,
\end{align*}
and hence by the substitution $u = \sigma_w t$ we get for all $x,u,w$
\begin{align*}
\Big|S_{\frac{u}{\sigma_w}}^{\lambda,q}(D)(0) \Big(1- \frac{\frac{u}{\sigma_w}}{x}\Big)^{k} w^{q+1}e^{-w\frac{u}{\sigma_w}} \chi_{[0,x]}(\frac{u}{\sigma_w})\Big| \frac{1}{\sigma_w}
\leq
C(q) \sec{(\gamma)} u^q  e^{-u}\,.
\end{align*}
Hence again by the dominated convergence theorem (uniform variant)
\[
\int_0^{x\sigma_w}
S_{\frac{u}{\sigma_w}}^{\lambda,q}(D)(0) \Big(1- \frac{\frac{u}{\sigma_w}}{x}\Big)^{k} w^{q+1}e^{-w\frac{u}{\sigma_w}}
\frac{1}{\sigma_w} du
\to
\int_0^{\infty}
S_{\frac{u}{\sigma_w}}^{\lambda,q}(D)(0)  w^{q+1}e^{-w\frac{u}{\sigma_w}}
\frac{1}{\sigma_w} du
\]
uniformly in $w$ as $x \to \infty$. Substituting back, finishes the proof. \end{proof}

\bigskip
\subsection{Proof of  Theorem~\ref{corona2}}

\begin{proof}[{\bf Proof of Theorem~\ref{corona2}}] Let $k>0$ and $m\in \N_{0}$ such that $m< k\le m+1$. If $D$ is $(\lambda,k)$-summable at $s_{0}=\sigma_{0}+i\tau_{0}$, then for some constant $C=C(s_{0},k)$ we by Lemma \ref{alwaysabel} have that for all $x >0$
\begin{equation} \label{lisa}
|R_{x}^{\lambda,k}(D)(0)|\le C e^{\sigma_{0} x}.
\end{equation}
We fix $s = \sigma + i\tau \in [\re>\sigma_{0}]$, and deduce from Lemma \ref{lemma-limit}  and Lemma~\ref{heute} that
\begin{align*}
\lim_{x\to \infty} \frac{x^{-k}}{\Gamma(m+2)}&\int_{0}^{x}S_{t}^{\lambda,m+1}(D)(0) s^{m+2}(x-t)^{k} e^{-st} dt
\\ &=\Gamma(1+k)^{-1}s^{k+1}\int_{0}^{\infty} S_{y}^{\lambda,k}(D)(0) e^{-st}  dy.
\end{align*}
According to Lemma~\ref{2} we have to show that the term
\[
 e^{-\sigma x} |R_{x}^{(\lambda,k)}(D)(0)|+ C  \sum_{j=1}^{m+1} x^{-k} |s|^{j}\int_{0}^{x} |S_{y}^{\lambda,k}(D)(0)|e^{-\sigma y} (x-y
)^{j-1}  dy
\]
vanishes as $x \to \infty$. Observe that \eqref{lisa} implies the claim for the first summand. Writing $\varepsilon= \sigma-\sigma_{0}>0$ we use again  \eqref{lisa} to see that for $j=1,\ldots, m+1$
\begin{align*}
x^{-k} \int_{0}^{x} |S_{y}^{\lambda,k}(D)(0)|e^{-\sigma y} (x-y
)^{j-1}  dy
&
\le C x^{-k}\int_{0}^{x} e^{-\varepsilon y} y^{k}(x-y)^{j-1} dy
\\&
\le C x^{m-k} \int_{0}^{\infty} e^{-\varepsilon y} y^{k} dy,
\end{align*}
which vanishes as $x\to \infty$, since $k>m$.

It remains to prove the 'moreover-part'. We may assume without loss of generality that  $\lambda_{1}=0$,
 $s_0 =0$ and $\lim_{x\to \infty} R_{x}^{\lambda,k}(D)(0)=0$; if this particular case is settled and $s_{0}$ is arbitrary, then  we consider the modified Dirichlet series $\sum b_n e^{-s_0 \lambda_n} e^{-s \lambda_n}$ with $b_{1}=a_{1}-\lim_{x\to \infty} R_{x}^{\lambda,k}(D)(s_{0})$ and $b_{n}=a_{n}$, $n\ge 2$, which then
 converges uniformly on  each cone $|\arg (s)|\le \gamma < \frac{\pi}{2}$, and so implies the conclusion.

 So let $s_0 =0$. According to  Lemma \ref{2} and Lemma \ref{lemma-limit}, we have to check that uniformly for all $s\in [\re>0]$ with $\arg(s)\le \gamma<\frac{\pi}{2}$
\begin{equation} \label{duschen}
\lim_{x\to \infty} \sum_{j=1}^{m+1} |\sec(\gamma)|^{j}  x^{-k}  \int_{0}^{x}  |S^{\lambda,k}_{y}(D)(0)|   y^{-j} (x-y
)^{j-1}  dy=0.
\end{equation}
Therefore, choose  $\varepsilon>0$ and $x_{0}>1$  such   that $|R_{x}^{\lambda,k}(D)(0)|\le \varepsilon$ for all $x>x_{0}$.
We split the integral
\begin{align*}
&x^{-k}  \int_{0}^{x}  |S^{\lambda,k}_{y}(D)(0)|   y^{-j} (x-y
)^{j-1}  dy\\ &=x^{-k}\int_{0}^{x_0}  |S^{\lambda,k}_{y}(D)(0)|   y^{-j} (x-y
)^{j-1}  dy+x^{-k}\int_{x_{0}}^{x}  |S^{\lambda,k}_{y}(D)(0)|   y^{-j} (x-y
)^{j-1}  dy.
\end{align*}
Then note that for $1\le j \le m+1$ and $x\ge 2x_{0}$
\begin{align*}
&x^{-k}\int_{0}^{x_0}  |S^{\lambda,k}_{y}(D)(0)|   y^{-j} (x-y
)^{j-1}  dy \le C(D,x_{0}) x^{-k}\int_{0}^{x_{0}} y^{k-j}(x-y)^{j-1} dy
\\&
\le C(D,x_{0}) x^{m-k} \int_{0}^{x_{0}} y^{k-j} dy
\leq C(D,x_{0})x^{m-k}   x_{0}^{k+1-j},
\end{align*}
which vanishes as $x\to \infty$, since $k>m$, and moreover for $x\ge 2x_{0}$
\begin{align*}
&x^{-k}\int_{x_0}^x  |S_{y}^{\lambda,k}(D)(0)|   y^{-j} (x-y
)^{j-1}  dy
\le x^{-k}\varepsilon \int_{x_{0}}^{x}y^{k-j}(x-y)^{j-1} dy \\ &\le \varepsilon x^{-k} x^{k-j} \int_{x_{0}}^{x} (x-y)^{j-1} dy
= \varepsilon x^{-j} j^{-1}(x-x_{0})^{j}\le \varepsilon \Big(1-\frac{x_{0}}{x}\Big)^{j}\le \varepsilon.
\end{align*}
This finishes the proof of \eqref{duschen}.
\end{proof}

\subsection{Proof of  Theorem~\ref{Bohr-Cahen-Riesz}}

\begin{proof}[{\bf Proof of Theorem \ref{Bohr-Cahen-Riesz}}]
Assume that
$$\sigma_{0}>\limsup_{x\to \infty} x^{-1}\log( |R_{x}^{\lambda,k}(D)(0)|)=:L.$$
 Then \eqref{lisa} is satisfied, and from the previous proof of Theorem \ref{corona2}
 (using again  Lemma \ref{lemma-limit}, Lemma~\ref{heute} and Lemma~\ref{2})
  we see that $D$ is $(\lambda,k)$-summable on $[\re>\sigma_{0}]$.  This proves that $L \leq \sigma^{\lambda,k}_{c}(D)$.

  Assume conversely that $\sigma^{\lambda,k}_{c}(D)\ge 0$, and let $\varepsilon>0$.
  We define $\sigma_{0}=\sigma^{\lambda,k}_{c}(D)+\varepsilon$. Then $D$ is $(\lambda,k)$-summable at $\sigma_{0}$ and so by Lemma \ref{alwaysabel} we have
$$|R_{x}^{\lambda,k}(D)(0)|\le C(\sigma_{0})e^{\sigma_{0}x},~ x>0.$$
Consequently,  $L\le \sigma_{0}=\sigma^{\lambda,k}_{c}(D)+\varepsilon$, and hence as desired  $L= \sigma^{\lambda,k}_{c}(D)$.
\end{proof}

\subsection{Proof of  Theorem~\ref{basic}}

    \begin{proof}[{\bf Proof of Theorem~\ref{basic},(i)}]
    Without loss of generality we may assume that the $(\lambda,k)$-Riesz sum of $D$ is $0$\,, i.e
    $
    S_x^{\lambda,k}(D) = o(x^k)\,.
    $
    Since by Lemma~\ref{harrie} for all $x$
    \begin{equation*}
S_x^{\lambda,\ell}(D)  =\frac{\Gamma(\ell +1)}{\Gamma(k + 1)\Gamma(\ell-k)}
\int_0^x S_u^{\lambda,k}(D) (x-u)^{\ell-k-1}  du\,\,,
\end{equation*}
 we check that
\[
\int_0^x S_u^{\lambda,k}(D) (x-u)^{\ell-k-1}  du =   o(x^\ell)\,.
\]
Fix some $\varepsilon >0$, and choose $x_0$ such that for all $x > x_0$ we have $S_x^{\lambda,k}(D) < \varepsilon x^k$.
Then, using \eqref{betafunctioncalc}, for all $x > x_0$
\begin{align*}
&
  \int_0^x S_u^{\lambda,k}(D) (x-u)^{\ell-k-1}  du
\\&
 =
  \int_0^{x_0} S_u^{\lambda,k}(D) (x-u)^{\ell-k-1}  du +\int_{x_0}^x S_u^{\lambda,k}(D) (x-u)^{\ell-k-1}  du
  \\&
  \leq
  \sup_{u < x_0} S_u^{\lambda,k}(D) \int_0^{x_0}  (x-u)^{\ell-k-1}  du +
  \varepsilon \int_{x_0}^x u^{k} (x-u)^{\ell-k-1}  du
  \\&
  \leq
  \sup_{u < x_0} S_u^{\lambda,k}(D) \int_0^{x_0}  (x-u)^{\ell-k-1}  du +
  \varepsilon \int_{0}^x u^{k} (x-u)^{\ell-k-1}  du
  \\&
  \leq
  \sup_{u < x_0} S_u^{\lambda,k}(D)  \,\, x_0^{\ell-k-1}  +
  \varepsilon x^\ell \frac{\Gamma(\ell +1)\Gamma(\ell-k)}{\Gamma(k + 1)}\,,
\end{align*}
the conclusion.
          \end{proof}

\smallskip

    \begin{proof}[{\bf Proof of Theorem~\ref{basic},(ii)}]
  Without loss of generality we may assume that the $(e^\lambda,k)$-Riesz sum of $D$ is zero. Hence we know by assumption from
  Lemma~\ref{AbelfirstcaseB}, \eqref{A} that
  \begin{equation} \label{wissen}
    U_x^{\lambda,k}(D) = k\int_0^x S_t^{\lambda,0}(D) \,\, (e^x-e^t)^{k-1} e^t dt = o(e^{kx})\,,
  \end{equation}
    and the job is to show that
    \begin{equation}
    \label{wollen}
    S_y^{\lambda,k}(D) = k\int_0^y S_s^{\lambda,0}(D)\, (y-s)^{k-1}  ds = o(y^k)\,.
         \end{equation}
Substitution with   $s = \log t$ and $y = \log x$ gives that \eqref{wollen}
in fact is equivalent to
\begin{equation}
    \label{wollenA}
    S_{\log x}^{\lambda,k}(D) = k \int_1^x S_t^{e^\lambda,0}(D)\,\,  \frac{(\log x- \log t)^{k-1}}{t}  dt = o(\log^k x)\,.
         \end{equation}
 Indeed, the proof in \cite[p. 32]{HardyRiesz}  proves this for the two  cases  $k \in \mathbb{N}$ and  $0 < k < 1$ separately, and sketches an  argument for the case $k>1, k \notin \mathbb{N}$. We, in a first step,  for the case $0 < k < 1$  follow the proof from \cite{HardyRiesz}, and modify it in a second step to verify the general case $k \ge 1$.

 \smallskip

    \noindent
The case $0 <k < 1$: By Theorem~\ref{basic},(i) we know that
$
S_x^{e^\lambda, 1} (D) = o(x)\,,
$
hence, given $\varepsilon >0$, there is $x_0 > 1$ such that for all $x >x_0$
\begin{equation}\label{firth}
  S_x^{e^\lambda, 1} (D) \leq \varepsilon x.
\end{equation}
For $x > 3 x_0$ we split the integral from \eqref{wollenA} into three pieces,
\begin{equation}\label{three-int}
  \int_1^x S_t^{e^\lambda,0}(D)\,\,  \frac{(\log x- \log t)^{k-1}}{t}  dt
=
\int_1^{x_0} + \int_{x_0}^{x/3} + \int_{x/3}^x = J_1 +J_2+ J_3\,,
\end{equation}
 and estimate  each integral separately.

\smallskip

\noindent Integral $J_1:$
Clearly,
\[
J_1
\leq \Big[\sup_{t \leq x_0} S_t^{e^\lambda,0}(D)\Big] x_0  \log^{k-1}  \Big(\frac{x}{x_0} \Big)= o(\log^k(x))\,,
\]
where we use that here $\log^{k-1}  \big(\frac{x}{t} \big) \leq    \log^{k-1}  \big(\frac{x}{x_0} \big)$, since
$t \leq x_0 \leq x$ and $k-1 <0$.\\

\noindent Integral $J_2:$  Using \eqref{B} and integrating by parts we have
\begin{align*}
  J_2
    =  \log^{k-1}(3) & \,\frac{3}{x}S_{\frac{x}{3}}^{e^\lambda,1}(D)
-  \log^{k-1}  \Big(\frac{x}{x_0} \Big) \,\frac{1}{x_0}S_{x_0}^{e^\lambda,1}(D)
\\&
  +\int_{x_0}^{x/3}
  S_{t}^{e^\lambda,1}(D)
  \Big[  (k-1)
   \log^{k-2}  \Big(\frac{x}{t} \Big)+  \log^{k-1}  \Big(\frac{x}{t} \Big)
  \Big] \frac{dt}{t^2}\,.
  \end{align*}
In absolute value the first two terms are less than an absolute  constant times $\varepsilon$
(we use \eqref{firth}, and the fact that $\log(x/x_0) > \log 3 >1$ since $x_0 < x/3$). The integral in the preceding equality
we estimate from above by

\begin{align*}
  \varepsilon k
  \int_{x_0}^{x/3}
     \log^{k-1}  \Big(\frac{x}{t} \Big) \frac{dt}{t} = \varepsilon k\Big(\frac{1}{k}\log^k(3)- \frac{1}{k}\log^k\big(\frac{x}{x_0}\big)\Big)\,,
\end{align*}
using again \eqref{firth} together with the fact that  $\log \big(\frac{x}{t} \big) > \log \big(\frac{3t}{t} \big) = \log(3) > 1$
for $x_0 <  t < x/3$, implying
$\log^{k-2}  \big(\frac{x}{t} \big)  < \log^{k-1}  \big(\frac{x}{t} \big)$.
All together we have proved that   $J_2 = o(\log^{k}(x))$.\\

\noindent Integral $J_3:$   By the second mean value theorem there is $x/3 \leq \xi \leq x$ such that
\begin{align*}
  J_3
  &
  = \int_{x/3}^x S_t^{e^\lambda,0}(D)    \log^{k-1}  \Big(\frac{x}{t} \Big) \frac{1}{t}   dt
  \\&
= \frac{3}{x}\int_{x/3}^\xi S_t^{e^\lambda,0}(D)    \log^{k-1}  \Big(\frac{x}{t} \Big)   dt
\\&
= \frac{3}{x} \int_{x/3}^\xi S_t^{e^\lambda,0}(D)    (x-t)^{k-1}    \Big(\frac{\log x -\log t}{x-t}\Big)^{k-1} dt\,.
\end{align*}
But the function $\Big(\frac{\log x -\log \bullet}{x- \bullet}\Big)^{k-1}$ is increasing on $[1,x]$ with limit $x^{1-k}$
as $t$ tends to~$x$ (the quotient is decreasing and $k-1 <0$). Hence another application of the second mean value theorem shows that there is
$x/3 \leq \xi_1\leq \xi \leq x$ such that
\begin{align*}
  |J_3|
  &
  = \frac{3}{x} \Big(\frac{\log x -\log \xi_1}{x-\xi_1}\Big)^{k-1}\Big|\int_{x/3}^\xi S_t^{e^\lambda,0}(D)    (x-t)^{k-1}    dt\Big|
  \\&
  \leq  3 x^{-k} \Big|\int_{\xi_1}^\xi S_t^{e^\lambda,0}(D)    (x-t)^{k-1}     dt\Big|\,.
\end{align*}
Now by  Lemma~\ref{todo} (write $\int_{\xi_1}^\xi = \int_{0}^\xi - \int_{0}^{\xi_1}$) there is some constant $c = c(k)>0$ for which
\begin{equation*}
\Big|  \int_{\xi_1}^{\xi} S^{e^\lambda,0}_{t}(D)  (x-t)^{k-1}  dt \Big|  \leq c  \sup_{x/3<t<x}|S^{e^\lambda,k}_{t}(D)|\,.
\end{equation*}
Since by assumption $S^{e^\lambda,k}_{x}(D) = o(x^k)$, we see that
$
J_3 = o(1) = o(\log^k(x))\,.
$
Collecting all estimates we got for the three integrals $J_1$, $J_2$, and $J_3$, we in fact proved  (what we were aiming at
in \eqref{wollenA})
\[
\int_1^x S_t^{e^\lambda,0}(D)\,\,  \frac{(\log x- \log t)^{k-1}}{t}  dt = o(\log^k x)\,.
\]

It remains to consider the case $k \ge1$, which is going to follow from a refinement of  the preceding arguments: Choose some $m \in \mathbb{N}$ such that $m < k\leq m+1$, and note again that by Theorem~\ref{basic},(i) we have
$
S_x^{e^\lambda, m+1} (D) = o(x)\,.
$
Hence, given $\varepsilon >0$, there is $x_0 > 1$ such that for all $x >x_0$
\begin{equation}\label{firthB}
  S_x^{e^\lambda, m+1} (D) \leq \varepsilon x.
\end{equation}
For $x > 3 x_0$ we again  consider the three integrals $J_1, J_2$, and $J_3$ from  \eqref{three-int},
and handel each piece separately in order to show \eqref{wollenA} .

\smallskip

\noindent Integral $J_1:$
Clearly,
\[
J_2
\leq \Big[\sup_{t \leq x_0} S_t^{e^\lambda,0}(D)\Big] x_0  \log^{k-1} (x)= o(\log^k(x))\,,
\]
since now $k-1 \ge 0$.\\

\noindent Integral $J_2:$  We have that for all $1\leq \ell \leq m+1$
\[
\frac{d^{\ell}}{dt}  \frac{(\log x - \log t)^{k-1}}{t}
=  \frac{1}{t^{\ell+1}} \sum_{j=0}^{\ell} c_j(k) (\log x - \log t)^{k-1-j}\,.
\]
Integrating $m+1$ times  by parts (again using \eqref{B}) we have
\begin{align*}
  J_2
    = S_t^{e^\lambda,m+1}(D) \frac{1}{t^{m+1}}
    &
    \sum_{j=0}^{m} c_j(k) \log^{k-1-j}\Big(\frac{x}{t}\Big)\bigg|_{x_0}^{x/3}
\\&
  +\int_{x_0}^{x/3}
  S_{t}^{e^\lambda,{m+1}}(D)
  \frac{1}{t^{m+2}} \sum_{j=0}^{m+1} c_j(k) \log^{k-1-j}\Big(\frac{x}{t}\Big)dt\,.
  \end{align*}
      As above we have that $\log\big(\frac{x}{t}\big) > \log (3) >1$  for $x_0 \leq t \leq x/3$, and hence
      \begin{align*}
        \log^{k-1}\Big(\frac{x}{t}\Big)
        &
        \ge \ldots \ldots\ge \log^{(k-1)-(m-1)}\Big(\frac{x}{t}\Big)
        \\&
        \ge \log^{(k-1)-m}\Big(\frac{x}{t}\Big)
        \ge \log^{(k-1)-(m+1)}\Big(\frac{x}{t}\Big)\,.
      \end{align*}
     Together with \eqref{firthB} this  shows that  the first term in the above formula for $J_2$ is
     less than a constant (only depending on $k$) times $\varepsilon \log^{k-1}(x)$, and the second term
     is
     less than a constant (only depending on $k$) times
\begin{align*}
  \varepsilon
  \int_{x_0}^{x/3}
     \log^{k-1}  \Big(\frac{x}{t} \Big) \frac{dt}{t} \leq \log^{k-1} (x) \int_{x_0}^{x/3}
    \frac{1}{t} dt = o(\log^{k}(x))
\end{align*}
(again taking into account that $k -1 \ge 0 $).
All together we get   $J_2 = o(\log^{k}(x))$.\\

\noindent Integral $J_3:$   We write
\begin{align*}
  J_3
  = \int_{x/3}^x S_t^{e^\lambda,0}(D)    (x-t)^{k-1}    \Big(\frac{\log x -\log t}{x-t}\Big)^{k-1} \frac{1}{t}dt\,.
\end{align*}
Since $k \ge 1$,  the function $\Big(\frac{\log x -\log \bullet}{x- \bullet}\Big)^{k-1} \frac{1}{\bullet}$ is decreasing on $[1,x]$ (look at the graph of $\log \bullet$), and
\[
\Big(\frac{3 \log 3}{2}\Big)^{k-1} x^{1-k}\frac{3}{x} = \lim_{t \to x/3} \Big(\frac{\log x -\log t}{x- t}\Big)^{k-1} \frac{1}{t}\,.
\]
Then the second mean value theorem yields some
$x/3 <  \xi < x$ such that
\begin{align*}
  J_3
    = 3\Big(\frac{3 \log 3}{2}\Big)^{k-1} x^{-k} \int_{x/3}^\xi S_t^{e^\lambda,0}(D)    (x-t)^{k-1}    dt = \ldots \int_{0}^\xi - \int_{0}^\frac{x}{3} \ldots
\end{align*}\
Now by  Lemma~\ref{todo}  there is some constant $c = c(k)>0$ such that
\begin{equation*}
\Big|  \int_{0}^{\xi} S^{e^\lambda,0}_{t}(D)  (x-t)^{k-1}  dt \Big|  \leq c  \sup_{0<t<x}|S^{e^\lambda,k}_{t}(D)|\,,
\end{equation*}
and hence by \eqref{firthB} we finally get
$
J_3 = o(1) = o(\log^k(x))\,.
$
As in the first case this proves the claim\,.
\end{proof}

\subsection{Proof of   Theorem~\ref{niceproof?}}
Let us finally turn to the proof of  Theorem~\ref{niceproof?}. Our task is to prove that,
given a $\lambda$-Dirichlet series $D=\sum a_{n}e^{-\lambda_{n}s}\in \mathcal{D}(\lambda)$  and $k > 0$,
  \[
  \sigma_c^{\lambda,k} (D) \leq \sigma_c^{e^\lambda,k} (D)\,.
  \]
It suffices to check  that  $D$  is  $(e^\lambda,k)$-summable on $[\re> \re~s_{0}]$, provided
$D$ is  $(\lambda,k)$-summable at some $s_{0}\in [\re\ge 0]$. As a by-product our argments again show that then the  limit function $f:[\re > 0] \to \mathbb{C}$
is given by
\begin{equation}\label{final2}
\Gamma(1+k) \frac{f(s)}{s^{1+k}}=\int_{0}^{\infty}e^{-s t}  S^{\lambda,k}_{t}(D)(0) dt\,,
\end{equation}
as it should be according to Theorem~\ref{basic} and Theorem~\ref{corona2}.
In fact the proof we give is a   modification of the proof given for Theorem~\ref{corona2}, and so we start modifying Lemma~\ref{2}.

\smallskip

\begin{Lemm} \label{2A}
Let $D\in \mathcal{D}(\lambda)$,  $k>0$, and $m\in \N_{0}$ with $m< k\le  m+1$.
 Then there is a constant $C=C(k)$ such that for all
$w=\sigma+i\tau \in [\re>0]$,  all $s\in \mathbb{C}$, all $0\leq \varepsilon < \frac{1}{m+2} \min\{k, \sigma\}$, and all $x>0$
\begin{align*}
&
\Big|\Gamma(m+2)^{-1}\int_{0}^{x}S_{t}^{\lambda,m+1}(D)(s) w^{m+2} (e^x-e^t)^k e^{-wt} dt- U_{x}^{\lambda,k}(D)(s+w)\Big|
,\\&
\,\,\,\,\,\,\,\,\,
\leq
e^{-\sigma x} |U_{x}^{\lambda,k}(D)(s)|
\\&
\,\,\,\,\,\,\,\,\,\,\,\,\,\,\,
+
C(k) \bigg[\sum_{j=1}^{m+1} |w^j| \sum_{\ell=1}^{m+2-j} e^{(k-\varepsilon \ell)x} x^{m-\ell+1} \int_{0}^{x} |S_{y}^{\lambda,k}(D)(s)|   e^{-(\sigma-\varepsilon \ell  ) y}  dy
\\&
\,\,\,\,\,\,\,\,\,\,\,\,\,\,\,
+
|w|\sum_{\ell=1}^{m+2} e^{(k-\varepsilon \ell)x} x^{m+2-\ell}  \int_{0}^{x} |S_{y}^{\lambda,k}(D)(s)|   e^{-(\sigma-\varepsilon \ell  ) y}  dy\bigg]
\end{align*}
\end{Lemm}

\begin{proof}
Again we first deal with the case $k< m+1$, and at  the end  we consider the more simple case $k= m+1$.
By Lemma~\ref{rekto}
 for all $s,w\in \C$ and all $x >0$
\begin{align} \label{start}
\begin{split}
   U_{x}^{\lambda,k}(D)(s+w) & -
\frac{1}{\Gamma(m+2)}\int_{0}^{x} S_{t}^{\lambda,m+1}(D)(s) w^{m+2} e^{-wt} (e^x-e^t)^kdt
\\&
= -e^{-wt}  U_{x}^{\lambda,k}(D)(s)
\\&
\,\,\,\,\,\,\,
+  C(k) \int_0^x  S_{y}^{\lambda,k}(D)(s)  \int_y^x (t-y)^{m-k} (g_1(t) + g_2(t)) dt dy\,,
\end{split}
\end{align}
where
\begin{align}\label{both}
\begin{split}
   &
  g_1(t) = \sum_{j=1}^{m+1} \binom{m+2}{j} (-w)^j e^{-wt} \partial_t^{(m+2-j)}(e^x-e^t)^k
  \\&
  g_2(t) =  (e^{-wt} - e^{-wx})\partial_t^{(m+2)}(e^x-e^t)^k
      \,.
\end{split}
\end{align}
A straigt foreward caculation shows that for every $N \in \mathbb{N}$ we have
\[
\partial_t^{N}(e^x-e^t)^k = \sum_{\ell =1 }^N c_\ell(k) (e^x-e^t)^{k -\ell} e^{\ell t}\,,
\]
with constants $c_\ell(k) > 0$. Then by the mean value theorem and the equations \eqref{betafunctioncalc}, \eqref{againagain} we have
\begin{align*}
&
  \Big| \int_y^x (t-y)^{m-k} g_1(t) dt \Big|
  \\&
  =
  \Big|\int_y^x (t-y)^{m-k} \sum_{j=1}^{m+1} \binom{m+2}{j} (-w)^j e^{-wt} \sum_{\ell =1 }^{m+2-j} c_\ell(k) (e^x-e^t)^{k -\ell} e^{\ell t} dt \Big|  \\&
  \leq
    \sum_{j=1}^{m+1} \binom{m+2}{j} |w|^j  \sum_{\ell =1 }^{m+2-j} c_\ell(k)  \int_y^x (t-y)^{m-k} (x-t)^{k -\ell}e^{(k-\ell)x}e^{-\sigma t} e^{\ell t} dt
        \\&
  =
    \sum_{j=1}^{m+1} \binom{m+2}{j} |w|^j  \sum_{\ell =1 }^{m+2-j} c_\ell(k)  \int_y^x (t-y)^{m-k} (x-t)^{k -\ell}e^{(k-\ell)x}e^{-\sigma t} e^{(1-\varepsilon)\ell t}e^{\varepsilon \ell t} dt
       \\&
  \leq
    \sum_{j=1}^{m+1} \binom{m+2}{j} |w|^j  \sum_{\ell =1 }^{m+2-j} c_\ell(k) e^{(k-\varepsilon \ell)x}
    e^{-(\sigma-\varepsilon \ell) y} (x-y)^{m-\ell+1}
        c(k)
     \,,
        \end{align*}
        where
        \[
        c(k) =\int_0^{1} \beta^{m-k} (1- \beta)^{m-\ell+1} d \beta = \frac{\Gamma(m-k+1)\Gamma(m-\ell+2)}{\Gamma(2(m+1)-k -\ell)}\,.
        \]
   And similarly, using \eqref{cdu}, we have
    \begin{align*}
&
  \Big| \int_y^x (t-y)^{m-k} g_2(t) dt \Big|
  \\&
  \leq
  \int_y^x (t-y)^{m-k}  (x-t)e^{-\sigma t} |w|\sum_{\ell=1}^{m+2} (x-t)^{k-\ell} e^{(k-\ell)x} e^{\ell t} dt
  \\&
  =
     |w|    \sum_{\ell=1}^{m+2} \int_y^x (t-y)^{m-k}(x-t)^{k-\ell+1}  e^{(k-\ell)x} e^{-\sigma t} e^{\ell t}  dt
          \\&
  \leq
     |w|    \sum_{\ell=1}^{m+2} e^{(k-\varepsilon \ell)x}  \int_y^x (t-y)^{m-k}(x-t)^{k-\ell+1} e^{-(\sigma-\varepsilon)\ell) t}  dt
   \\&
  \leq
     |w|    \sum_{\ell=1}^{m+2} e^{(k-\varepsilon \ell)x} (x-y)^{m+2-\ell} e^{-(\sigma-\varepsilon \ell) y}
        c(k) \,.
          \end{align*}
          Implementing the preceding two estimates into \eqref{start}, gives exactly what we were aiming at.

          Let us finally look at the case $k= m+1$. This case follows similarly -- but it is more simple.
          We start replacing \eqref{start} by \eqref{finally}, implement the derivatives of $g_1$ and $g_2$ from \eqref{both}, and finish as above (using  again the mean value theorem,   the estimate \eqref{cdu}, and the splitting
          $e^{\ell t} = e^{(1-\varepsilon)\ell t}e^{\varepsilon \ell t}$).
\end{proof}

\smallskip

Next we need a simple modification of Lemma~\ref{lemma-limit}.

\smallskip

\begin{Lemm} \label{lemma-limit}
Let $D\in \mathcal{D}(\lambda)$,   $k,q >0$,  and $m\in \N_{0}$ with $m< k\le  m+1$.
Then for all $w \in \mathbb{C}$ and $\varepsilon >0$ such that $0 < \varepsilon < \re w $ and
\[
|S_{t}^{\lambda,q}(D)(s)|\le C(q,\varepsilon) t^q e^{\varepsilon t}\,,\,\,\,\,t >0\,,
\]
we have
\begin{equation} \label{limit}
\lim_{x\to \infty} \int_{0}^{x}S_{t}^{\lambda,q}(D)(s) \Big(1- \frac{e^t}{e^x}\Big)^{k} w^{q+1}e^{-wt} dt=\int_{0}^{\infty} S_{y}^{\lambda,q}(D)(s) w^{q+1}e^{-wy}  dy\,.
\end{equation}

\end{Lemm}

\begin{proof}
We assume without loss of generality that $s=0$. In order to prove~\eqref{limit}, we fix $w$ and $\varepsilon$, and put $\sigma = \re w -\varepsilon$. Then for all $x,t$
\begin{align*}
\Big|S_{t}^{\lambda,q}(D)(0) \Big(1- \frac{e^t}{e^x}\Big)^{k} w^{q+1}e^{-wt} \chi_{[0,x]}(t)\Big|
\leq
C(q, \varepsilon) |w|^{q+1} e^{-\sigma t}t^q \,,
\end{align*}
and hence~\eqref{limit} is an immediate consequence of the dominated convergence theorem.
 \end{proof}

\smallskip
The final argument for Theorem~\ref{niceproof?} is now very similar to that of Theorem~\ref{corona2}.
\smallskip

\begin{proof}[{\bf Proof of Theorem~\ref{niceproof?}}] Let $k>0$ and $m\in \N_{0}$ such that $m< k\le m+1$. If $D$ is $(\lambda,k)$-summable at $s_{0}=\sigma_{0}+i\tau_{0}$, then for some constant $C=C(s_{0},k)$ we by Lemma \ref{alwaysabel} have that for all $x >0$
\begin{equation} \label{lisaaaa}
|R_{x}^{\lambda,k}(D)(0)|\le C e^{\sigma_{0} x}.
\end{equation}
We fix $s = \sigma + i\tau \in [\re>\sigma_{0}]$, and deduce from Lemma \ref{lemma-limit}  and Lemma~\ref{heute} that
\begin{align*}
\lim_{x\to \infty} \frac{e^{-kx}}{\Gamma(m+2)}&\int_{0}^{x}S_{t}^{\lambda,m+1}(D)(0) s^{m+2}(e^x-e^t)^{k} e^{-st} dt\\ &=\Gamma(1+k)^{-1}s^{k+1}\int_{0}^{\infty} S_{y}^{\lambda,k}(D)(0) e^{-st}  dy.
\end{align*}
We choose some $0< \varepsilon < \frac{1}{m+2} \min\{k, \sigma-\sigma_0\}$. According to Lemma~\ref{2A} we have to show that the term

\begin{align} \label{stoma}
\begin{split}
&
  e^{-\sigma x} |T_{x}^{\lambda,k}(D)(s)|
\\&
+
C(k) \bigg[\sum_{j=1}^{m+1} |w^j| \sum_{\ell=1}^{m+2-j}\frac{e^{(k-\varepsilon \ell)x}}{e^{kx}} e^{(k-\varepsilon \ell)x} x^{m-\ell+1} \int_{0}^{x} |S_{y}^{\lambda,k}(D)(s)|   e^{-(\sigma-\varepsilon \ell  ) y}  dy
\\&
+
|w|\sum_{\ell=1}^{m+2} \frac{e^{(k-\varepsilon \ell)x}}{e^{kx}} x^{m+2-\ell}  \int_{0}^{x} |S_{y}^{\lambda,k}(D)(s)|   e^{-(\sigma-\varepsilon \ell  ) y}  dy\bigg]
\end{split}
\end{align}
vanishes as $x \to \infty$.  From \eqref{lisaaaa} we know that for all $y>0$
\begin{equation*} \label{lisA}
|S_{y}^{\lambda,k}(D)(0)|\le C y^k e^{\sigma_{0} y}\,.
\end{equation*}
Hence for each of the integrals in \eqref{stoma} we have
\[
\int_{0}^{x} |S_{y}^{\lambda,k}(D)(0)|   e^{-(\sigma-\ell\varepsilon) y}  dy
\leq
C \int_{0}^{\infty} y^k    e^{-(\sigma-\sigma_0-\ell \varepsilon) y}  dy < \infty\,,
\]
and  it remains to check that
\begin{align*}
  0 = \lim_{x \to \infty}
e^{-\sigma x} &
 |T_{x}^{\lambda,k}(D)(0)|\,.
\end{align*}
To do this, we note first that by Lemma~\ref{newnow}  for some constant $C >0$ and  all $x >0$
\[
|U_{x}^{\lambda,k}(D)(0)| \leq C  e^{kx} \sup_{y<x} |S_{x}^{\lambda,k}(D)(0)|\,,
\]
and hence by \eqref{lisaaaa}
\[
e^{-\sigma x}|T_{x}^{\lambda,k}(D)(0)| \leq C  e^{-\sigma x} \sup_{y<x} |y^kR_{x}^{\lambda,k}(D)(0)| \leq C  e^{(\sigma_0 -\sigma)x}x^k\,,
\]
implying the conclusion.
\end{proof}

\end{document}